\documentclass[a4paper,twoside,11pt,english,intlimits,openany]{article}

\usepackage{latexsym,amsfonts,amssymb,exscale,enumerate}
\usepackage{amsmath}
\usepackage{amsthm}
\usepackage{amssymb}
\usepackage{geometry}
\usepackage{fancyhdr,times,euler,euscript,color}
\usepackage{bbm}
\usepackage{hyperref}
\usepackage{etex}
\usepackage{multicol}
\usepackage{titlesec}
\usepackage{comment}
\usepackage{appendix}
\usepackage{etoolbox}
\makeatletter
\patchcmd{\ttlh@hang}{\parindent\z@}{\parindent\z@\leavevmode}{}{}
\patchcmd{\ttlh@hang}{\noindent}{}{}{}
\makeatother

\usepackage{graphicx}

\usepackage[all]{xy}
\SelectTips{cm}{}

\usepackage{pstricks}

\psset{linewidth=0.3pt,dimen=middle}
\psset{xunit=.70cm,yunit=0.70cm}
\psset{arrowsize=1pt 5,arrowlength=0.6,arrowinset=0.7}
\usepackage{tikz}

\titleformat{\chapter}[frame]
{\normalsize}%
{\filright\sffamily\Large%
\enspace Chapitre \thechapter\enspace}%
{5pt}
{\rule{0pt}{30pt}\sffamily\Huge\bfseries\filcenter}

\titlespacing*{\chapter}{-1cm}{-2cm}{1cm}

\titleformat{\section}[block]{\scshape\filcenter\Large}{\thesection.}{.5em}{}
\titleformat{\subsection}[block]{\bfseries\filcenter\large}{\thesubsection.}{.5em}{}
\titleformat{\subsubsection}[runin]{\bfseries}{\thesubsubsection.}{.5em}{}[.]

\swapnumbers
  	\newtheorem{theorem}[subsubsection]{Theorem}
  	\newtheorem{corollary}[subsubsection]{Corollary}
  	\newtheorem{lemma}[subsubsection]{Lemma}
  	\newtheorem{proposition}[subsubsection]{Proposition}
\theoremstyle{definition}
	\newtheorem{definition}[subsubsection]{Definition}
	
	\newtheorem{remark}[subsubsection]{Remark}

\input{macros.tex}

\newcommand{\somme}[2]{\sum\limits_{#1}^{#2}}

\newcommand{\er}[2]{\,_{#1}{#2}}

\newcommand{\ere}[2]{\,_{#1}{#2}_{#1}}
\newcommand{\ER}{\er{E}{R}}

\newcommand{\ERE}{\ere{E}{R}}
\newcommand{\cl}[1]{\overline{#1}}

\newcommand{\Ceil}[1]{#1_{i,\lambda}}

\def\dar[#1,#2,#3]{\ar@<0.8ex>[#1] ^{#3} \ar@<-0.8ex>[#1] _{#2} }

\newcommand{\tck}[1]{#1^{\top}}
\newcommand{\KLR}{\text{KLR}}
\newcommand{\KLRb}{\mathcal{KLR}}

\newcommand{\context}[1]{m_1 \star_1 (m_2 \star_0 #1 \star_0 m_3) \star_1 m_4}

\newcommand{\Ro}{R^{\ast (1)}}
\newcommand{\So}{S^{\ast (1)}}

\newcommand{\tilda}[1]{\widetilde{\widehat{#1}}}

\newcommand{\catego}[1]{\mathbf{#1}}
\newcommand{\catklr}{\mathcal{C}^{\text{KLR}}}

\begin{document}
\thispagestyle{empty}

\begin{center}

\begin{Large}\begin{uppercase}
{Rewriting modulo isotopies in Khovanov-Lauda-Rouquier's categorification of quantum groups.}
\end{uppercase}\end{Large}

\vskip+8pt

\bigskip\hrule height 1.5pt \bigskip

\vskip+10pt

\begin{large}\begin{uppercase}
{Benjamin Dupont}
\end{uppercase}\end{large}

\vskip+40pt

\begin{small}\begin{minipage}{14cm}
\noindent\textbf{Abstract --}
We study a presentation of Khovanov - Lauda - Rouquier's candidate $2$-categorification of a quantum group using algebraic rewriting methods. We use a computational approach based on rewriting modulo the isotopy axioms of its pivotal structure to compute a family of linear bases for all the vector spaces of $2$-cells in this $2$-category. We show that these bases correspond to Khovanov and Lauda's conjectured generating sets, proving the non-degeneracy of their diagrammatic calculus. 
This implies that this $2$-category is a categorification of Lusztig's idempotent and integral quantum group $\bf{U}_{q}(\mathfrak{g})$ associated to a symmetrizable simply-laced Kac-Moody algebra $\mathfrak{g}$.

\bigskip


\smallskip\noindent\textbf{M.S.C. 2010 -- 68Q42, 18D05, 17B37, 16T20.} 
\end{minipage}\end{small}
\end{center}

\vspace{0.8cm}

\begin{center}
\begin{small}\begin{minipage}{12cm}
\renewcommand{\contentsname}{}
\setcounter{tocdepth}{2}
\tableofcontents
\end{minipage}
\end{small}
\end{center}

\clearpage

\section*{Introduction}

In \cite{KL3}, Khovanov and Lauda introduced a $2$-category $\Ug$ which is a candidate to be a categorification of Lusztig's idempotent and integral version of the quantum group $\mathbf{U}_q (\mathfrak{g})$ associated to a symmetrizable Kac-Moody algebra $\mathfrak{g}$. They proved that $\Ug$ is a categorification of $\mathbf{U}_q (\mathfrak{g})$ if and only if the diagrammatic calculus they introduce in \cite{KL3} is non-degenerated, that is the spaces of $2$-cells in $\Ug$ have an explicit linear basis, proving that the relations of $\Ug$ do not vanish all the string diagrams to $0$. In \cite{ROU08}, Rouquier independently introduced a $2$-Kac-Moody algebra, which turns out to be isomorphic to $\Ug$, see \cite{BRU15}. The main objective of this paper is to prove that the sets conjectured by Khovanov and Lauda to be linear bases of the spaces of $2$-cells in $\Ug$ are linear bases using rewriting methods. The $2$-category $\Ug$ is a $\mathbb{K}$-linear~$(2,2)$-category, as recalled in Section \ref{SSS:Linear22Cat}, that is all the spaces of $2$-cells in $\Ug$ are $\mathbb{K}$-vector spaces for a ground field $\mathbb{K}$. Since $\Ug$ admits a pivotal structure, we use the context of rewriting modulo the isotopy axioms of a pivotal linear~$(2,2)$-category introduced in \cite{DUP19} to compute these linear bases.

\subsubsection*{Higher-dimensional representation theory and categorification of quantum groups} 
In representation theory, one study actions of an algebra of a vector space. Higher dimensional representation theory aims at replacing
these vector spaces by categories, and linear maps by functors. The objective in that
process is to construct a categorification of the given algebra, that is an higher
dimensional abelian, additive, or triangulated category whose corresponding Grothendieck group is isomorphic to the algebra. In this paper,
we are interested in Khovanov and Lauda's categorification of a quantum group associated to a symmetrizable Kac-Moody algebra $\mathfrak{g}$, \cite{KL1,KL2,KL3}. Given any root datum corresponding to a symmetrizable Kac-Moody
algebra $\mathfrak{g}$, they defined in \cite{KL3} a candidate $2$-category to be a categorification of Lusztig's idempotent and integral version of the quantum
group $\mathbf{U}_q (\mathfrak{g})$ associated with this root datum. 
The $2$-category $\Ug$ is defined from a presentation by generators and relations, that is it has a set of $0$-cells, a set of generating $1$-cells and a set of generating $2$-cells, and the compositions that one can make with these generating $2$-cells are subject to some relations. Khovanov and Lauda
established \cite[Theorem 1.1, Theorem 1.2]{KL3} that $\Ug$ is a categorification of $\mathbf{U}_q (\mathfrak{g})$ if the diagrammatic calculus they introduce in \cite{KL3} is non degenerated, which corresponds to the fact that each vector space of $2$-cells in $\Ug$ admits an explicit linear basis described in
\cite[Section 3.2.3]{KL3}. Khovanov and Lauda proved in \cite{KL3} the non-degeneracy of their calculus for symmetrizable Kac-Moody
algebras of type $A$, namely for $\mathfrak{sl}_n$. It remained unknown in general until the works of Webster \cite{WEB13}, who proved this non-degeneracy for any root datum of finite type and for any field $\K$ using slightly different methods than in this paper. It was also proved by Kang and Kashiwara in \cite{KK12} from a $2$-representation of Rouquier's $2$-Kac Moody algebra, proving that the diagrammatic calculus  in $\Ug$ can not be degenerated. In this paper, we give a new proof of these results using a rewriting theoretical approach. We restrict our study to the case of simply-laced symmetrizable Kac-Moody algebras, that is Kac-Moody algebras whose Dynkin graph does not admit loops nor multiple edges. Indeed, as explained in \ref{R:SimplyLaced}, the relations appearing in the presentation of $\Ug$ are more simple in this setting, and thus it simplifies the computations when studying it using rewriting methods. However, we expect that the methods provided in this paper extend to the non simply-laced setting.

In \cite{ROU08}, Rouquier defined a Kac-Moody $2$-category $\Ag$, which has less generating $2$-cells than
$\Ug$, so that rewriting in this $2$-category is more adapted. Brundan proved in \cite{BRU15} that these two $2$-categories
$\Ug$ and $\Ag$ are isomorphic. In this paper, we will thus choose to work in the $2$-category
$\Ag$ and its diagrammatic presentation given by Brundan, and translate the computations in $\Ug$ through this isomorphism.

\subsubsection*{Rewriting and linear polygraphs}
Polygraphs are algebraic objects used to generate higher-dimensional globular strict categories, introduced independently by Burroni \cite{BUR93} and Street \cite{STR86,STR87}. They have been widely used in rewriting theory \cite{MET03,LAF07,MET08,GM09,MIM10,GM12,LUC17,HAD17,GM18} to compute in various algebraic structures and describe the properties of the computations using a homotopical approach. The properties of presentations of $2$-categories by $3$-polygraphs have been studied in a non-linear setting in \cite{GM09}. These methods have been extended to the structure of $\mathbb{K}$-linear~$(2,2)$-categories, that is categories enriched in $\mathbb{K}$-linear categories for a given field $\mathbb{K}$ in \cite{AL16}, following the notion of linear polygraphs introduced in \cite{GHM17} to rewrite in associative algebras. We study presentations of linear~$(2,2)$-categories by rewriting systems called linear~$(3,2)$-polygraphs. There are two fundamental rewriting properties that we study in order to compute linear bases: termination, establishing that an element can not be reduced infinitely many times, and confluence, stating that two paths of reductions starting from the same element must reach the same result. It was proven in \cite{AL16} that, given a presentation of a linear~$(2,2)$-category $\C$ by a convergent, that is terminating and confluent, linear~$(3,2)$-polygraph $P$, one can obtain a hom-basis of $\C$, that is a family of sets $B_{p,q}$ indexed by pairs of $1$-cells $p$ and $q$ in $\C$ such that $B_{p,q}$ is a linear basis of the vector space $\C_2(p,q)$ of $2$-cells with $1$-source $p$ and $1$-target $q$. Such a hom-basis is constructed by considering monomials in $\mathcal{C}$ in normal form with respect to $P$. This was then extended in a non-terminating context, as Alleaume pointed out in \cite{AL16} that many presentations of the linear~$(2,2)$-categories arising in representation theory cannot be oriented in a terminating way. However, the presentations are in general quasi-terminating, that is each non-terminating rewriting sequence is derived from a rewriting cycle. Alleaume extended the basis result in \cite{ALPhD} to quasi-terminating linear~$(3,2)$-polygraphs by proving that a hom-basis is given by monomials in quasi-normal form, that is monomials on which we can only apply some rules that give rise to a cycle, and no other relation. 

In this paper, we will study two linear~$(2,2)$-categories with different structures using a rewriting theoretical approach. The first one is a linear~$(2,2)$-category $\catklr$ defined in such a way that the Khovanov-Lauda-Rouquier (KLR) algebras \cite{KL1,ROU08} can be recovered as some spaces of $2$-cells of $\catklr$, as explained in Section \ref{SSS:CategoricalStructure}. This linear~$(2,2)$-category does not admit a pivotal structure, and we prove in Section \ref{SS:LinearPolKLR} that it can be presented by a convergent linear~$(3,2)$-polygraph. The second linear~$(2,2)$-category is the Khovanov-Lauda-Rouquier's $2$-categorification of a quantum group, which admits a pivotal structure, 
that is each $1$-cell is equipped with a biadjoint, yielding unit and counit $2$-cells, and the remaining $2$-cells satisfy some cyclicity relations in the sense of \cite{CKS00}. In terms of string diagrams, the unit and counit $2$-cells will be depicted by caps and cups diagrams, and the axioms of pivotality implies that all the string diagrams will be depicted up to isotopy. The isotopy relations make the confluence difficult to prove, because of a great number of overlappings between a defining relation and an isotopy relation. As a consequence, the methods provided in \cite{AL16} are difficult to apply in this situation. This issue is solved by using rewriting modulo these relations, as explained in \cite{DUP19}, where these rules are not considered as rewriting rules anymore, but as axioms that we freely use in rewriting paths.

\subsubsection*{Khovanov-Lauda-Rouquier algebras}
In the construction of a categorification of a quantum group, the family of KLR algebras, or quiver Hecke algebras, emerged.
These algebras were discovered independently by Rouquier \cite{ROU08}, Khovanov and Lauda
\cite{KL1} since the category of finitely-generated projective modules over these algebras categorify the
negative part of the associated quantum group, see \cite{KL2}. Furthermore, these algebras act
on some endomorphism spaces of $2$-cells of $\Ug$. 
In Section \ref{SS:KLRAlgebras}, we define following \cite{ROU08} the family of KLR algebras that we specialize to Khovanov and Lauda's diagrammatic presentation. We also restrict to KLR algebras associated to simply-laced Cartan datum, and thus simply-laced symmetrizable Kac-Moody algebras to simplify the computations.
We define in Section \ref{SSS:CategoricalStructure} a linear~$(2,2)$-category $\catklr$ encoding the whole family of KLR algebras $(R(\mathcal{V}))_{\mathcal{V} \in \N [I]}$ in the following sense: for any $\mathcal{V}$ in $\N [I]$, the algebra $R(\mathcal{V})$ can be recovered as some endomorphism spaces of $2$-cells in $\catklr$. In Section \ref{S:KLRAlgebras}, we define a linear~$(3,2)$-polygraph $\text{KLR}$ presenting $\catklr$, and prove the following result:

\begin{quote}
\noindent{\bf Theorem \ref{T:ConvergentPresKLR}}.
\emph{
The linear~$(3,2)$-polygraph $\text{KLR}$ is a convergent presentation of the linear~$(2,2)$-category $\catklr$.}
\end{quote}

Computing monomials in normal form with respect to $\text{KLR}$, we obtain linear bases for each algebra $R(\mathcal{V})$. In particular, we recover the linear bases described by Khovanov and Lauda in  \cite[Theorem 2.5]{KL1}. In \cite[Theorem 3.7]{ROU08}, Rouquier described that these algebras admit a \emph{Poincar\'e - Birkhoff-Witt} (PBW) property, which he proves equivalent to the fact that a given set is a basis of this algebra. In Section \ref{SS:PBW}, we prove that this set correspond to the monomials in normal form for $\text{KLR}$, so that the KLR algebras admit PBW bases.

\subsubsection*{Rewriting modulo isotopy in pivotal categories}
The linear~$(2,2)$-category $\Ug$ is a pivotal linear~$(2,2)$-category. In such a $2$-category, every $1$-cell $x$ admits a dual $1$-cell $\hat{x}$ which is both its right and left adjoint. In this case, for any $i$ indexing the Dynkin graph of the Kac-Moody algebra $\mathfrak{g}$, the lift of the generator $E_i$ multiplied by an idempotent $1_\lambda$ in Lusztig's idempotent completion of the quantum group into the categorification is the dual of the lift of $F_i$ multiplied by the idempotent $ 1_{\lambda + \alpha_i}$. Following \cite{DUP19}, we do not orient the isotopy relations provided by this structure as ordinary rewriting rules, but we rewrite modulo these latter. The theory of rewriting modulo extends the usual theory of rewriting by considering a set $R$ of oriented rules, and a set $E$ of non-oriented equations that we take into account when rewriting. It was explained in \cite{DMpp18,DUP19} that rewriting modulo allows more flexbility in computations to reach confluence. It also reduces the number of overlappings between relations that need to be considered in the analysis of confluence.
In \cite{DMpp18}, a categorical and polygraphic model was introduced to rewrite modulo in various algebraic structures, following Huet \cite{Huet80} and Jouannaud and Kirchner's \cite{JouannaudKirchner84} approaches . An abstract local confluence criteria and a critical branching lemma were proved, under some termination assumption for the rewriting system $\ERE$, consisting in rewriting with $R$ on equivalences classes for the congruence generated by $E$.

In \cite{DUP19}, these results were extended in the linear setting with the introduction of linear~$(3,2)$-polygraphs modulo. Moreover, it is proved in \cite{CDM19} that the termination assumption for $\ERE$ needed to prove confluence modulo from confluence modulo of critical branchings can be weakened to an assumption of quasi-termination and exponentiation freedom, depicting the fact that a rewriting path cannot grow exponentially using the same rewriting rule at each step.
However, proving confluence of a linear~$(3,2)$-polygraph modulo $(R,E,S)$ modulo $E$ when $S$, and thus $\ERE$, is not terminating is more difficult. This can be done using the notion of decreasingness modulo, introduced in \cite{DUP19} following Van Oostrom's abstract decreasingness property \cite{VOO94}. Indeed, recall from \cite[Theorem 2.3.8]{DUP19} that if $(R,E,S)$ is decreasing modulo $E$, it is confluent modulo $E$. Moreover, it is proved in \cite{CDM19} that when $S$ is quasi-terminating, decreasingness can be obtained from decreasingness of critical branchings modulo with respect to the quasi-normal form labelling, counting the distance between a $2$-cell and its fixed quasi-normal form.

In \cite{DUP19}, a method to compute a hom-basis of a linear~$(2,2)$-category using rewriting modulo was given. Namely, when considering a  normalizing linear~$(3,2)$-polygrah modulo, one can consider two different kinds of normal forms: normal forms with respect to $S$ or normal forms with respect to the convergent polygraph $E$ for which we rewrite modulo. It is proved in \cite[Theorem 2.5.4]{DUP19} that if $P$ is a linear~$(3,2)$-polygraph presenting a linear~$(2,2)$-category $\C$, splitted into a linear~$(3,2)$-polygraph $R$ of rewritings and a convergent linear~$(3,2)$-polygraphs of axioms modulo, considering all the monomials in normal form with respect to $S$, then taking their normal form with respect to $E$ and considering all the monomials in these latter gives a hom-basis of $\C$. Moreover, this result is adapted when $S$ is quasi-terminating \cite[Theorem 2.5.6]{DUP19}, by fixing a set of quasi-normal forms for $S$, taking the monomials in quasi-normal form for $S$ and considering the monomials in the support of their $E$-normal form.

\subsubsection*{Categorification of quantum groups}
In the last part of this paper, in Section \ref{S:KLR2Cat}, we define the linear~$(2,2)$-category $\Ag$ and we explicit Brundan's isomorphism with the definition of the additional generators and relations provided by these. Instead of using the presentation suggested by Brundan's paper, we prove additional relations in order to obtain symmetries in our set of relations. We then define a linear~$(3,2)$-polygraph $\KLRb$ presenting the linear~$(2,2)$-category $\Ag$, that we split into two parts: a convergent linear~$(3,2)$-polygraph $E$ containing all isotopy $3$-cells and a linear~$(3,2)$-polygraph $R$ containing the remaining $3$-cells. The pair of linear~$(3,2)$-polygraphs $(R,E)$ is called a \emph{convergent splitting} of $\KLRb$, and we prove:
\begin{quote}
\noindent{\bf Theorem \ref{T:Quasi-ConvModKLR}}.
\emph{Let $(R,E)$ be the convergent splitting of $\KLRb$ defined in \ref{SSS:SplittingKLR}. Then $\ER$ is quasi-terminating and $\ER$ is confluent modulo $E$.
}
\end{quote}

As a consequence, fixing a set of monomials in quasi-normal forms with $1$-source $E_{\mathbf{i}} \one$ and $1$-target $E_{\mathbf{j}} \one$, and taking their normal form with respect to $E$ gives a linear basis of $\Ug(E_{\mathbf{i}} \one, E_{\mathbf{j}} \one)$ for any $1$-cells $E_{\mathbf{i}} \one$ and $E_{\mathbf{j}} \one$ of $\Ug$. We prove that such a choice of quasi-normal form correspond to a choice of Khovanov-Lauda's generating set $\mathcal{B}_{\mathbf{i}, \mathbf{j}, \lambda}$, and thus that the following result holds:
\begin{quote}
\noindent{\bf Theorem \ref{T:BasisKLCategory}}.
\emph{The set $\mathcal{B}_{i,j,\lambda}$ defined in \ref{SSS:MonQuasiNF} is a linear basis of $\Ug (E_{\mathbf{i}} \one, E_{\mathbf{j}} \one)$.
}
\end{quote}
This proves the non-degeneracy of Khovanov and Lauda's diagrammatic calculus in that case, and thus that for a simply-laced symmetrizable Kac-Moody algebra $\mathfrak{g}$, the linear~$(2,2)$-category $\Ug$ is a categorification of the Lusztig's quantum group $\mathbf{U}_q(\mathfrak{g})$ associated to $\mathfrak{g}$.

\subsubsection*{Organization of the paper}
In the first Section of this paper, we recall some properties of linear~$(2,2)$-categories and linear~$(3,2)$-polygraphs from \cite{AL16} and linear~$(3,2)$-polygraphs modulo from \cite{DUP19}. 
In Section \ref{S:KLRAlgebras}, we define following \cite{ROU08} the family of KLR algebras and specialize it to Khovanov and Lauda's diagrammatic definition in simply-laced type \cite{KL1}. In Section \ref{SS:KacMoodySetup}, we introduce all the needed material about Kac-Moody algebras. In Subsection  \ref{SS:LinearPolKLR}, we define a linear~$(3,2)$-polygraph $\KLR$ presenting the linear~$(2,2)$-category $\catklr$. As a consequence, we prove in \ref{SS:PBW} that the KLR algebras admit PBW bases.

In the last section of this paper, Section \ref{S:KLR2Cat}, we extend the study of the KLR algebras to the $2$-category $\Ug$. In Subsection \ref{SS:KLRCategories}, we recall following \cite{BRU15} a diagrammatic presentation for this $2$-category. In Subsection \ref{SS:LinPolKLRb}, we define a linear~$(3,2)$-polygraph presenting $\Ug$ to which we associate a convergent splitting $(R,E)$. We prove that $\ER$ is quasi-terminating in Section \ref{SS:QuasiTermKLR}. In Section\ref{SS:ConfluenceModuloKLR}, we prove that $\ER$ is confluent modulo $E$ using decreasingness of its critical branchings modulo with respect to a quasi-normal form labelling. In Section \ref{SS:CategorificationQuantumGroups}, we compare the linear bases obtained using rewriting modulo to Khovanov and Lauda's generating sets for each space of $2$-cells. We prove that one can make a choice of quasi-normal forms so that these two sets are the same, so that $\Ug$ really is a categorification of $\U_q(\mathfrak{g})$.

\section{Preliminaries}
\label{S:Preliminaries}
If $\mathcal{C}$ is an $n$-category, we denote by $\mathcal{C}_n$ the set of $n$-cells in $\mathcal{C}$. For any $0 \leq k < n$ and any $k$-cells $p$ and $q$ in $\mathcal{C}$, we denote by $\mathcal{C}_{k+1}(p,q)$ the set of $(k+1)$-cells in $\C$ with $k$-souce $p$ and $k$-target $q$. If $p$ is a $k$-cell of $\mathcal{C}$, we denote respectively by $s_i(p)$ and $t_i(p)$ the $i$-source and $i$-target of $p$ for $0 \leq i \leq k-1$. These assignments define source and target maps, satisfying the globular relations
\[ s_i \circ s_{i+1} = s_i \circ t_{i+1} \quad \text{and} \quad t_i \circ s_{i+1} = t_i \circ t_{i+1} \] 
for any $0 \leq i \leq n-2$. Two $k$-cells $p$ and $q$ are \emph{$i$-composable} when $t_i(p) = s_i(q)$. In that case, their $i$-composition is denoted by $p \star_i q$. The compositions of $\mathcal{C}$ satisfy the \emph{exchange relations}:
\begin{equation}
\label{E:ExchangeRel} 
(p_1 \star_i q_1) \star_j (p_2 \star_i q_2) = (p_1 \star_j p_2) \star_i (q_1 \star_j q_2) 
\end{equation}
for any $i < j$ and for all cells $p_1$,$p_2$,$q_1$,$q_2$ such that both sides are defined. If $p$ is a $k$-cell of $\mathcal{C}$, we denote by $1_{p}$ its identity $(k+1)$-cell.
A $k$-cell $p$ of $\mathcal{C}$ is \emph{invertible with respect to $\star_i$-composition} ($i$-invertible for short) when there exists a (necessarily unique) $k$-cell $q$ in $\mathcal{C}$ with $i$-source $t_i(p)$ and $i$-target $s_i(p)$ such that
\[ p \star_i q = 1_{s_i(p)} \quad \text{and} \quad q \star_i p = 1_{t_i(p)} \]

\noindent Throughout this paper, $2$-cells in $2$-categories are represented using the classical representation by string diagrams, see \cite{LAU12,SAV18} for surveys on the correspondance between $2$-cells and diagrams. The $\star_0$ composition of $2$-cells is depicted by placing two diagrams next to each other, the $\star_1$-composition is vertical concatenation of diagrams.

\subsection{Rewriting modulo in linear~$(2,2)$-categories}
In this subsection, we recall from \cite{DUP19} the notions on linear~$(2,2)$-categories and their presentations by linear~$(3,2)$-polygraphs modulo. Throughout this paper, we fix an arbitrary field $\mathbb{K}$.

\subsubsection{Linear~$(2,2)$-categories}
\label{SSS:Linear22Cat}
A \emph{linear~$(2,2)$-category} (over $\mathbb{K}$) is a $2$-category $\mathcal{C}$ such that for any $1$-cells $p$ and $q$ in $\mathcal{C}$, the set $\mathcal{C}_1(p,q)$ of $2$-cells with $1$-source $p$ and $1$-target $q$ is a $\mathbb{K}$-vector space. For any $p,q,r$ in $\C_1$, the map $\star_1: \C_2(p,q) \otimes \C_2(q,r) \to \C_2(p,r)$ is $\K$-linear. If a linear~$(2,2)$-category $\mathcal{C}$ is presented by generators and relations, a $2$-cell $\phi$ obtained using $\star_0$ and $\star_1$ compositions of generating $2$-cells is called a \emph{monomial} in $\C$. Any $2$-cell $\phi$ in $\C$ can be uniquely decomposed into a sum of monomials $\phi= \sum \phi_i$, which we call the \emph{monomial decomposition} of $\phi$. We set the \emph{support} of $\phi$, denoted by $\text{Supp}(\phi)$, to be the set $\lbrace \phi_i \rbrace$ of $2$-cells that appear in this monomial decomposition. Recall from \cite{DUP19} that a \emph{hom-basis} of $\mathcal{C}$ is a family $(\mathcal{B}_{p,q})$ of sets indexed by pairs $(p,q)$ of $1$-cells of $\C$ such that for any $1$-cells $p$ and $q$, $\mathcal{B}_{p,q}$ is a linear basis of $\C_2(p,q)$.

\subsubsection{Linear~$(3,2)$-polygraphs modulo}
Recall from \cite{AL16}, that a (left-monomial) linear~$(3,2)$-polygraph is a triple $P= (P_0,P_1,P_2,P_3)$ where $(P_0,P_1)$ is an oriented graph, $P_2$ is a cellular extension on the free $1$-category denoted by $P_1^\ast$ on $(P_0,P_1)$, that is a set equipped with $1$-source and $1$-target maps $s_1$,$t_1: P_2 \fl P_1^\ast$, and $P_3$ is a cellular category on the free linear~$(2,2)$-category $P_2^\ell$ generated by $(P_0,P_1,P_2)$ such that for any $\alpha$ in $P_3$, $s_2(\alpha)$ is a monomial in $P_2^\ell$. For a cellular extension $\Gamma$ of $P_1^\ast$, we will denote by $||f||_\Gamma$ the number of occurences of $2$-cells of $\Gamma$ in the $2$-cell $f$ in $P_2^\ast$.

Let $P$ be a linear~$(3,2)$-polygraph, we denote by $P_{\leq k}$ the underlying $k$-polygraph of $P$, for $k = 1,2$. We denote by $P_3^\ell$ the free linear~$(3,2)$-category generated by $P$, as defined in \cite[Section 3.1]{ALPhD}. Recall from \cite[Proposition 1.2.3]{GHM17} that every $3$-cell $\alpha$ in $P_3^\ell$ is $2$-invertible, its inverse being given by $1_{s_2(\alpha)} - \alpha + 1_{t_2(\alpha)}$. The \textit{congruence} generated by $P$ is the equivalence relation $\equiv$ on $\Sl$ defined by 
\[ \text{$ u \equiv v$ if there is a $3$-cell $\alpha$ in $P_3^\ell$ such that $s_2(\alpha) = u$ and $t_2(\alpha)=v$}. \]
 We say that a linear~$(2,2)$-category $\mathcal{C}$ is presented by $P$ 
if $\mathcal{C}$ is isomorphic to the quotient category $\Sl \slash \equiv$. A \textit{rewriting step} of a linear~$(3,2)$-polygraph $P$ is a $3$-cell of the following form:
\begin{equation}
\label{E:RewritingStep}
 C[\alpha]: \; \lambda \context{s_2(\alpha)} + h \fl \lambda \context{t_2(\alpha)} + h, 
\end{equation}
where $\alpha$ is a generating $3$-cell in $P_3$, the $m_i$ are monomials in $P_2^\ell$ and $h$ is a $2$-cell in $P_2^\ell$, such that the monomial $m_1 \star_1 ( m_2 \star_0 s_2(\alpha) \star_0 m_3)\star_1 m_4$ does not appear in the monomial decomposition of $h$. The element $C = \lambda \context{ \square} + h$ is a \emph{context} of the linear~$(2,2)$-category $P_2^\ell$, as defined in \cite{DUP19}, and is called the context of application of the rule $\alpha$ in (\ref{E:RewritingStep}). Such a rewriting step will thus be denoted by $C[\alpha] : C[s_2(\alpha)] \fl C[t_2(\alpha)]$ in the sequel. 
A \textit{rewriting sequence} of $P$ is a finite or infinite sequence of rewriting steps of $P$. We say that a $2$-cell is a \textit{normal form} if it can not be reduced by any rewriting step.

A $3$-cell $\alpha$ of $P_3^\ell$ is called \emph{positive} if it is a $\star_2$-composition $\alpha= \alpha_1 \star_2 \dots \star_2 \alpha_n$ of rewriting steps of $P$. The \emph{length} of a positive $3$-cell $\alpha$ in $P_3^\ell$ is the number of rewriting steps of $P$ needed to write $\alpha$ as a $\star_2$-composition of these rewriting steps.
The linear~$(3,2)$-polygraph $P$ equipped with this notion of rewriting step defines an abstract rewriting system.

\subsubsection{Termination and confluence}
A \emph{branching} (resp. \emph{local branching}) of a linear~$(3,2)$-polygraph $P$ is a pair of rewriting sequences (resp. rewriting steps) of $P$ which have the same $2$-cell as $2$-source.
Such a branching is \emph{confluent} if it can be completed by rewriting sequences $f'$ and $g'$ of $P$ as follows:
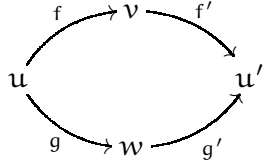
\[ \xymatrix @C=2.6em@R=1.2em{
& v
	\ar @/^1.5ex/ [dr] ^-{f'}
\\
u 
	\ar @/^1.5ex/ [ur] ^-{f}
	\ar @/_1.5ex/ [dr] _-{g}
&& u'
\\
& w
	\ar @/_1.5ex/ [ur] _-{g'}
}	\]

\noindent A linear~$(3,2)$-polygraph $P$ is said:
\begin{enumerate}[{\bf i)}]
\item \emph{terminating} if there is no infinite rewriting sequences in $P$.
\item \emph{quasi-terminating} if for each sequence $(u_n)_{n \in \N}$ of $2$-cells such that for each $n$ in $\N$ there is a rewriting step from $u_n$ to $u_{n+1}$, the sequence $(u_{n})_{n \in \N}$ contains an infinite number of occurences of the same $2$-cell.
\item \emph{confluent} if all the branchings of $P$ are confluent.
\item \emph{convergent} if it is both terminating and confluent.
\item \emph{exponentiation free} is for any $2$-cell $u$, there does not exist a $3$-cell $\alpha$ in $P_3^\ell$ such that $$ 
\text{$u \overset{\alpha}{\fl} \lambda u + h$ with $\lambda \in \K \backslash \{ 0 \}$ and $h \ne 0$}. $$ 
\end{enumerate}

 A \emph{normal form} of a linear~$(3,2)$-polygraph $P$ is a $2$-cell $u$ that cannot be rewritten by any rewriting step of $P$. When $P$ is terminating, any $2$-cell admits at least one normal form, and exactly one when it is also confluent. A \emph{quasi-normal form} is a $2$-cell $u$ such that for any rewriting step from $u$ to another $2$-cell $v$, there exist a rewriting sequence from $v$ to $u$.

Newman's lemma states that if a linear~$(3,2)$-polygraph $P$ is terminating, then the confluence of $P$ is equivalent to its local confluence. Alleaume proved in \cite{AL16} that if $P$ is a left-monomial and terminating linear~$(3,2)$-polygraph, then it is locally confluent if and only if its critical branchings are confluent. He also established that a hom-basis of a linear~$(2,2)$-category can be obtained from a convergent presentation of $\mathcal{C}$. More precisely, if $\mathcal{C}$ is a linear~$(2,2)$-category presented by a convergent linear~$(3,2)$-polygraph $P$, then the set of all monomials in normal form with respect to $P$ is a hom-basis of $\mathcal{C}$, where $\mathcal{B}_{p,q}$ is the set of all monomials in normal form with $1$-source $p$ and $1$-target $q$.

\subsubsection{Linear~$(3,2)$-polygraphs modulo}
A linear~$(3,2)$-polygraph modulo is the data of a triple $(R,E,S)$ where 
\begin{enumerate}[{\bf i)}]
\item $R$ and $E$ are linear~$(3,2)$-polygraphs having the same underlying $1$-polygraph, and such that $E_2 \subseteq R_2$,
\item $S$ is a cellular extension of the free linear~$2$-category $R_2^\ell$ such that the following inclusions of cellular extensions $R \subseteq S \subseteq \ERE$ holds, where the cellular extension \[ \ERE \overset{\gamma^{\ERE}}{\fl} \text{Sph}(R_2^\ell) \] correspond to $2$-spheres $(u,v) \in R_{2}^\ell$ which is the boundary of a $3$-cell $f$ in $R^\ell_{2}[R_3,E_3,E_3^-]/\text{Inv}(E_3,E_3^-)$, the free linear~$(2,2)$-category generated by $(R_0,R_1,R_2)$ augmented by the cellular extensions $R$, $E$ and the formal inverses $E^-$ of $E$ modulo the corresponding inverse relations, with shape \[ f = e_1 \star_{2} f_1 \star_{2} e_2, \] where $e_1,e_2$ are $3$-cells in $E_3^\ell$ and $f_1$ a rewriting step of $R$.
\end{enumerate}

We refer to \cite{DMpp18} for a detailed definition of the definition of higher-dimensional polygraphs modulo. Given a linear~$(3,2)$-polygrah modulo $(R,E,S)$, the quadruple $(R_0,R_1,R_2,S)$ is a linear~$(3,2)$-polygraph that we denote by $S$ in the sequel.

\subsection{Confluence modulo and decreasingness}
We recall from \cite{DUP19} confluence properties for linear~$(3,2)$-polygraphs modulo, and we give different methods to prove confluence modulo from local confluence modulo assumptions. In this subsection, we fixe a linear~$(3,2)$-polygraph modulo $(R,E,S)$.

\subsubsection{Branchings and confluence modulo}
A \emph{branching modulo $E$} of the linear $(3,2)$-polygraph modulo~$(R,E,S)$ is a triple~$(f,e,g)$ where $f$ and $g$ are positive $3$-cells of $S_3^\ell$ of length $1$ with $f$ non-identity and $e$ is a $3$-cell in $E_3^\ell$. Such a branching is depicted by
\begin{equation}
\label{E:branchingModulo}
\raisebox{0.55cm}{
\xymatrix @R=1.5em @C=2em {
u 
  \ar[r] ^-{f} 
  \ar[d] _-{e}
& 
u'
\\
v
  \ar[r] _-{g} 
&
v'
}}
\end{equation}
A branching modulo $E$ as in (\ref{E:branchingModulo}) is \emph{confluent modulo $E$} if there exists positive $3$-cells~$f',g'$ in~$S_3^\ell$ and a $3$-cell $e'$ in $E_3^\ell$ as in the following diagram:
\[
\raisebox{0.55cm}{
\xymatrix @R=1.5em @C=2em {
u
  \ar[r] ^-{f}
  \ar[d] _-{e}
&
u' 
  \ar@{.>}[r] ^-{f'} 
& 
w
  \ar@{.>}[d] ^-{e'}
\\
v
  \ar [r] _-{g}
&
v'
  \ar@{.>}[r] _-{g'} 
&
w'
}}
\]
We then say that the triple $(f',e',g')$ is a confluence modulo $E$ of the branching $(f,e,g)$ modulo $E$. The linear~$(3,2)$-polygraph $S$ is \emph{confluent modulo $E$} if all its branchings modulo $E$ are confluent modulo~$E$. A branching $(f,e,g)$ modulo $E$ is \emph{local} if $f$ is a positive $3$-cell of $S_3^\ell$ of length $1$, $g$ is a positive $3$-cell of $S_3^\ell$ and $e$ is a $3$-cell of $E_3^\ell$ such that $\ell(g) + \ell(e) = 1$. Following \cite[Section 2.2.6]{DUP19}, local branchings are classified in the following families: local aspherical, local Peiffer, local additive, local Peiffer modulo, local additive modulo and overlappings modulo which are all the remaining local branchings modulo.
Let $\sqsubseteq$ be the order on monomials of the linear~$(3,2)$-polygraph $S$ defined by $f \sqsubseteq g$ if there exists a context $C = \context{\square}$ of the free $2$-category $S_2^\ast$ generated by $S$ such that $g = C[f]$.
A \emph{critical branching modulo $E$} is an overlapping local branching modulo $E$ that is minimal for the order $\sqsubseteq$. When $\ERE$ is terminating, \cite[Theorem 2.2.7]{DUP19} proves that $S$ is confluent modulo $E$ if and only if the critical branchings $(f,e)$ and $(f,g)$ of $S$ modulo $E$ with $f$ positive $3$-cell in $S_3^\ell$ of length $1$, $g$ positive $3$-cell in $R_3^\ell$ of length $1$ and $e$ in $E_3^\ell$ of length $1$ are confluent modulo $E$.

\subsubsection{Labelling to the quasi-normal form}
To prove termination of $(R,E,S)$ when $\ERE$ is quasi-terminating, we introduce the notion of decreasingness modulo following the property of decreasingness introduced by Van-Oostrom in \cite{VOO94}.
Given a quasi-terminating linear~$(3,2)$-polygraph $P$, any $2$-cell $u$ in $P_2^\ell$ admits at least quasi normal form. For such a $2$-cell $u$, we fix a choice of a quasi normal form denoted by $\widetilde{u}$. Then we get a quasi-normal form map $s: P_2^\ast \fl P_2^\ast$ sending a $2$-cell $u$ in $P_2^\ast$ on $\widetilde{u}$. The \emph{labelling to the quasi-normal form}, labellling QNF for short associates to the map $s$ the labelling $\psiqnf: P_{\text{stp}} \fl \mathbb{N}$ defined by
\[ \psiqnf (f) = d(t_1(f), \cl{t_1(f)}) \]
where $d(t_1(f), \cl{t_1(f)})$ represent the minimal number of rewriting steps needed to reach the quasi normal form $\cl{t_1(f)}$ from $t_1(f)$, and $P_{\text{stp}}$ denotes the set of rewriting steps of $P$. Given a rewriting sequence $f = f_1 \star_1 \ldots \star_1 f_k$, we denote by $L^X(f)$ the set $\{ \psiqnf (f_1) ,\ldots , \psiqnf (f_k) \}$.

\subsubsection{Decreasingness modulo}
Let $(R,E,S)$ be a quasi-terminating linear~$(3,2)$-polygraph modulo equipped with its labelling $\psiqnf$ to the quasi-normal form on $S$. Recall from \cite{DUP19} that a local branching $(f,g)$ (resp. $(f,e)$) of $S$ modulo $E$ is decreasing modulo $E$ if there exists confluence diagrams of the following form
\[
\raisebox{0.55cm}{
	\xymatrix @R=2em @C=2em {
		{}
		\ar[r] ^-{f}
		\ar[d] _-{\fleq}
		&
		{} 
		\ar@{.>}[r] ^-{f'} 
		& 
		{} \ar@{.>}[r] ^-{g''} 
		& {} \ar@{.>}[r] ^-{h_1} 
		& {}
		\ar@{.>}[d] ^-{e'}
		\\
		{}
		\ar [r] _-{g}
		&
		{}
		\ar@{.>}[r] _-{g'} 
		& {} \ar@{.>}[r] _-{f''} 
		& {} \ar@{.>}[r] _-{h_2}
		&
		{}
	}} , \qquad \text{ (resp. }
	\raisebox{0.55cm}{
	\xymatrix @R=2em @C=2em {
		{}
		\ar[r] ^-{f}
		\ar[d] _-{e}
		&
		{} 
		\ar@{.>}[r] ^-{f'} 
		& {} \ar@{.>}[r] ^-{h_1} 
		& {}
		\ar@{.>}[d] ^-{e'}
		\\
		{}
		\ar@{.>} [rrr] _-{h_2}
		&
		{}
		& {} 
		&
		{}
	}} )
	\]
	such that the following properties hold:
	\begin{enumerate}[{\bf i)}]
		\item $k < \psiqnf (f)$ for all $k$ in $L^X(f')$.
		\item $k < \psiqnf (g)$ for all $k$ in $L^X (g')$.
		\item $f''$ is an identity or a rewriting step labelled by $\psiqnf (f)$.
		\item $g''$ is an identity or a rewriting step labelled by $\psiqnf (g)$.
		\item $k < \psiqnf (f)$ or $k < \psiqnf (g)$ for all $k$ in $L^X(h_1) \cup L^X(h_2)$ (resp. $k \leq \psiqnf (f)$ for any $k$ in $L^X(h_2)$ and $k' < \psiqnf (f)$ for any $k'$ in $L^X(h_1)$).
	\end{enumerate}

We then say that a linear~$(3,2)$-polygraph $(R,E,S)$ is decreasing modulo $E$ if all its local branchings are decreasing modulo $E$. From \cite[Theorem 2.3.8]{DUP19}, if $(R,E,S)$ is decreasing modulo $E$, then it is confluent modulo $E$. Following \cite{CDM19,DUP19}, one can prove confluence modulo of a linear~$(3,2)$-polygraph modulo $(R,E,S)$ such that $\ERE$ is quasi-terminating and $S$ is exponentiation free by proving decreasingness of its critical branchings.

\subsection{Linear bases from confluence modulo}
We recall following \cite{DUP19} the method to compute an hom-basis of a linear~$(2,2)$-category from a presentation of this latter by a linear~$(3,2)$-polygraph $P$ such that a subset of the relations satisfy an assumption of confluence modulo the remaining relations.

\subsubsection{Splitting of a linear~$(3,2)$-polygraph} Given a linear~$(3,2)$-polygraph $P$, recall that a \emph{subpolygraph} of $P$ is a linear~$(3,2)$-polygraph $P'$ such that $P'_i \subseteq P_i$ for any $0 \leq i \leq 3$. A \emph{splitting} of $P$ is a pair $(E,R)$ of linear~$(3,2)$-polygraphs such that:
\begin{enumerate}[{\bf i)}]
\item $E$ is a subpolygraph of $P$ such that $E_{\leq 1} = P_{\leq 1}$,
\item $R$ is a linear~$(3,2)$-polygraph such that $R_{\leq 2} = P_{\leq 2}$ and $P_3 = R_3 \coprod E_3$.
\end{enumerate}
Such a splitting is called  \emph{convergent} if we require that $E$ is convergent. Note that any linear~$(3,2)$-polygraph $P$ admits a convergent splitting given by $(P_0,P_1,P_2,\emptyset)$ and $(P_0,P_1,P_2,P_3)$. It is not unique in general. The data of a convergent splitting of a linear~$(3,2)$-polygraph $P$ gives two distinct linear~$(3,2)$-polygraphs $R= (P_0,P_1,P_2,R_3)$ and $E= (P_0,P_1,E_2,E_3)$ satisfying $R_{\leq 1} = E_{\leq 1}$ and $E_2 \subseteq P_2$, so that we can construct a linear~$(3,2)$-polygraph modulo from $R$ and $E$.

\subsubsection{(Quasi)-Normal forms modulo}
Let us consider a linear~$(3,2)$-polygraph modulo $(R,E,S)$ such that $(R,E,S)$ is confluent modulo $E$. We define two notions of normal forms modulo, depending on whether is terminating (or at least normalizing, that is each rewriting sequence reaches a normal form) or quasi-terminating.
 
 If $S$ is normalizing, each $2$-cell $u$ of $R_2^\ell$ admits at least one normal form with respect to $E$, and all these normal forms are congruent with respect to $E$. We fix such a normal form that we denote by $\widehat{u}$, with the convention that if $u$ is already a normal form with respect to $E$, then $\widehat{u} = u$. By convergence of $E$, any $2$-cell $u$ of $R_2^\ell$ admits a unique normal form with respect to $E$, that we denote by $\widetilde{u}$. Note that when $S$ is confluent modulo $E$, the element $\tilda{u}$ does not depend on the chosen normal form $\widehat{u}$ for $u$ with respect to $S$, since two normal forms of $u$ being equivalent with respect to $E$, they have the same normal form with respect to $E$. A \emph{normal form for $(R,E,S)$} of a $2$-cell $u$ in $R_2^\ell$ is a $2$-cell $v$ such that $v$ appears in the monomial decomposition of $\widetilde{w}$ where $w$ is a monomial in the support of $\widehat{u}$. Such a set is obtained by reducing a $2$-cell $u$ in $R_2^\ell$ into its chosen normal form with respect to $S$, then taking all the monomials appearing in the $E$-normal form of each element in $\text{Supp}(\widehat{u})$. Note that when $E$ is also right-monomial, as it is the case for the linear~$(3,2)$-polygraph of isotopies, then the $E$-normal form of a monomial in normal form with respect to $S$ already is a monomial.

If $S$ is quasi-terminating, instead of fixing a normal form $\widehat{u}$ with respect to $S$ for any $u$ in $R_2^\ell$, we fix a choice of a quasi-normal form $\cl{u}$ for $u$ satisying $\cl{u} = u$ if $u$ already is a quasi-normal form with respect to $S$. By confluence modulo, $u$ and $v$ are $2$-cells of $R_2^\ell$ such that there is a $3$-cell $e: u \fl v$ in $\tck{E}$, then the $2$-cells $\cl{u}$ and $\cl{v}$ are equivalent modulo $E$. A \emph{quasi-normal form} for $(R,E,S)$ is a monomial appearing in the monomial decomposition of the $E$-normal form of a monomial in $\text{Supp}(\cl{u})$.

\subsubsection{Hom-bases from confluence modulo}
Following \cite[Theorems 2.5.4 $\&$ 2.5.6]{DUP19}, if $P$ is a linear~$(3,2)$-polygraph presenting a linear~$(2,2)$-category $\mathcal{C}$, $(E,R)$ is a splitting of $P$ and $(R,E,S)$ is a linear~$(3,2)$-polygraph modulo such that $S$ is normalizing (resp. quasi-terminating), and confluent modulo $E$, then the set of normal forms (resp. quasi-normal forms) for $(R,E,S)$ is a hom-basis of $\mathcal{C}$.

\section{A convergent presentation of the simply-laced KLR algebras} 
\label{S:KLRAlgebras}
We provide a convergent presentation of a linear~$(2,2)$-category $\catklr$ encoding the KLR algebras in its spaces of $2$-cells, and prove that these algebras admit Poincar\'e-Birkhoff-Witt bases.

\subsection{Cartan datum and Kac-Moody algebras} \label{SS:KacMoodySetup}
Let us recall the notions of Cartan datum and root datum needed to introduce the KLR algebras and the $2$-category $\Ug$. 

\subsubsection{Cartan matrices and Cartan datum} A matrix $A=(a_{i,j}) \in \mathcal{M}_n(\mathbb{K})$ is called a \emph{generalized Cartan matrix} if it satisfies the following conditions:

\begin{enumerate}[{\bf i)}]
\item for any $1 \leq i \leq n$, $a_{i,i} = 2$;
\item for any $ 1 \leq i < j \leq n$, $a_{i,j} \in \Z_{<0}$;
\item for any $ 1 \leq i , j \leq n$, $a_{i,j}=0$ if and only if $a_{j,i} = 0$.
\end{enumerate}
Given $A=(a_{i,j})_{1 \leq i,j \leq n}$ a matrix of rank $l$ with coefficients in $\mathbb{K}$, we say that \textit{a realization of $A$} is the data of a 
triple $(\mathfrak{h}, \Pi, \Pi^\vee)$ where $\mathfrak{h}$ is a $\mathbb{K}$-vector space and $\Pi= \lbrace \alpha_1, \dots, \alpha_n \rbrace \subset \mathfrak{h}^*$, $\Pi^\vee = \lbrace \alpha_1^\vee , \dots, \alpha_n^\vee \rbrace \subset \mathfrak{h} $ satisfying:
\begin{center}
\begin{enumerate}[{\bf i')}]
\item $\Pi$ and $\Pi^\vee$ are free;
\item For all $1 \leq i,j \leq n$, $\langle \alpha_i^\vee , \alpha_j \rangle = a_{i,j}$;
\item dim($\mathfrak{h}$)= $2n-l$.
\end{enumerate}
\end{center}
We call $\Pi$ the basis of roots and $\Pi^\vee$ the basis of co-roots. The elements of $\Pi$ and $\Pi^\vee$ are respectively called \textit{simple roots} and \textit{simple co-roots}. Following \cite[Chapter 1]{KAC90}, from such a generalized Cartan matrix and a realization of it, one can build a \textit{Kac-Moody algebra} $\mathfrak{g}(A)$. As in the usual representation theory of Lie algebras, an integrable $\mathfrak{g}(A)$-module admits a decomposition of the form 
$$ \text{$ V = \bigoplus\limits_{\lambda \in \mathfrak{h}^\star} V_\lambda$ where $V_\lambda= \lbrace v \in V \quad | \quad h(v)= \langle \lambda,h \rangle v$ for $h \in \mathfrak{h} \rbrace$}. $$
$V_\lambda$ is then called a \textit{weight space} and $\lambda \in \mathfrak{h}^\star$ is called a \textit{weight} if $V_\lambda \ne 0$. \\

\begin{definition}
A \textit{Cartan datum} $(I,\cdot)$ consists of a finite set $I$ and a bilinear form on $\Z[I]$, taking values in $\Z$ such that:
\begin{itemize}
\item[{\bf i)}] $i.i \in \lbrace 2,4,6, \dots \rbrace  $ for any $i \in I$;
\item[{\bf ii)}] $ - d_{i,j} := 2 \frac{i.j}{i.i} \in \lbrace 0,-1,-2, \dots \rbrace $ for any $i \ne j \in I$.
\end{itemize}
We say that such a Cartan datum is \textit{simply-laced} if the two following conditions hold:
\begin{itemize}
\item[{\bf i')}] For any $i \in I$, $i \cdot i=2$; 
\item[{\bf ii')}] For any $i,j \in I$,  $i \cdot j \in \lbrace 0,-1 \rbrace$.
\end{itemize}
\end{definition}

\begin{remark}
If we set $(I, \cdot)$ a Cartan datum and $A= \left( 2 \frac{i.j}{i.i} \right)_{1 \leq i,j \leq \# I}$, then A is a generalized Cartan matrix and so we can associate to each Cartan datum a Kac-Moody algebra.
\end{remark}

From now, we fix $\Gamma$ a non-oriented graph (with possible loops and multiple edges) whose set of vertices is denoted by $I$. We assume here that $I$ is a finite set. In general, there is a one-to-one correspondance between such graphs and Cartan data. In particular, let $\Gamma$ be a simply-laced graph, that is without loops nor multiple edges. Then we build a simply-laced Cartan datum associated to it as follows: let $\cdot $ be a bilinear form on $\Z[I]$ such that: \begin{equation} \label{sldatum} \left\lbrace \begin{array}{ccc}
 i \cdot i = 2  &          \\
 i \cdot j = -1 &  \textrm{if there is an edge in $\Gamma$ from $i$ to $j$} \\
 i \cdot j = 0  & \textrm{otherwise}.
 \end{array} \right. 
 \end{equation} 
 
\subsubsection{Root datum and quantum groups}
Let $(I,\cdot)$ be a Cartan datum. \textit{A root datum of type $(I,\cdot)$} consists of
\begin{itemize}
  \item two free finitely generated abelian groups $X$,$Y$ and a perfect
  pairing $\langle,\rangle \maps Y \times X \to \Z$;
  \item injections $I \subset X $ $\;(i \mapsto i_X)$ and $I \subset
  Y$ $\; (i \mapsto i)$ such that $\langle i, j_X \rangle = 2\frac{i \cdot j}{i \cdot i}= -d_{ij}$ for all $i,j \in
  I$.
\end{itemize}

We associate to such a root datum a quantum group $\textbf{U}$, which is the unital associative  $\Q(q)$-algebra given by generators $E_i$, $F_i$, $K_{\mu}$ for $i
\in I$ and $\mu \in Y$, subject to a family of relations given in \cite[Section 3.1]{LUS10}.

\subsubsection{The sets $\text{Seq}(\mathcal{V})$ and $\text{SSeq}(\mathcal{V})$} Following \cite{KL1,KL3}, we introduce useful sets to define the KLR algebras and the KLR $2$-category. Let $I$ be the set of vertices of a simply-laced graph $\Gamma$. Let $\mathcal{V}= \sum\limits_{i \in I}{\mathcal{V}_{i} . i} \in \N[I]$ be an element of $\mathbb{N}[I]$, the free semi-group generated by $I$. We set $m := | \mathcal{V}| = \sum{\mathcal{V}_i}$. \\

We consider the set $\textrm{Seq}(\mathcal{V})$ which consists of all sequences of vertices of $\Gamma$ with length $m$ in which the vertex $i$ appears exactly $\mathcal{V}_{i}$ times. 
For instance, $\textrm{Seq}( 3i + j) = \lbrace iiij, iiji, ijii, jiii \rbrace$. There is an action of the symmetric group $\mathcal{S}_m$ on the set $\textrm{Seq}(\mathcal{V})$ defined by
\[ s_k \cdot i_1 \dots i_m = i_1 \dots i_{k+1} i_k \dots i_m \]
for any $1 \leq k \leq m-1$, where $s_k$ denotes the permutation $(k \: \: k+1)$ of $\mathcal{S}_m$.

We will also consider in Section \ref{S:KLR2Cat} a signed version of this set, with \textit{signed sequences} of vertices of $\Gamma$: 
$$ \text{$ \textbf{i} =
(\epsilon_1i_1,\epsilon_2i_2, \dots, \epsilon_mi_m)$, where $\epsilon_1, \dots,
\epsilon_m \in \{ +,-\}$ and $i_1, \dots, i_m \in I$}. $$
We define $\SSv$ to be the set of all such signed sequences. We say that a sequence is \textit{positive} (resp. \emph{negative}) if all signs $\epsilon_i$ are positive (resp. negative).

\subsection{The KLR algebras}
\label{SS:KLRAlgebras}
We recall here Rouquier's algebraic definition of the KLR algebras \cite{ROU08} and their diagrammatic interpretation provided by Khovanov and Lauda in \cite{KL1}.

\begin{definition} \label{D:KLRAlgebras} \cite[Definition 3.2.1]{ROU08}
Let $Q=(Q_{i,j})_{i,j \in I}$ a matrix with coefficients in $\mathbb{K}[u,v]$, where $u$ and $v$ are indeterminates, such that $Q_{i,i} = 0$ for any $i$ in $I$. For any $\mathcal{V}$ in $\N[I]$, we define a (possibly non-unitary) $\mathbb{K}$-algebra $H_\mathcal{V}(Q)$ by generators and relations. It is generated by elements $1_{\textbf{i}}, x_{k,\textbf{i}}$ for $ k \in \lbrace 1,\dots, n \rbrace $ and $\tau_{k, \textbf{i}}$ for $ k \in \lbrace 1,\dots, n \rbrace $ and $\textbf{i} \in \textrm{Seq}(\mathcal{V})$.
The relations are:
\begin{enumerate}[{\bf i)}]
\begin{multicols}{2}
\item $1_{\textbf{i}} 1_{\textbf{j}} = \delta_{\textbf{i},\textbf{j}} 1_{\textbf{i}}$
\item $\tau_{k, \textbf{i}} = 1_{s_{k}(\textbf{i})} \tau_{k,\textbf{i}} 1_{\textbf{i}}$
\item $x_{k, \textbf{i}} = 1_{\textbf{i}} x_{k, \textbf{i}} 1_{\textbf{i}}$
\item $x_{k, \textbf{i}} x_{l, \textbf{i}} = x_{l, \textbf{i}} x_{k, \textbf{i}} $
\item $\tau_{k, s_k(\textbf{i})} \tau_{k, \textbf{i}} = Q_{i_k, i_{k+1}}(x_{k, \textbf{i}}, x_{k+1,\textbf{i}})$
\item $\tau_{k, s_l(\textbf{i})} \tau_{l, \textbf{i}} = \tau_{l, s_k(\textbf{i})} \tau_{k, \textbf{i}}$ if $|k-l| > 1$
\end{multicols}
\item $\tau_{k,\textbf{i}} x_{l, \textbf{i}} - x_{s_k(l),s_k(\textbf{i})} \tau_{k,\textbf{i}} = \left\lbrace \begin{array}{ccc}
-1_{\textbf{i}}   & \mbox{if} & l=k \quad \textrm{and} \quad i_k = i_{k+1} \\
1_{\textbf{i}} & \mbox{if} & l=k+1 \quad \textrm{and} \quad i_k=i_{k+1} \\
0 & & \textrm{otherwise} .

\end{array} \right.$
\item $\tau_{k+1, s_k s_{k+1}(\textbf{i})} \tau_{k, s_{k+1}(\textbf{i})} \tau_{k+1, \textbf{i}} - \tau_{k, s_{k+1} s_k ( \textbf{i}) } \tau_{k+1, s_k(\textbf{i})} \tau_{k, \textbf{i}} = \\ \left\lbrace
\begin{array}{ccc}
(x_{k+2,\textbf{i}} - x_{k,\textbf{i}})^{-1} (Q_{i_k,i_{k+1}}(x_{k+2,\textbf{i}},x_{k+1,\textbf{i}}) - Q_{i_k,i_{k+1}}(x_{k,\textbf{i}},x_{k+1,\textbf{i}}))  & \mbox{if} & i_k = i_{k+2}  \\
0 &  & \textrm{otherwise} 

\end{array} \right. $

\end{enumerate}
\end{definition}

In \cite{KL2}, Khovanov and Lauda gave a definition of a ring associated to an element $\mathcal{V}\in \N[I]$ which is in fact a specialization of Rouquier's algebra $H_\mathcal{V}(Q)$ in which $$ \text{ $Q_{i,j}(u,v)= u^{d_{i,j}} + v^{d_{j,i}}, \quad \forall \quad i,j \in I$, where $d_{i,j}= -2 \frac{i \cdot j}{i \cdot i}$}.$$ 
In the simply-laced setting, these coefficients are equal to $0$ when $i$ and $j$ are not linked by an edge in the graph and to $1$ when they are. Moreover, they provide a diagrammatic approach for these algebras: for $\textbf{i} = i_1 \dots i_m \in \textrm{Seq}(\mathcal{V})$, the generators are pictured by the diagrams
\[ x_{k, \textbf{i}} = \gendot{i_1}{i_k}{i_m} \quad \text{and} \quad \tau_{k,\textbf{i}} =\gencross{i_1}{i_k}{i_{k+1}}{i_m}  \]

\noindent The local relations of \ref{D:KLRAlgebras} are then diagrammatically depicted by:

\begin{eqnarray} \label{e:klr1} 
   \scalebox{0.8}{\xy   (0,0)*{\tcross{i}{j}}; \endxy}
 & = & \left\lbrace
\begin{array}{ccc}
  0 & \qquad & \text{if $i=j$, } \\ \\
  \scalebox{0.8}{\did{i}{j}}
  & &
 \text{if $i \cdot j=0$, }
  \\    \\
  \scalebox{0.8}{\diddl{i}{j}{d_{i,j}}}
  \quad \raisebox{4mm}{$+$} \quad
   \scalebox{0.8}{\diddr{i}{j}{d_{j,i}}}
 & &
 \text{if $i \cdot j=-1$. }
\end{array}
\right.
\end{eqnarray}

\begin{eqnarray} \label{klr2}
  \scalebox{0.8}{\xy  (0,0)*{\dcrossul{i}{j}};  \endxy}
 \quad  = \;\;
   \scalebox{0.8}{\xy  (0,0)*{\dcrossdr{i}{j}};   \endxy} + \; \delta_{i,j}  \scalebox{0.8}{\xy (-3,0)*{\ident{i}}; (3,0)*{\ident{i}}; \endxy}
\; , \qquad
   \scalebox{0.8}{\xy  (0,0)*{\dcrossur{i}{j}};  \endxy}
 \quad = \;\;
   \scalebox{0.8}{\xy  (0,0)*{\dcrossdl{i}{j}};  \endxy} - \; \delta_{i,j}  \scalebox{0.8}{\xy (-3,0)*{\ident{i}}; (3,0)*{\ident{i}}; \endxy}
\end{eqnarray}

\begin{eqnarray} \label{ybg_simplylaced}   
\scalebox{0.9}{\xy  (0,0)*{\ybg{i}{j}{k}}; \endxy}
  &=&
\scalebox{0.9}{\xy (0,0)*{\ybd{i}{j}{k}}; \endxy}
 \qquad \text{unless $i=k$ and $i \cdot j=-1$   \hspace{1in} }
\end{eqnarray}
\begin{eqnarray}
\label{ybg_simplylaced2}   
\hspace{-2cm}      
\scalebox{0.9}{\xy (0,0)*{\ybg{i}{j}{i}}; \endxy}
  &-&
\scalebox{0.9}{\xy (0,0)*{\ybd{i}{j}{i}}; \endxy}
 \quad = \quad \raisebox{-7mm}{$\scalebox{0.9}{\tid{i}{j}{i}}$} 
 \qquad \text{if $i \cdot j=-1$ }
\end{eqnarray}

Note that the first diagrammatic relation corresponds to the relation $\mathbf{v)}$ in \ref{D:KLRAlgebras}, the second relation corresponds to relation $\mathbf{vii)}$ and the last one corresponds to relation $\mathbf{viii)}$ for this particular choice of polynomials $Q_{i,j}$. The other relations are not taken into account since they are structural relations when the algebra is interpreted as endomorphism spaces of $2$-cells in the linear~$(2,2)$-category $\catklr$ defined in \ref{SSS:CategoricalStructure}. Namely, the first relation corresponds to the fact that $1_{\mathbf{i}}$ is an identity, and the other relations correspond to exchange relations (\ref{E:ExchangeRel}).
\begin{remark}
\label{R:SimplyLaced}
	Throughout this paper, we study the case of simply-laced Cartan datum for simplicity in the proofs of confluence of critical branchings. In the general case, the KLR relations are more complicated, for instance the relation reducing a double crossing or the Yang-Baxter braid become
	\[  \raisebox{-5mm}{$\scalebox{0.9}{\tcross{i}{j}}$} =  \raisebox{-5mm}{$\didldot{i}{j}{d_{i,j}}$} +  \raisebox{-5mm}{$\didrdot{i}{j}{d_{j,i}}$} \]
	whenever $i \cdot j \ne 0$, and
	\[ \raisebox{-7mm}{$\scalebox{0.9}{\ybg{i}{j}{k}}$} =  \raisebox{-7mm}{$\scalebox{0.9}{\ybd{i}{j}{k}}$} + \sum\limits_{a=0}^{d_{i,j}-1}  \raisebox{-6mm}{$\tidlrdots{i}{j}{k}{a}{d_{i,j}-1-a}$} \]
	whenever $i = k$ and $i \cdot j \ne 0$. However, we expect that the proof of confluence in the general setting works similarly as in the simply-laced setting, but the confluence of critical branchings is more difficult to ensure due to these relations.
\end{remark}

\subsubsection{$2$-categorical structure}
\label{SSS:CategoricalStructure}
Following \cite{KL1},  we consider for any $\textbf{i}$ and $\textbf{j}$ in $\textrm{Seq}(\mathcal{V})$ the set ${}_{\textbf{j}} R(\mathcal{V})_{\textbf{i}} $ of \emph{braid-like Khovanov-Lauda diagrams} with source $\textbf{i}$ and target $\textbf{j}$, given by string diagrams satisfying the following conditions:
\begin{itemize} 
\item[-] the strands are labelled by vertices of $\Gamma$, and reading the labels on the bottom (resp. the top) of the diagram gives the sequence $\textbf{i}$ (resp. $\textbf{j}$);
\item[-] a strand does not intersect with itself.
\end{itemize}
For any $\textbf{i}$ and $\textbf{j}$ in $\textrm{Seq}(\mathcal{V})$, the set ${}_{\textbf{j}} R(\mathcal{V})_{\textbf{i}} $ is a $\K$-vector space. Following \cite{KL1}, we have that $R (\mathcal{V} ) = \bigoplus_{\textbf{i}, \textbf{j} \in \text{Seq}(\mathcal{V})} {}_{\textbf{j}} R(\mathcal{V})_{\textbf{i}}$.
Let us consider the linear~$(2,2)$-cateory $\catklr$ defined by:
\begin{enumerate}[{\bf i)}]
\item only one $0$-cell denoted by $\ast$,
\item its generating one cells are the elements of $I$, and the $\star_0$ composition of $1$-cells is formal concatenation of vertices, so that the $1$-cells of $\catklr$ correspond to sequences of vertices of $I$.
\item its generating $2$-cells are given by
\begin{equation}
\label{E:GeneratorsKLR}
\raisebox{-6mm}{$\scalebox{0.8}{\dcrossnodot{i}{j}}$}: i \star_0 j \fl j \star_0 i, \quad \quad \scalebox{0.9}{$\dottsl{i}$} : i \fl i 
\end{equation}
for any $i$ and $j$ in $I$, so that the $2$-cells of $\catklr$ are obtained by all the diagrams one can form by vertical and horizontal compositions of these generating $2$-cells. We require that the $2$-cells of $\catklr$ are subject to relations (\ref{e:klr1}), (\ref{klr2}), (\ref{ybg_simplylaced}) and (\ref{ybg_simplylaced2}).
\end{enumerate}

Note that it is clear from the definition of $\catklr$ that if $\textbf{i}$ and $\textbf{j}$ are sequences of vertices of $I$ which does not belong to the same set $\text{Seq}(\mathcal{V})$, then we have $\catklr_2 (\textbf{i}, \textbf{j}) = \emptyset$. When they belong to the same $\text{Seq}(\mathcal{V})$, we have
$ \catklr_2 (\textbf{i}, \textbf{j}) = {}_{\textbf{j}} R(\mathcal{V})_{\textbf{i}} $. As a consequence, we have an isomorphism of algebras \[ R(\mathcal{V}) \simeq \coprod\limits_{\textbf{i}, \textbf{j} \in \text{Seq}(\mathcal{V})} \catklr_2 (\textbf{i}, \textbf{j}) \]
so that for any $\mathcal{V}$ in $\N [I]$, the KLR algebra $\catklr$ is encoded in the linear~$(2,2)$-category $\catklr$.

\subsection{The linear~$(3,2)$-polygraph $\text{KLR}$}
\label{SS:LinearPolKLR}
In this section, we will define linear~$(3,2)$-polygraphs presenting these simply-laced KLR algebras and prove that they are convergent.

\begin{definition}
 Let $\text{KLR}$ be the linear~$(3,2)$-polygraph defined by:
   \begin{itemize}
   \item[-] One 0-cell denoted by $\ast$,
   \item[-] Its generating $1$-cells are the elements $i$ of $I$,
   \item[-] Its generating $2$-cells are given by the elements of (\ref{E:GeneratorsKLR}),
   \item[-] Its generating $3$-cells are given by the following oriented relations:
   \begin{enumerate}[{\bf i)}]
\item For any $i,j \in I$, \begin{center} $\scalebox{0.9}{\xymatrix{\raisebox{-7mm}{$\dcrossul{i}{j}$} \ar@3[r] ^-{\alpha_{i,j}^L} & \raisebox{-7mm}{$\dcrossdr{i}{j}$} }} \qquad \text{and} \qquad   \scalebox{0.9}{\xymatrix{\raisebox{-7mm}{$\dcrossur{i}{j}$} \ar@3[r] ^-{\alpha_{i,j}^R} & \raisebox{-7mm}{$\dcrossdl{i}{j}$} }} $ \end{center}
\item For any $i \in I$, \begin{center} $\scalebox{0.9}{\xymatrix{\raisebox{-7mm}{$\dcrossul{i}{i}$} \ar@3[r] ^-{\alpha_{i}^L} & \raisebox{-7mm}{$\dcrossdr{i}{i}$}  + \raisebox{-7mm}{$\did{i}{i}$}}} \qquad \text{and} \qquad   \scalebox{0.9}{\xymatrix{\raisebox{-7mm}{$\dcrossur{i}{i}$} \ar@3[r] ^-{\alpha_{i}^R} & \raisebox{-7mm}{$\dcrossdl{i}{i}$} - \raisebox{-7mm}{$\did{i}{i}$} }} $ \end{center}
\item For any $i \in I$, \begin{center} $\scalebox{0.8}{\xymatrix{\raisebox{-6mm}{$\tcross{i}{i}$} \ar@3[r] ^-{\beta_i} & 0 }} $   \end{center}
\item For any $i,j \in I$ such that $i \cdot j=0$, \begin{center} $\scalebox{0.8}{\xymatrix{\raisebox{-7mm}{$\tcross{i}{j}$} \ar@3[r] ^-{\beta_{i,j}} & \raisebox{-7mm}{$\did{i}{j}$} }} $    \end{center}
\item For any $i,j \in I$ such that $i \cdot j=-1$, \begin{center} $\scalebox{0.8}{\xymatrix{\raisebox{-7mm}{$\tcross{i}{j}$} \ar@3[r] ^-{\beta_{i,j}} & \raisebox{-7mm}{$\didl{i}{j}$} + \raisebox{-7mm}{$\didr{i}{j}$} }} $   \end{center}
\item For any $i,j,k \in I$, and unless $i=k$ and $i \cdot j \ne -1$, \begin{center} $\scalebox{0.8}{\xymatrix{ \raisebox{-7mm}{$\ybg{i}{j}{k}$} \ar@3[r] ^-{\gamma_{i,j,k}} & \raisebox{-7mm}{$\ybd{i}{j}{k}$}}} $ \end{center}
\item For any $i,j \in I$ such that $i \cdot j = -1$,   \begin{center} $\scalebox{0.8}{\xymatrix{ \raisebox{-7mm}{$\ybg{i}{j}{i}$} \ar@3[r] ^-{\gamma_{i,j,k}} & \raisebox{-7mm}{$\ybd{i}{j}{i}$} + \raisebox{-7mm}{$\tid{i}{j}{i}$} }} \quad . $ \end{center}
   \end{enumerate}
   \end{itemize}
\end{definition}   

We then establish the first main result of this paper:
\begin{theorem}
\label{T:ConvergentPresKLR}
The linear~$(3,2)$-polygraph $\text{KLR}$ is a convergent presentation of the linear~$(2,2)$-category $\catklr$.
\end{theorem}

The $3$-cells of the linear~$(3,2)$-polygraph $\text{KLR}$ correspond to an orientation of the relations of $\catklr$. As a consequence, it is a presentation of the linear~$(2,2)$-category $\catklr$. On the one hand, we show that $\text{KLR}$ is terminating using the derivation method given in \cite[Theorem 4.2.1]{GM09} to prove termination of $3$-polygraphs, extended in the linear setting in \cite{DUP19}. On the other hand, we prove that $\text{KLR}$ is confluent by proving confluence of all its critical branchings, using \cite[Theorem 4.2.13]{AL16}.  

\subsubsection{Termination of \text{KLR}} 
\label{SSS:TerminationKLR}
Proving termination by derivation consists in the construction of $2$-functors  $X: \catklr \fl \catego{Ord}$ and $Y: (\catklr)^{\text{op}} \fl \catego{Ord}$, where $\catego{Ord}$ is the category of partially ordered sets and monotone maps, viewed as a $2$-category with one $0$-cell, ordered sets as $1$-cells and monotone maps as $2$-cells and $(\catklr)^{\text{op}}$ is the $2$-category $\catklr$ in which all sources and targets of $2$-cells have been exchanged, and a derivation $d: \catklr \fl  M_{X,Y,G}(C)$ with values in the module $M_{X,Y,G}$ on $\catklr$ defined from $X$, $Y$ and $G$ an internal abelian group in $\catego{Ord}$ defined as in \cite{GM09}. This data is required to satisfy the following conditions to ensure termination of the linear~$(3,2)$-polygraph $\text{KLR}$:
\begin{enumerate}[{\bf i)}]
\item For any $1$-cell $a$ in $\text{KLR}_1$, the sets~$X(a)$ and $Y(a)$ are non-empty and, for any generating $3$-cell $\alpha$ in $\text{KLR}_3$, the inequalities $X(s(\alpha))\geq X(h)$ and $Y(s(\alpha))\geq Y(h)$ hold for any $h$ in $\text{Supp}(t(\alpha))$.
\item The addition in $G$ is strictly monotone in both arguments and every decreasing sequence of non-negative elements of $G$ is stationary.
\item For any monomial $f$ in $\text{KLR}_2^\ell$, we have $d(f) \geq 0$ and, for every $3$-cell $\alpha$ in $\text{KLR}_3$, the strict inequality $d(s(\alpha))> d(h)$ holds for any $h$ in $\text{Supp}(t(\alpha))$.
\end{enumerate}

We consider the internal abelian group $\mathbb{Z}$ in $\catego{Ord}$ and we set $Y$ to be the trivial $2$-functor. We define the values of the $2$-functor $X: \textrm{KLR}_2^* \to \textbf{Ord}$ on generating $2$-cells by: 
\[ X \big( \ident{} \big) (i)=i \hspace{1cm} X \big( \identdot{} \big) (i)= i+1 \hspace{1cm} X \big( \raisebox{-7mm}{$\scalebox{0.9}{\crossing{}{}}$} \big) (i,j) = (j+1,i ) \quad \forall i,j \in \N \]
for any labels of the strands of the diagrams, and we consider the $\text{KLR}_2^\ast$-module $M_{X,*,\Z}$. The following inequalities hold
\begin{align*}
X \big( \raisebox{-6mm}{$\scalebox{0.9}{\tcross{}{}}$} \big) (i,j) = (i+1,j+1) \geq & \textrm{max}\big( X \big( \raisebox{-6mm}{$\scalebox{0.9}{\didl{}{}}$} \big)(i,j), X \big( \raisebox{-6mm}{$\scalebox{0.9}{\didr{}{}}$} \big) (i,j), X \big( \raisebox{-5mm}{$\scalebox{0.9}{\did{}{}}$} \big) (i,j) \big) \\
& =\textrm{max} \big( (i+1,j) , (i,j+1) , (i,j) \big),
\end{align*}
\[  X \big( \raisebox{-6mm}{$\scalebox{0.9}{\dcrossul{}{}}$} \big) = (j+2,i) \geq (j+2,i) =\textrm{max} \big( X \big( \raisebox{-6mm}{$\scalebox{0.9}{\dcrossdr{}{}}$} \big) (i,j), X \big( \raisebox{-6mm}{\scalebox{0.9}{\did{}{}}} \big)(i,j) \big), \]
\[ X \big( \raisebox{-6mm}{$\scalebox{0.9}{\dcrossur{}{}}$} \big) (i,j) = (j+1,i+1) \geq (j+1,i+1)= \textrm{max} \big( X \big(   \raisebox{-6mm}{$\scalebox{0.9}{\dcrossdr{}{}}$} \big)(i,j), X \big( \raisebox{-5mm}{$\scalebox{0.7}{\did{}{}}$  } \big) (i,j) \big), \]
\[ X \big( \raisebox{-7mm}{$\scalebox{0.7}{\ybg{}{}{}}$} \big) (i,j,k) = (k+2,j+1,i) \geq \textrm{max}\big( X \big( \raisebox{-5mm}{$\scalebox{0.8}{\ybd{}{}{} }$} \big) (i,j,k), X \big( \raisebox{-5mm}{$\scalebox{0.8}{\tid{}{}{}}$} \big) (i,j,k) \big). \]

Let us now define the derivation $d$ of $\text{KLR}_2^*$ into $M_{X,*,\Z}$ on the generating $2$-cells by \[ d \big( \raisebox{-7mm}{$\scalebox{0.9}{\crossing{}{}}$} \big) (i,j) = i \hspace{1cm} d \big( \ident{} \big) (i) = 0 = d \big( \identdot{} \big)(i). \]
The following inequalities hold:
\[ d \big( \raisebox{-7mm}{$\scalebox{0.9}{\tcross{}{}}$} \big) (i,j) = i+j+1 > 0 = d \big( \raisebox{-7mm}{$\scalebox{0.9}{\did{}{}}$} \big) (i,j) = \textrm{max} \big( d \big( \raisebox{-7mm}{$\scalebox{0.9}{\didl{}{}}$} \big) , d \big( \raisebox{-7mm}{$\scalebox{0.9}{ \didr{}{} }$} \big)  \big)(i,j), \]
\[  d \big( \raisebox{-7mm}{$\scalebox{0.8}{\ybg{}{}{}}$} \big) (i,j,k) = 2i+j+1 > 2i+j = \textrm{max}\big( d \big( \raisebox{-5mm}{$\scalebox{0.8}{\ybd{}{}{}}$} \big) , d \big( \raisebox{-5mm}{$\scalebox{0.8}{\tid{}{}{}}$}  \big) \big) (i,j,k), \]
\[ d \big(  \raisebox{-7mm}{$\scalebox{0.9}{\dcrossul{}{} }$} \big) (i,j) = i > 0 = \textrm{max} \big( d \big( \raisebox{-7mm}{$\scalebox{0.9}{\dcrossdr{}{}}$} \big), d \big( \raisebox{-6mm}{$\scalebox{0.9}{\did{}{} }$} \big) \big) (i,j),  \]
\[ d \big( \raisebox{-7mm}{$\scalebox{0.9}{\dcrossur{}{}}$} \big) (i,j) =  i > 0 = \textrm{max} \big( d \big( \raisebox{-7mm}{$\scalebox{0.9}{\dcrossul{}{} }$} \big) , d \big( \raisebox{-6mm}{$\scalebox{0.9}{ \did{}{} }$} \big) \big) (i,j).  \]

so that the $2$-functor $X$ and the derivation $d$ satisfy the conditions ${\bf i)}$, ${\bf ii)}$ and ${\bf iii)}$, and thus the linear~$(3,2)$-polygraph $\text{KLR}$ is terminating.

\subsubsection{Critical branchings of \text{KLR}}
\label{SSS:CriticalBranchingsKLRAlgebra}
There are four different forms for the sources of $3$-cells, that we denote as follows:
\[ \raisebox{-7mm}{$\scalebox{0.9}{\dcrossul{i}{j}}$} \leftrightsquigarrow \text{ldot}_{i,j} \;, \quad \raisebox{-7mm}{$\scalebox{0.9}{\dcrossur{i}{j}}$} \leftrightsquigarrow \text{rdot}_{i,j} \; , \quad
 \raisebox{-7mm}{$\scalebox{0.8}{\tcross{i}{j}}$} \leftrightsquigarrow \text{dcr}_{i,j} \;, \quad \raisebox{-6mm}{$\scalebox{0.7}{\ybg{i}{j}{k}}$} \leftrightsquigarrow \text{ybg}_{i,j,k} \; . \]

There are six families of regular critical branchings, which we all prove confluent confluent in Appendix \ref{A:KLRCriticalBranchings}. The exhautive list of critical branchings is given below, listing all the pairs of sources of $3$-cells that overlap:

\begin{enumerate}[{\bf a)}]
\item Crossings with two dots of the form (ldot$_{i,j}$, rdot$_{i,j}$) for any $i$ and $j$ in $I$.
\item Triple crossings of the form (dcr$_{j,i}$, dcr$_{i,j}$) for any $i,j$ in $I$ and any value of the bilinear form $i \cdot j$.
\item Double crossings with dots of the form
(ldot$_{j,i}$, dcr$_{i,j}$) and (rdot$_{j,i}$, dcr$_{i,j}$) for any $i$ and $j$ in $I$ and any value of $i \cdot j$.
\item Double Yang-Baxters of the form (ybg$_{j,k,i}$, ybg$_{i,j,k}$) for any $i,j$ and $k$ in $I$ and any values of $i \cdot j$, $j \cdot k$ and $i \cdot k$.
\item Yang-Baxters and crossings of the form 
(ybg$_{i,j,k}$, dcr$_{j,i}$) and (dcr$_{k,j}$, ybg$_{i,j,k}$)  for any $i,j$ and $k$ in $I$ and any values of $i \cdot j$ and $j \cdot k$.
\item Yang Baxter and dots of the form
(ldot$_{k,j}$, ybg$_{i,j,k}$) ; (rdot$_{k,j}$, ybg$_{i,j,k}$) ; (rdot$_{i,k}$, ybg$_{i,j,k}$) for any $i,j$ and $k$ in $I$ and any values of $i \cdot j$, $i \cdot k$ and $j \cdot k$.
\end{enumerate}

There also are right-indexed critical branchings of the form 
\begin{equation} 
\label{indexyb}
\scalebox{0.792}{$\indexyb{K}$}
\end{equation}

Following the study of the $3$-polygraphs of permutations in \cite[Section 5.4]{GM09}, the $2$-cells $k$ in normal form that can be plugged in (\ref{indexyb}) are identities or simple crossings. If we take into account the additional dot $2$-cell, one can make move using the KLR relations any dot to the bottom of the diagram. For instance, we have the following rewriting sequence in $\text{KLR}^\ast$:
\begin{align*}
\xymatrix{ \scalebox{0.792}{$\exdotsa$} \ar@3[r] &  \scalebox{0.822}{$\exdotsb$} \quad (\pm D_1^{\text{low}}) \ar@3[r] & \dots \ar@3[r] & \scalebox{0.822}{$\exdotsc$} \left( \pm \sum D_i^{\text{low}} \right)}
\end{align*}
where the $D_i^{\text{low}}$ are eventually diagrams with fewer crossings. With this observation, the set of normal forms we can plug in (\ref{indexyb}) is given by the following $2$-cells:
\begin{enumerate}[{\bf i)}]
\item $\identdots{i}{n}$ for every $n \in N$, which is an identity if $n=0$.
\item \raisebox{-6mm}{$\scalebox{1}{\dcrossdldot{i}{l}{n}}$}  \: for all $n \in \N$ and for any $l$ in $I$.
\end{enumerate}
All the right-indexed critical branchings are confluent, and are drawn in Appendix \ref{A:KLRCriticalBranchings}.

\subsection{Poincar\'e-Birkhoff-Witt bases} 
\label{SS:PBW}
From \cite[Proposition 4.2.15]{AL16}, the set of monomials in normal form between two $1$-cells $\textbf{i}$ and $\textbf{j}$ form a basis of the vector space $\jRi$. In \cite{KL1}, Khovanov and Lauda described a linear basis for the vector space $\jRi$, that we recover here using constructive methods from rewriting. If $\textbf{i}$ and $\textbf{j}$ are in $\text{Seq}(\mathcal{V})$ with $m = | \mathcal{V} |$, it is given by braid diagrams corresponding to a choice of minimal representatives for the Coxeter presentation of $\mathcal{S}_m$, with an arbitrary number of dots at the bottom of each strand. This choice of minimal representatives is obtained using the relations (\ref{ybg_simplylaced}) and (\ref{ybg_simplylaced2}).

\subsubsection{Rouquier's PBW property}
In \cite{ROU08}, Rouquier established that these bases are \emph{Poincar\'e-Birkhoff-Witt} (PBW for short) bases. 
For any $n$ in $\N$, denote by $R_n$ the $\K$-algebra  $\mathbb{K}[x_1, \dots, x_n] \otimes (\mathbb{K}^{(I)})^{\otimes n} := (\mathbb{K}^{(I)}[x])^{\otimes n}$, where $\mathbb{K}^{(I)}$ denotes the set of functions $f: I \fl \mathbb{K}$ such that only a finite number of $f(i)$ for $i$ in $I$ are non-zero.
We denote by $1_s$ the idempotent corresponding to the $s$-th factor of $\mathbb{K}^{(I)}$ and we put $1_\mathcal{V}= 1_{i_1} \otimes \dots \otimes 1_{i_m}$ for $\textbf{i}= i_1 \dots i_m \in \text{Seq}(\mathcal{V})$.
There is a morphism of algebras from $R_m$ to $H_\mathcal{V}(Q)$ sending $x_k 1_{\textbf{i}}$ to $x_{k,\textbf{i}}$.
Let $J$ be a set of finite sequences of elements of $\lbrace 1, \dots, m-1 \rbrace$ such that $ \lbrace s_{i_1} \dots s_{i_r} \rbrace _ {(i_1, \dots, i_r) \in J }$ is a set of minimal length representatives of elements of $\mathcal{S}_m$ for its Coxeter presentation. This choice of minimal representatives in the Coxeter presentation of $\mathcal{S}_m$ is here fixed by the orientation of the relations on three strands.

The algebra $H_\mathcal{V}(Q)$ is enriched with a graduation with $1_{\textbf{i}}$ and $x_{k,\textbf{i}}$ in degree $0$ and $\tau_{k,\textbf{i}}$ in degree $1$, giving a filtration $H_\mathcal{V}(Q)= \bigoplus_{i \in \N}{ \mathcal{F}_i}$ where $\mathcal{F}_i$ is the set of elements in degree $i$. We denote by $\textrm{gr} H_\mathcal{V}(Q)$ the graded algebra corresponding to this filtration. The morphism $R_m  \to H_\mathcal{V}(Q)$ extends to a surjective algebra morphism \begin{equation} \label{pbw2} \mathbb{K}^{(I)}[x] \wr \mathcal{NH}_m \to \textrm{gr} H_\mathcal{V}(Q) 
\end{equation} where $\mathcal{NH}_m$ is the nilHecke algebra of degree $m$ and $\wr$ stands for the wreath product.  In \cite[Theorem 3.7]{ROU08}, Rouquier named the Poincar\'e-Birhoff-Witt property the fact that the morphism (\ref{pbw2}) is an isomorphism. He proved that is equivalent to the fact that the set \[ S= \lbrace \tau_{i_1, s_{i_2} \dots s_{i_r} (\textbf{j})} \dots \tau_{i_r, \textbf{j}} x_{1,\textbf{j}}^{a_1}\dots x_{m,\textbf{j}}^{a_m} \rbrace _{(i_1, \dots, i_r ) \in J, (a_1, \dots, a_m ) \in \N^{m} , \textbf{j} \in \text{Seq}(\mathcal{V})} \] is a linear basis of the algebra $H_{\mathcal{V}}(Q)$.

The multiplication by the $x_{k,\textbf{i}}$ to the right corresponds to adding an arbitrary number of dots at the bottom of each strand in the diagrams. The products
\[ \tau_{i_1, s_{i_2} \dots s_{i_r} (\textbf{j})} \dots \tau_{i_r, \textbf{j}} \] are given in that case by the choices of braid diagrams corresponding to minimal elements in the Coxeter presentation of $\mathcal{S}_m$ for the degree lexicographic order induced by
\[ s_{1} > s_2 > \dots > s_{m-1}. \] As a consequence, for this choice the elements of $S$ correspond to a set of normal forms for $\text{KLR}$, proving the following result:

\begin{corollary}
The simply-laced KLR algebra $R(\mathcal{V})$ admit a PBW basis.
\end{corollary}

\section{Rewriting modulo in Khovanov-Lauda-Rouquier's $2$-category}
\label{S:KLR2Cat}
In this section, we define a linear~$(3,2)$-polygraph presenting Khovanov and Lauda's linear~$(2,2)$-category $\Ug$ and prove that rewriting modulo the isotopy relations using the remaining defining $3$-cells gives a quasi-terminating and confluent modulo linear~$(3,2)$-polygraph. As a consequence, we compute linear bases for the spaces of $2$-cells in $\Ug$ and prove non-degeneracy of Khovanov and Lauda's diagrammatic calculus.

\subsection{The $2$-categories $\Ag$ and $\Ug$}
\label{SS:KLRCategories}
In this subsection, we define the linear~$(2,2)$-categories $\Ag$ and $\Ug$ defined respectively by Khovanov-Lauda and Rouquier and we recall Brundan's isomorphism theorem between these two categories.

\subsubsection{Rouquier's Kac-Moody $2$-category}
Let $\mathbb{K}$ be a field and $(I,\cdot,X,Y)$ be a root datum. The \textit{Kac-Moody $2$-category} $\mathcal{A}(\mathfrak{g})$ defined in \cite{ROU08} is the strict additive $\K$-linear $2$-category whose
\begin{itemize}
\item $0$-cells are given by the elements $\lambda$ in the weight lattice $X$ of the Kac-Moody algebra;
\item generating $1$-cells are given by  
$
E_i \one: \lambda \to \lambda + \alpha_i $ and $
F_i \one : \lambda \to \lambda - \alpha_i;
$
\item generating $2$-morphisms are given by $x_i: E_i \one \to  E_i \one $, $\tau_{i,j} : E_i E_j \one \to E_j E_i \one $, $\eta_i: \one \to F_i E_i \one $ and $ \varepsilon: E_i F_i \one \to \one $ which are represented respectively by the following diagrams:
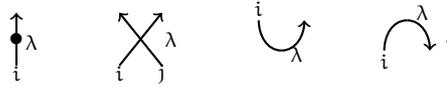
\[ 
\begin{tikzpicture}[baseline = 0]
	\draw[->,thick] (0.08,-.3) to (0.08,.4);
      \node at (0.08,0.05) {$\bullet$};
   \node at (0.08,-.4) {$\scriptstyle{i}$};
   \node at (0.28,0) {$\scriptstyle{\lambda}$};
\end{tikzpicture} \qquad
\begin{tikzpicture}[baseline = 0]
	\draw[->,thick] (0.28,-.3) to (-0.28,.4);
	\draw[->,thick] (-0.28,-.3) to (0.28,.4);
   \node at (-0.28,-.4) {$\scriptstyle{i}$};
   \node at (0.28,-.4) {$\scriptstyle{j}$};
   \node at (.4,.05) {$\scriptstyle{\lambda}$};
\end{tikzpicture}
\qquad
\cupr{i} \qquad \capr{i}  \: .
\]
\end{itemize} 

These two morphisms are subject to the following relations:
\begin{enumerate}[{\bf i)}]
\item The KLR relations for both upward and downward orientations.
\item Right adjunction relations:
\begin{align}\label{rightadj}
\mathord{
\begin{tikzpicture}[baseline = 0]
  \draw[->,thick,black] (0.3,0) to (0.3,.4);
	\draw[-,thick,black] (0.3,0) to[out=-90, in=0] (0.1,-0.4);
	\draw[-,thick,black] (0.1,-0.4) to[out = 180, in = -90] (-0.1,0);
	\draw[-,thick,black] (-0.1,0) to[out=90, in=0] (-0.3,0.4);
	\draw[-,thick,black] (-0.3,0.4) to[out = 180, in =90] (-0.5,0);
  \draw[-,thick,black] (-0.5,0) to (-0.5,-.4);
   \node at (-0.5,-.5) {$\scriptstyle{i}$};
   \node at (0.5,0) {$\scriptstyle{\lambda}$};
\end{tikzpicture}
}
&=
\mathord{\begin{tikzpicture}[baseline=0]
  \draw[->,thick,black] (0,-0.4) to (0,.4);
   \node at (0,-.5) {$\scriptstyle{i}$};
   \node at (0.2,0) {$\scriptstyle{\lambda}$};
\end{tikzpicture}
},\qquad
\mathord{
\begin{tikzpicture}[baseline = 0]
  \draw[->,thick,black] (0.3,0) to (0.3,-.4);
	\draw[-,thick,black] (0.3,0) to[out=90, in=0] (0.1,0.4);
	\draw[-,thick,black] (0.1,0.4) to[out = 180, in = 90] (-0.1,0);
	\draw[-,thick,black] (-0.1,0) to[out=-90, in=0] (-0.3,-0.4);
	\draw[-,thick,black] (-0.3,-0.4) to[out = 180, in =-90] (-0.5,0);
  \draw[-,thick,black] (-0.5,0) to (-0.5,.4);
   \node at (-0.5,.5) {$\scriptstyle{i}$};
   \node at (0.5,0) {$\scriptstyle{\lambda}$};
\end{tikzpicture}
}
=
\mathord{\begin{tikzpicture}[baseline=0]
  \draw[<-,thick,black] (0,-0.4) to (0,.4);
   \node at (0,.5) {$\scriptstyle{i}$};
   \node at (0.2,0) {$\scriptstyle{\lambda}$};
\end{tikzpicture}
},
\end{align}
which imply that $F_i 1_{\lambda+\alpha_i}$ is
the right dual of $E_i 1_\lambda$.
\item Some inversion relations: we require the following $2$-morphisms to be invertible in $\mathcal{A}(\mathfrak{g})$:
\begin{align}\label{inv1}
\trightcross{i}{j}
&:E_j F_i 1_\lambda \stackrel{\sim}{\rightarrow} F_i E_j 1_\lambda
&\text{if $i \neq j$,}\\
\label{inv2}
\trightcross{i}{j}
\oplus
\bigoplus_{n=0}^{\langle h_i,\lambda\rangle-1}
\mathord{
\begin{tikzpicture}[baseline = 0]
	\draw[<-,thick,black] (0.4,0) to[out=90, in=0] (0.1,0.4);
	\draw[-,thick,black] (0.1,0.4) to[out = 180, in = 90] (-0.2,0);
    \node at (-0.2,-.1) {$\scriptstyle{i}$};
  \node at (0.3,0.5) {$\scriptstyle{\lambda}$};
      \node at (-0.3,0.2) {$\color{black}\scriptstyle{n}$};
      \node at (-0.15,0.2) {$\color{black}\bullet$};
\end{tikzpicture}
}
&:
E_i F_i 1_\lambda\stackrel{\sim}{\rightarrow}
F_i E_i 1_\lambda \oplus 1_\lambda^{\oplus \langle h_i,\lambda\rangle}
&\text{if $\langle h_i,\lambda\rangle \geq
  0$},\\
\trightcross{i}{j}
\oplus
\bigoplus_{n=0}^{-\langle h_i,\lambda\rangle-1}
\mathord{
\begin{tikzpicture}[baseline = 0]
	\draw[<-,thick,black] (0.4,0.2) to[out=-90, in=0] (0.1,-.2);
	\draw[-,thick,black] (0.1,-.2) to[out = 180, in = -90] (-0.2,0.2);
    \node at (-0.2,.3) {$\scriptstyle{i}$};
  \node at (0.3,-0.25) {$\scriptstyle{\lambda}$};
      \node at (0.55,0) {$\color{black}\scriptstyle{n}$};
      \node at (0.38,0) {$\color{black}\bullet$};
\end{tikzpicture}
}
&
:E_i F_i 1_\lambda \oplus 
1_\lambda^{\oplus -\langle h_i,\lambda\rangle}
\stackrel{\sim}{\rightarrow}
 F_i E_i 1_\lambda&\text{if $\langle h_i,\lambda\rangle \leq
  0$}.\label{inv3}
\end{align}
\end{enumerate}

This condition of invertibility in $\Ag$ imposes that we have to defined new generating $2$-cells as the formal inverses of each summand in (\ref{inv1}) -- (\ref{inv3}). Let us denote by $\widehat{\Ag}$ the linear~$(2,2)$-category obtained by forgetting the direct sums operations and the grading on $1$-cells in $\Ag$. In order to compute linear bases of $\Ag$, it is sufficient to compute linear bases in the vector spaces of $2$-cells in $\widehat{\Ag}$. 

\subsubsection{Khovanov-Lauda's $2$-category $\Ug$}
The $2$-category defined by Khovanov and Lauda has the same objects and $1$-morphisms than $\Ag$. There are additional generating $2$-morphisms $x':F_i 1_\lambda \rightarrow F_i 1_\lambda$,
$\tau':F_i F_j 1_\lambda \rightarrow F_j F_i 1_\lambda$,
$\eta':1_\lambda \rightarrow E_i F_i 1_\lambda$ and $\eps':F_i E_i
1_\lambda \rightarrow 1_\lambda$
represented diagrammatically by
\begin{align}\label{solid2}
x'
&= 
\mathord{
\begin{tikzpicture}[baseline = 0]
	\draw[<-,thick,black] (0.08,-.3) to (0.08,.4);
      \node at (0.08,0.1) {$\color{black}\bullet$};
   \node at (0.08,.5) {$\scriptstyle{i}$};
\end{tikzpicture}
}
{\scriptstyle\lambda}\:,
\qquad
\tau'
= 
\mathord{
\begin{tikzpicture}[baseline = 0]
	\draw[<-,thick,black] (0.28,-.3) to (-0.28,.4);
	\draw[<-,thick,black] (-0.28,-.3) to (0.28,.4);
   \node at (-0.28,.5) {$\scriptstyle{j}$};
   \node at (0.28,.5) {$\scriptstyle{i}$};
   \node at (.4,.05) {$\scriptstyle{\lambda}$};
\end{tikzpicture}
}\:,
\qquad
\eta'
= 
\mathord{
\begin{tikzpicture}[baseline = 0]
	\draw[-,thick,black] (0.4,0.3) to[out=-90, in=0] (0.1,-0.1);
	\draw[->,thick,black] (0.1,-0.1) to[out = 180, in = -90] (-0.2,0.3);
    \node at (0.4,.4) {$\scriptstyle{i}$};
  \node at (0.3,-0.15) {$\scriptstyle{\lambda}$};
\end{tikzpicture}
}\:,\qquad
\eps'
= 
\mathord{
\begin{tikzpicture}[baseline = 0]
	\draw[-,thick,black] (0.4,-0.1) to[out=90, in=0] (0.1,0.3);
	\draw[->,thick,black] (0.1,0.3) to[out = 180, in = 90] (-0.2,-0.1);
    \node at (0.4,-.2) {$\scriptstyle{i}$};
  \node at (0.3,0.4) {$\scriptstyle{\lambda}$};
\end{tikzpicture}
}.
\end{align} which satisfy many relations as the KLR relations for both upward and downard orientations and the local "$\msl$" relations which come from Lauda's categorification of $\msl$, \cite{LAU12}. We refer to \cite[Section 3.1]{KL3} to see the whole definition of this $2$-category.

\subsubsection{Brundan's isomorphism theorem}
In \cite[Main Theorem]{BRU15}, Brundan defined a $2$-functor from $\Ag$ to $\Ug$ that he proved to be an isomorphism. As these two categories have the same $0$-cells and $1$-cells, this functor is the identity on $0$-cells and $1$-cells. On $2$-cells, it is the identity on the $4$ generating $2$-cells of $\Ag$, which are also in $\Ug$. It then remains to define new $2$-cells $x',\tau', \eta', \varepsilon'$ in $\Ag$ that will be the images of the additionnal generators in $\Ug$ under the inverse functor. We recall here the definition of these new $2$-cells in $\Ag$ and the relations implied by these definitions. First of all, we define the downward dot and crossing as being the right mates of the upward ones:
\begin{equation*} \label{E:DownwardMates}
x' = 
\mathord{
\begin{tikzpicture}[baseline = 0]
	\draw[<-,thick,black] (0.08,-.3) to (0.08,.4);
      \node at (0.08,0.1) {$\color{black}\bullet$};
     \node at (0.08,.55) {$\scriptstyle{i}$};
\end{tikzpicture}
}
{\scriptstyle\lambda}
:=
\mathord{
\begin{tikzpicture}[baseline = 0]
  \draw[->,thick,black] (0.3,0) to (0.3,-.4);
	\draw[-,thick,black] (0.3,0) to[out=90, in=0] (0.1,0.4);
	\draw[-,thick,black] (0.1,0.4) to[out = 180, in = 90] (-0.1,0);
	\draw[-,thick,black] (-0.1,0) to[out=-90, in=0] (-0.3,-0.4);
	\draw[-,thick,black] (-0.3,-0.4) to[out = 180, in =-90] (-0.5,0);
  \draw[-,thick,black] (-0.5,0) to (-0.5,.4);
   \node at (-0.1,0) {$\color{black}\bullet$};
  \node at (-0.5,.55) {$\scriptstyle{i}$};
  \node at (0.5,0) {$\scriptstyle{\lambda}$};
\end{tikzpicture}
},
 \qquad 
\tau' =
\mathord{
\begin{tikzpicture}[baseline = 0]
	\draw[<-,thick,black] (0.28,-.3) to (-0.28,.4);
	\draw[<-,thick,black] (-0.28,-.3) to (0.28,.4);
   \node at (-0.28,.55) {$\scriptstyle{j}$};
   \node at (0.28,.55) {$\scriptstyle{i}$};
   \node at (.4,.05) {$\scriptstyle{\lambda}$};
\end{tikzpicture}
}:=
\crossdalt{i}{j} \: .
\end{equation*} 

In \cite{BRU15}, Brundan defined an additional generator for the isomorphism $2$-cell: 
\begin{equation} \label{E:RightCrossing}
\sigma =  
\mathord{
\begin{tikzpicture}[baseline = 0]
	\draw[<-,thick,black] (0.28,-.3) to (-0.28,.4);
	\draw[->,thick,black] (-0.28,-.3) to (0.28,.4);
   \node at (-0.28,-.45) {$\scriptstyle{j}$};
   \node at (-0.28,.55) {$\scriptstyle{i}$};
   \node at (.4,.05) {$\scriptstyle{\lambda}$};
\end{tikzpicture}
}
:=
\mathord{
\begin{tikzpicture}[baseline = 0]
	\draw[->,thick,black] (0.3,-.5) to (-0.3,.5);
	\draw[-,thick,black] (-0.2,-.2) to (0.2,.3);
        \draw[-,thick,black] (0.2,.3) to[out=50,in=180] (0.5,.5);
        \draw[->,thick,black] (0.5,.5) to[out=0,in=90] (0.8,-.5);
        \draw[-,thick,black] (-0.2,-.2) to[out=230,in=0] (-0.5,-.5);
        \draw[-,thick,black] (-0.5,-.5) to[out=180,in=-90] (-0.8,.5);
  \node at (-0.8,.7) {$\scriptstyle{i}$};
   \node at (0.28,-.7) {$\scriptstyle{j}$};
   \node at (1.05,.05) {$\scriptstyle{\lambda}$};
\end{tikzpicture}
} .
\end{equation} 

He then defined a leftward crossing as the "formal inverse" of this new generator. Using the cyclicity relations proved by Brundan in \cite[Section 5]{BRU15}, $\Ag$ is a pivotal linear~$(2,2)$-category and thus recall from \cite{CKS00}
that the $2$-cells in $\Ag$ are represented up by isotopy. It is then natural to set
\begin{equation}
\sigma' = 
\mathord{
\begin{tikzpicture}[baseline = 0]
	\draw[->,thick,black] (0.28,-.3) to (-0.28,.4);
	\draw[<-,thick,black] (-0.28,-.3) to (0.28,.4);
   \node at (0.28,-.4) {$\scriptstyle{j}$};
   \node at (0.28,.5) {$\scriptstyle{i}$};
   \node at (.4,.05) {$\scriptstyle{\lambda}$};
\end{tikzpicture}
} = \tleftcross{i}{j} :
F_i E_j 1_\lambda \rightarrow E_j F_i 1_\lambda,
\end{equation}

Let us now define the new generators from \cite{BRU15}. In what follows, we have to be careful about the sign of $\hil$ because these definitions will slightly differ. 
\begin{itemize}
\item First, let assume that $\hil \geq 0$. 
We define the $2$-morphisms $\sigma'$ and $\eta'$
so that \begin{equation}\label{comfort}
- \tleftcross{i}{j} \oplus\cdots\oplus
\mathord{
\begin{tikzpicture}[baseline = 0]
	\draw[-,thick,black] (0.4,0.4) to[out=-90, in=0] (0.1,0);
	\draw[->,thick,black] (0.1,0) to[out = 180, in = -90] (-0.2,0.4);
    \node at (0.4,.5) {$\scriptstyle{i}$};
  \node at (0.3,-0.05) {$\scriptstyle{\lambda}$};
\end{tikzpicture}
}
:=
{\bigg(}
\trightcross{i}{i}
\oplus\cdots\oplus
\mathord{
\begin{tikzpicture}[baseline = 0]
	\draw[<-,thick,black] (0.4,0) to[out=90, in=0] (0.1,0.4);
	\draw[-,thick,black] (0.1,0.4) to[out = 180, in = 90] (-0.2,0);
    \node at (-0.2,-.1) {$\scriptstyle{i}$};
  \node at (0.3,0.5) {$\scriptstyle{\lambda}$};
      \node at (-0.8,0.38) {$\color{black}\scriptstyle{\langle h_i,\lambda\rangle}-1$};
      \node at (-0.15,0.2) {$\color{black}\bullet$};
\end{tikzpicture}
}
\bigg)^{-1},
\end{equation}
assuming that $\sigma'$ is just the inverse of $\sigma$ if $\hil = 0$. We also define $$ \capl{i} \: := \: - \tfishulpf{i}{ \hil}. $$

\item Now, let assume that $\hil \leq 0$. We define $\sigma'$ and $\eps'$
so that \begin{equation}\label{comfort2}
- \tleftcross{i}{j} \oplus\cdots\oplus
\:\mathord{
\begin{tikzpicture}[baseline = 0]
	\draw[-,thick,black] (0.4,0) to[out=90, in=0] (0.1,0.4);
	\draw[->,thick,black] (0.1,0.4) to[out = 180, in = 90] (-0.2,0);
    \node at (0.4,-.1) {$\scriptstyle{i}$};
  \node at (0.3,0.5) {$\scriptstyle{\lambda}$};
\end{tikzpicture}
}
:=
\bigg(
\trightcross{i}{i}
\oplus\cdots\oplus
\mathord{
\begin{tikzpicture}[baseline = 0]
	\draw[<-,thick,black] (0.4,0.2) to[out=-90, in=0] (0.1,-.2);
	\draw[-,thick,black] (0.1,-.2) to[out = 180, in = -90] (-0.2,0.2);
    \node at (-0.2,.3) {$\scriptstyle{i}$};
  \node at (0.3,-0.25) {$\scriptstyle{\lambda}$};
      \node at (1.15,0.02) {$\color{black}\scriptstyle{-\langle h_i,\lambda\rangle}-1$};
      \node at (0.38,0) {$\color{black}\bullet$};
\end{tikzpicture}
}
\bigg)^{-1},
\end{equation}
assuming again that $\sigma'$ is just the inverse of $\sigma$ if $\hil = 0$. We set $$ \raisebox{-2mm}{$\cupl{i}$} \: := \: \tfishdlpf{i}{- \hil} \: . $$
\end{itemize}

Using these definitions, Brundan also proved that $F_i 1_{\lambda + \alpha_i}$ also is the right dual of $E_1 \one$, yielding adjunction relations of the form \begin{equation}\label{adjfinal}
\mathord{
\begin{tikzpicture}[baseline = 0]
  \draw[-,thick,black] (0.3,0) to (0.3,-.4);
	\draw[-,thick,black] (0.3,0) to[out=90, in=0] (0.1,0.4);
	\draw[-,thick,black] (0.1,0.4) to[out = 180, in = 90] (-0.1,0);
	\draw[-,thick,black] (-0.1,0) to[out=-90, in=0] (-0.3,-0.4);
	\draw[-,thick,black] (-0.3,-0.4) to[out = 180, in =-90] (-0.5,0);
  \draw[->,thick,black] (-0.5,0) to (-0.5,.4);
   \node at (0.3,-.5) {$\scriptstyle{i}$};
   \node at (0.5,0) {$\scriptstyle{\lambda}$};
\end{tikzpicture}
}
=
\mathord{\begin{tikzpicture}[baseline=0]
  \draw[->,thick,black] (0,-0.4) to (0,.4);
   \node at (0,-.5) {$\scriptstyle{i}$};
   \node at (0.2,0) {$\scriptstyle{\lambda}$};
\end{tikzpicture}
},\qquad
\mathord{
\begin{tikzpicture}[baseline = 0]
  \draw[-,thick,black] (0.3,0) to (0.3,.4);
	\draw[-,thick,black] (0.3,0) to[out=-90, in=0] (0.1,-0.4);
	\draw[-,thick,black] (0.1,-0.4) to[out = 180, in = -90] (-0.1,0);
	\draw[-,thick,black] (-0.1,0) to[out=90, in=0] (-0.3,0.4);
	\draw[-,thick,black] (-0.3,0.4) to[out = 180, in =90] (-0.5,0);
  \draw[->,thick,black] (-0.5,0) to (-0.5,-.4);
   \node at (0.3,.5) {$\scriptstyle{i}$};
   \node at (0.5,0) {$\scriptstyle{\lambda}$};
\end{tikzpicture}
}
=
\mathord{\begin{tikzpicture}[baseline=0]
  \draw[<-,thick,black] (0,-0.4) to (0,.4);
   \node at (0,.5) {$\scriptstyle{i}$};
   \node at (0.2,0) {$\scriptstyle{\lambda}$};
\end{tikzpicture}
}.
\end{equation} where the $2$-cells $\eta'$ and $\varepsilon'$ are units and counits of this left adjunction $F_i 1_{\lambda + \alpha_i} \vdash E_i \one$. As a consequence, this yields some cyclicity relations: 
\begin{align*}
\cuprdl{i}{} = \cuprdr{i}{}, \qquad \caprdl{i}{} = \caprdr{i}{}, \qquad 
\cupldl{i}{}  = \cuprdr{i}{}, \qquad \capldl{i}{} = \capldr{i}{}. 
\end{align*}

\subsubsection{$\mathbb{Z}$-grading}
Following the definitions of Rouquier and Khovanov-Lauda, we define a $\Z$-grading on the $2$-morphisms in $\Ag$, by setting for all $i \in I$:
\[
\textrm{deg}(x_i)  = i \cdot i, \;
\textrm{deg}(\tau_i)  = - i \cdot j, \;
\textrm{deg}(\varepsilon_i)  = \frac{i \cdot i}{2} (1 - \hil), \;
\textrm{deg}(\eta_i)  = \frac{i \cdot i}{2} ( 1 + \hil ). 
\]
With the previous definitions of $x'_i, \tau'_i, \eta'_i$ and $\varepsilon'_i$, we can prove that 
\[
\textrm{deg}(x'_i) = i \cdot i, \;
\textrm{deg}(\tau'_i)  = - i \cdot j, \;
\textrm{deg}(\varepsilon'_i)  = \frac{i \cdot i}{2} (1 - \hil), \;
\textrm{deg}(\eta'_i)  = \frac{i \cdot i}{2} ( 1 + \hil ). 
\]
and that 
\[
\textrm{deg}(\sigma_{i,i}) = 0, \; \;
\textrm{deg}(\sigma'_{i,i}) = 0 \]  for all values of $\hil$, so that this grading exactly to the $\Z$-grading in $\Ug$ defined by Khovanov and Lauda. To compute the degree of a diagrammatic $2$-morphism, it suffices to sum up all the degrees of the generating $2$-morphisms that appear in that diagram. For convenience, we also set $\textrm{deg}(0) = - \infty$.

\subsubsection{Bubbles}
For each $\lambda \in P$, we can define $2$-cells in $\text{END}(1_{1_\lambda})$ by putting a cap over a cup whenever the directions and labels are compatible. Thus, there is two kinds of bubble morphisms, namely clockwise bubbles and counter clockwise bubbles, and we can decorate them by placing an arbitrary number of dots on each:
 $$ \posbubd{i}{n} \hspace{2cm} \negbubd{i}{n}.$$
 
If we compute the degree of such a bubble, we have:
$$ \textrm{deg} \left( \posbubd{i}{n} \right) = i \cdot i ( 1 - \hil + n ) \qquad ; \qquad \textrm{deg} \left( \negbubd{i}{n} \right) = i \cdot i ( 1 + \hil + n ). $$

We impose an additive condition on these bubbles, coming from \cite{LAU12,KL3}: that is bubbles with a negative degree are zero, and that bubbles of degree zero are identities. This corresponds to the following equalities: 
\begin{equation} \label{posbub}
\posbubd{i}{n} =  \left\{
    \begin{array}{ll}
        1_{1_\lambda} & \mbox{if } n = \hil - 1 \\
        0 & \mbox{if } n < \hil - 1 
    \end{array}
\right. 
\end{equation}
\begin{equation} \label{negbub}
\negbubd{i}{n} = \left\{
    \begin{array}{ll}
        1_{1_\lambda} & \mbox{if } n = - \hil - 1 \\
        0 & \mbox{if } n < - \hil - 1 
    \end{array}
\right.
\end{equation}

As in \cite[Section 3.6]{LAU12}, we introduce \textit{fake bubbles}. These bubbles are formal symbols which correspond to bubbles decorated with a negative number of dots. It is explained in \cite{LAU12} that these new symbols are added in order to have a better interpretation with only diagrams of the relations obtained by lifting the relations in $\msl$. They are defined in terms of linear combinations of products of positively dotted bubbles. Following \cite{BRU15}, we set:
for $r, s < 0$:
\begin{align*}
\posbubd{i}{n}
&:=
\left\{
\begin{array}{cl}
- \sum\limits_{k \geq 0}{ \raisebox{-2mm}{$\dbbubf{i}{-n-k-1}{k - \hil}$}} & \text{if $n > \hil - 1$},\\
1_{1_\lambda}&\text{if $n=\langle h_i, \lambda\rangle-1$},\\
0&\text{if $n< \langle h_i,\lambda\rangle-1$},
\end{array}\right. \\
\negbubd{i}{n} & :=
\left\{
\begin{array}{cl}
- \sum\limits_{k \geq 0}{ \raisebox{-2mm}{$\dbbubf{i}{\hil + k}{-n-k-1}$}} & \text{if $n > - \hil - 1$},\\\\
1_{1_\lambda}&\text{if $n=-\langle h_i, \lambda\rangle-1$},\\
0&\text{if $n< -\langle h_i,\lambda\rangle-1$}.
\end{array}\right.
\end{align*}

The first condition for both orientations corresponds to Lauda's inductive definition of fake bubbles coming from the infinite Grassmaniann relation, see \cite[Section 3.6.2]{LAU12}. The second two other definitions impose the same condition that fake bubbles of negative degree are zero, and that fake bubbles of degree zero are identities. With this definition, Brundan proved that the Infinite Grassmaniann relation hold in $\Ag$, that is:
\begin{theorem}[\cite{BRU15}, Theorem 3.2] \label{IG}
For $t >0$, the following relation hold in $\Ag$: 
\begin{align*}\label{ig3}
\sum_{\substack{r,s \in \Z \\ r+s=t-2}}
\raisebox{-2mm}{$\mathord{
\begin{tikzpicture}[baseline = 0]
  \draw[<-,thick,black] (0,0.4) to[out=180,in=90] (-.2,0.2);
  \draw[-,thick,black] (0.2,0.2) to[out=90,in=0] (0,.4);
 \draw[-,thick,black] (-.2,0.2) to[out=-90,in=180] (0,0);
  \draw[-,thick,black] (0,0) to[out=0,in=-90] (0.2,0.2);
 \node at (0,-.1) {$\scriptstyle{i}$};
  \node at (-0.2,0.2) {$\color{black}\bullet$};
   \node at (-0.4,0.2) {$\color{black}\scriptstyle{r}$};
\end{tikzpicture}
}
\mathord{
\begin{tikzpicture}[baseline = 0]
  \draw[->,thick,black] (0.2,0.2) to[out=90,in=0] (0,.4);
  \draw[-,thick,black] (0,0.4) to[out=180,in=90] (-.2,0.2);
\draw[-,thick,black] (-.2,0.2) to[out=-90,in=180] (0,0);
  \draw[-,thick,black] (0,0) to[out=0,in=-90] (0.2,0.2);
 \node at (0,-.1) {$\scriptstyle{i}$};
   \node at (-0.4,0.2) {$\scriptstyle{\lambda}$};
   \node at (0.2,0.2) {$\color{black}\bullet$};
   \node at (0.4,0.2) {$\color{black}\scriptstyle{s}$};
\end{tikzpicture}
}$}
&= 0.
\end{align*}
\end{theorem}

Using the conditions on degrees, we can restrict this relation to the following one:
\begin{equation}
\label{E:InfiniteGrassmannian}
\sum\limits_{k = 0}^{ \alpha} {\raisebox{-2mm}{$\dbbubff{i}{\hil -1 + \alpha - l}{- \hil -1 + l}$}} = 0 \quad \text{for all $\alpha > 0 $.}
\end{equation}

\subsubsection{The relations in $\Ag$}
In this section, we recall some of the important defining relations that arise from the invertibility condition. In \cite{BRU15}, Brundan introduced new generators $\left( \pic{k} \right)_{0 \leq k \leq \hil - 1}$ and $\left( \trefle{k} \right) _{0 \leq k \leq - \hil -1}$ as follows:
\begin{itemize}
\item For $\hil \geq 0$, $\pic{k}$ is the $(n+1)$-th entry of the inverse vector of the invertible $2$-cell when $\hil \geq 0$, that is:  \begin{align}\label{pic}
- \tleftcross{i}{i}
 \oplus\bigoplus_{n=0}^{\langle h_i,\lambda\rangle-1}
\mathord{
\begin{tikzpicture}[baseline = 0]
	\draw[-,thick,black] (0.4,0.4) to[out=-90, in=0] (0.1,0);
	\draw[->,thick,black] (0.1,0) to[out = 180, in = -90] (-0.2,0.4);
    \node at (0.4,.55) {$\scriptstyle{i}$};
  \node at (0.5,0.15) {$\scriptstyle{\lambda}$};
      \node at (0.12,-0.25) {$\color{black}\scriptstyle{n}$};
      \node at (0.12,0.01) {$\color{black}\scriptstyle{\spadesuit}$};
\end{tikzpicture}
}
:=
\bigg(
\trightcross{i}{i}
\oplus
\bigoplus_{n=0}^{\langle h_i,\lambda\rangle-1}
\mathord{
\begin{tikzpicture}[baseline = 0]
	\draw[<-,thick,black] (0.4,0) to[out=90, in=0] (0.1,0.4);
	\draw[-,thick,black] (0.1,0.4) to[out = 180, in = 90] (-0.2,0);
    \node at (-0.2,-.1) {$\scriptstyle{i}$};
  \node at (0.3,0.5) {$\scriptstyle{\lambda}$};
      \node at (-0.35,0.2) {$\color{black}\scriptstyle{n}$};
      \node at (-0.15,0.2) {$\color{black}\bullet$};
\end{tikzpicture}
}
\bigg)^{-1} .
\end{align}

\item Similarly, $\trefle{k}$ is defined for $\hil \leq 0$ by:
\begin{align} -
\tleftcross{i}{j} \oplus\bigoplus_{n=0}^{-\langle h_i,\lambda\rangle-1}
\mathord{
\begin{tikzpicture}[baseline = 0]
	\draw[-,thick,black] (0.4,0) to[out=90, in=0] (0.1,0.4);
	\draw[->,thick,black] (0.1,0.4) to[out = 180, in = 90] (-0.2,0);
    \node at (0.4,-.1) {$\scriptstyle{i}$};
  \node at (0.5,0.45) {$\scriptstyle{\lambda}$};
      \node at (0.12,0.65) {$\color{black}\scriptstyle{n}$};
      \node at (0.12,0.4) {$\color{black}\scriptstyle{\clubsuit}$};
\end{tikzpicture}
}
:=
\bigg(
\trightcross{i}{i}
\oplus
\bigoplus_{n=0}^{-\langle h_i,\lambda\rangle-1}
\mathord{
\begin{tikzpicture}[baseline = 0]
	\draw[<-,thick,black] (0.4,0.4) to[out=-90, in=0] (0.1,0);
	\draw[-,thick,black] (0.1,0) to[out = 180, in = -90] (-0.2,0.4);
    \node at (-0.2,.5) {$\scriptstyle{i}$};
  \node at (0.25,-0.2) {$\scriptstyle{\lambda}$};
      \node at (0.6,0.2) {$\color{black}\scriptstyle{n}$};
      \node at (0.37,0.2) {$\color{black}\bullet$};
\end{tikzpicture}
}
\bigg)^{-1} .
\end{align}
\end{itemize}

To establish the isomorphism between $\Ag$ and $\Ug$, Brundan proved that the following relation have to hold in $\Ag$:
for all $0 \leq n \leq \hil -1$, 
\begin{align}\label{invA}
\mathord{
\begin{tikzpicture}[baseline = 0]
	\draw[-,thick,black] (0.4,0.4) to[out=-90, in=0] (0.1,0);
	\draw[->,thick,black] (0.1,0) to[out = 180, in = -90] (-0.2,0.4);
    \node at (0.4,.55) {$\scriptstyle{i}$};
  \node at (0.65,0.3) {$\scriptstyle{\lambda}$};
      \node at (0.12,-0.25) {$\color{black}\scriptstyle{n}$};
      \node at (0.12,0.01) {$\color{black}\scriptstyle{\spadesuit}$};
\end{tikzpicture}
}
&=
\sum_{r \geq 0}
\mathord{
\begin{tikzpicture}[baseline = 0]
	\draw[-,thick,black] (0.3,0.4) to[out=-90, in=0] (0,0);
	\draw[->,thick,black] (0,0) to[out = 180, in = -90] (-0.3,0.4);
    \node at (0.3,.55) {$\scriptstyle{i}$};
    \node at (0.85,0.5) {$\scriptstyle{\lambda}$};
  \draw[->,thick,black] (0.2,-0.3) to[out=90,in=0] (0,-0.1);
  \draw[-,thick,black] (0,-0.1) to[out=180,in=90] (-.2,-0.3);
\draw[-,thick,black] (-.2,-0.3) to[out=-90,in=180] (0,-0.5);
  \draw[-,thick,black] (0,-0.5) to[out=0,in=-90] (0.2,-0.3);
 \node at (0,-.65) {$\scriptstyle{i}$};
   \node at (0.2,-0.3) {$\color{black}\bullet$};
   \node at (0.9,-0.3) {$\color{black}\scriptstyle{-n-r-2}$};
   \node at (-0.25,0.15) {$\color{black}\bullet$};
   \node at (-0.45,0.15) {$\color{black}\scriptstyle{r}$};
\end{tikzpicture}
}
&&\text{if $0 \leq n < \langle h_i,\lambda\rangle$,}\\
\mathord{
\begin{tikzpicture}[baseline = 0]
	\draw[-,thick,black] (0.4,0) to[out=90, in=0] (0.1,0.4);
	\draw[->,thick,black] (0.1,0.4) to[out = 180, in = 90] (-0.2,0);
    \node at (0.4,-.1) {$\scriptstyle{i}$};
  \node at (0.5,0.45) {$\scriptstyle{\lambda}$};
      \node at (0.12,0.7) {$\color{black}\scriptstyle{n}$};
      \node at (0.12,0.4) {$\color{black}\scriptstyle{\clubsuit}$};
\end{tikzpicture}
}
&=
\sum_{r \geq 0}
\mathord{
\begin{tikzpicture}[baseline=0]
	\draw[-,thick,black] (0.3,-0.5) to[out=90, in=0] (0,-0.1);
	\draw[->,thick,black] (0,-0.1) to[out = 180, in = 90] (-0.3,-.5);
    \node at (0.3,-.6) {$\scriptstyle{i}$};
   \node at (0.25,-0.3) {$\color{black}\bullet$};
   \node at (.4,-.3) {$\color{black}\scriptstyle{r}$};
  \draw[<-,thick,black] (0,0.4) to[out=180,in=90] (-.2,0.2);
  \draw[-,thick,black] (0.2,0.2) to[out=90,in=0] (0,.4);
 \draw[-,thick,black] (-.2,0.2) to[out=-90,in=180] (0,0);
  \draw[-,thick,black] (0,0) to[out=0,in=-90] (0.2,0.2);
 \node at (0,.55) {$\scriptstyle{i}$};
   \node at (0.3,0) {$\scriptstyle{\lambda}$};
   \node at (-0.2,0.2) {$\color{black}\bullet$};
   \node at (-0.9,0.2) {$\color{black}\scriptstyle{-n-r-2}$};
\end{tikzpicture}
}
&&\text{if $0 \leq n < -\langle h_i,\lambda\rangle$.}\label{invB}
\end{align}

As a consequence, we do not have to consider these inverse $2$-cells as generators in the presentation, since we will replace them by their expression in term of the other generators whenever they appear. The invertibility conditions (\ref{inv1}) and (\ref{inv2}) can then be expressed diagrammatically by:
\begin{align}\label{pos}
\tdcrosslr{i}{i} 
&=
\!\!\sum_{n=0}^{\langle h_i, \lambda \rangle-1}
\sum_{r \geq 0}
\mathord{
\begin{tikzpicture}[baseline = 0]
	\draw[-,thick,black] (0.3,0.7) to[out=-90, in=0] (0,0.3);
	\draw[->,thick,black] (0,0.3) to[out = 180, in = -90] (-0.3,0.7);
    \node at (0.3,.85) {$\scriptstyle{i}$};
    \node at (-0.5,0.42) {$\scriptstyle{\lambda}$};
  \draw[->,thick,black] (0.2,0) to[out=90,in=0] (0,0.2);
  \draw[-,thick,black] (0,0.2) to[out=180,in=90] (-.2,0);
\draw[-,thick,black] (-.2,0) to[out=-90,in=180] (0,-0.2);
  \draw[-,thick,black] (0,-0.2) to[out=0,in=-90] (0.2,0);
 \node at (-0.3,0) {$\scriptstyle{i}$};
   \node at (0.2,0) {$\color{black}\bullet$};
   \node at (0.9,0) {$\color{black}\scriptstyle{-n-r-2}$};
   \node at (0.25,0.45) {$\color{black}\bullet$};
   \node at (0.45,0.45) {$\color{black}\scriptstyle{r}$};
	\draw[<-,thick,black] (0.3,-.7) to[out=90, in=0] (0,-0.3);
	\draw[-,thick,black] (0,-0.3) to[out = 180, in = 90] (-0.3,-.7);
    \node at (-0.3,-.85) {$\scriptstyle{i}$};
   \node at (0.25,-0.5) {$\color{black}\bullet$};
   \node at (0.5,-.5) {$\color{black}\scriptstyle{n}$};
\end{tikzpicture}
}
-\mathord{
\begin{tikzpicture}[baseline = 0]
	\draw[<-,thick,black] (0.08,-.3) to (0.08,.4);
	\draw[->,thick,black] (-0.28,-.3) to (-0.28,.4);
   \node at (-0.28,-.4) {$\scriptstyle{i}$};
   \node at (0.08,.5) {$\scriptstyle{i}$};
   \node at (.3,.05) {$\scriptstyle{\lambda}$};
\end{tikzpicture}
},\\\label{neg}
\tdcrossrl{i}{i}
&=
\!\!\sum_{n=0}^{-\langle h_i, \lambda \rangle-1}
\sum_{r \geq 0}
\mathord{
\begin{tikzpicture}[baseline=0]
	\draw[-,thick,black] (0.3,-0.7) to[out=90, in=0] (0,-0.3);
	\draw[->,thick,black] (0,-0.3) to[out = 180, in = 90] (-0.3,-.7);
    \node at (0.3,-.85) {$\scriptstyle{i}$};
   \node at (0.25,-0.5) {$\color{black}\bullet$};
   \node at (.45,-.5) {$\color{black}\scriptstyle{r}$};
  \draw[<-,thick,black] (0,0.2) to[out=180,in=90] (-.2,0);
  \draw[-,thick,black] (0.2,0) to[out=90,in=0] (0,.2);
 \draw[-,thick,black] (-.2,0) to[out=-90,in=180] (0,-0.2);
  \draw[-,thick,black] (0,-0.2) to[out=0,in=-90] (0.2,0);
 \node at (-0.25,0.1) {$\scriptstyle{i}$};
   \node at (-0.4,-0.2) {$\scriptstyle{\lambda}$};
   \node at (0.2,0) {$\color{black}\bullet$};
   \node at (0.9,0) {$\color{black}\scriptstyle{-n-r-2}$};
	\draw[<-,thick,black] (0.3,.7) to[out=-90, in=0] (0,0.3);
	\draw[-,thick,black] (0,0.3) to[out = -180, in = -90] (-0.3,.7);
   \node at (0.27,0.5) {$\color{black}\bullet$};
   \node at (0.5,0.5) {$\color{black}\scriptstyle{n}$};
    \node at (-0.3,.85) {$\scriptstyle{i}$};
\end{tikzpicture}
}
-\mathord{
\begin{tikzpicture}[baseline = 0]
	\draw[->,thick,black] (0.08,-.3) to (0.08,.4);
	\draw[<-,thick,black] (-0.28,-.3) to (-0.28,.4);
   \node at (-0.28,.5) {$\scriptstyle{i}$};
   \node at (0.08,-.4) {$\scriptstyle{i}$};
   \node at (.3,.05) {$\scriptstyle{\lambda}$};
\end{tikzpicture}
}.
\end{align}

Besides, some other relations directly follow from this isomorphism:
\begin{itemize}
\item For $\hil > 0$ and $0 \leq n < \hil$, we have
\begin{equation} \label{lemmepos}
\tfishdr{i}=0, \hspace{2cm} 
\tfishulp{i}{n} =0.
\end{equation}
\item For $\hil < 0$ and $ 0 \leq n < - \hil$, we have
\begin{equation} \label{lemmeneg}
\tfishur{i} =0, \hspace{2cm} 
\tfishdlp{i}{n} =0 .
\end{equation}
\end{itemize}

The following relations also hold, and correspond to the "$\msl$-relations of Khovanov and Lauda's $\Ug$, see \cite[Corollary 3.5]{BRU15}:
\begin{align}\label{everything}
\tfishur{i}
&=
\sum_{n=0}^{\langle h_i, \lambda\rangle}
\!\!\mathord{
\begin{tikzpicture}[baseline=0]
	\draw[<-,thick,black] (0.3,-.45) to[out=90, in=0] (0,-0.05);
	\draw[-,thick,black] (0,-0.05) to[out = 180, in = 90] (-0.3,-.45);
    \node at (-0.3,-.55) {$\scriptstyle{i}$};
   \node at (0.23,-0.25) {$\color{black}\bullet$};
   \node at (0.45,-.25) {$\color{black}\scriptstyle{n}$};
  \draw[-,thick,black] (0,0.45) to[out=180,in=90] (-.2,0.25);
  \draw[->,thick,black] (0.2,0.25) to[out=90,in=0] (0,.45);
 \draw[-,thick,black] (-.2,0.25) to[out=-90,in=180] (0,0.05);
  \draw[-,thick,black] (0,0.05) to[out=0,in=-90] (0.2,0.25);
 \node at (-.25,.45) {$\scriptstyle{i}$};
   \node at (0.45,0.6) {$\scriptstyle{\lambda}$};
   \node at (0.2,0.25) {$\color{black}\bullet$};
   \node at (0.72,0.25) {$\color{black}\scriptstyle{-n-1}$};
\end{tikzpicture}},
&  \tfishdr{i}
&=
-\sum_{n=0}^{-\langle h_i, \lambda\rangle}
\!\!\mathord{
\begin{tikzpicture}[baseline=0]
    \node at (-0.18,-.52) {$\scriptstyle{i}$};
   \node at (0.2,-0.25) {$\color{black}\bullet$};
   \node at (0.7,-.25) {$\color{black}\scriptstyle{-n-1}$};
  \draw[<-,thick,black] (0,-0.05) to[out=180,in=90] (-.2,-0.25);
  \draw[-,thick,black] (0.2,-0.25) to[out=90,in=0] (0,-.05);
 \draw[-,thick,black] (-.2,-0.25) to[out=-90,in=180] (0,-0.45);
  \draw[-,thick,black] (0,-0.45) to[out=0,in=-90] (0.2,-0.25);
 \node at (-.3,.58) {$\scriptstyle{i}$};
   \node at (0.45,-0.6) {$\scriptstyle{\lambda}$};
	\draw[<-,thick,black] (0.3,.45) to[out=-90, in=0] (0,0.05);
	\draw[-,thick,black] (0,0.05) to[out = -180, in = -90] (-0.3,.45);
   \node at (0.27,0.25) {$\color{black}\bullet$};
   \node at (0.52,0.25) {$\color{black}\scriptstyle{n}$};
\end{tikzpicture}}.
\end{align}

\subsubsection{Further relations} We prove some further relations that we will use in the last section to prove that the linear~$(3,2)$-polygraph presenting $\Ag$ is convergent.

\begin{lemma}
\label{L:FurtherRelations}
The following relations hold in $\Ag$: 
\vskip-10pt
\begin{align*}
\tfishul{i} & = - \sum_{n=0}^{- \langle h_i, \lambda\rangle}
\!\!\mathord{
\begin{tikzpicture}[baseline=0]
	\draw[-,thick,black] (0.3,-.45) to[out=90, in=0] (0,-0.05);
	\draw[->,thick,black] (0,-0.05) to[out = 180, in = 90] (-0.3,-.45);
    \node at (-0.3,-.55) {$\scriptstyle{i}$};
   \node at (0.25,-0.25) {$\color{black}\bullet$};
   \node at (0.50,-.25) {$\color{black}\scriptstyle{n}$};
  \draw[-,thick,black] (0,0.45) to[out=180,in=90] (-.2,0.25);
  \draw[-,thick,black] (0.2,0.25) to[out=90,in=0] (0,.45);
 \draw[<-,thick,black] (-.2,0.25) to[out=-90,in=180] (0,0.05);
  \draw[-,thick,black] (0,0.05) to[out=0,in=-90] (0.2,0.25);
 \node at (-.25,.45) {$\scriptstyle{i}$};
   \node at (0.65,-0.5) {$\scriptstyle{\lambda}$};
   \node at (0.2,0.25) {$\color{black}\bullet$};
   \node at (0.7,0.25) {$\color{black}\scriptstyle{-n-1}$};
\end{tikzpicture}},  \qquad \qquad
\tfishdl{i} = \sum_{n=0}^{\langle h_i, \lambda\rangle}
\!\!\mathord{
\begin{tikzpicture}[baseline=0]
    \node at (-0.21,-.52) {$\scriptstyle{i}$};
   \node at (0.2,-0.25) {$\color{black}\bullet$};
   \node at (0.7,-.25) {$\color{black}\scriptstyle{-n-1}$};
  \draw[<-,thick,black] (0,-0.05) to[out=180,in=90] (-.2,-0.25);
  \draw[-,thick,black] (0.2,-0.25) to[out=90,in=0] (0,-.05);
 \draw[-,thick,black] (-.2,-0.25) to[out=-90,in=180] (0,-0.45);
  \draw[-,thick,black] (0,-0.45) to[out=0,in=-90] (0.2,-0.25);
 \node at (-.3,.55) {$\scriptstyle{i}$};
   \node at (0.55,-0.68) {$\scriptstyle{\lambda}$};
	\draw[-,thick,black] (0.3,.45) to[out=-90, in=0] (0,0.05);
	\draw[->,thick,black] (0,0.05) to[out = -180, in = -90] (-0.3,.45);
   \node at (0.27,0.25) {$\color{black}\bullet$};
   \node at (0.5,0.25) {$\color{black}\scriptstyle{n}$};
\end{tikzpicture}}.
\end{align*}
\end{lemma}

\begin{proof} Using the symmetry in $\Ag$ coming from the anti-involution $\T$ defined by Brundan in \cite[Theorem 2.3]{BRU15}, it suffices to prove the first relation. For $\hil > 0$, it follows directly from the relations (\ref{lemmepos}). For $\hil = 0$, the left handside is equal to $ - \capl{i}$ using the definition of $\varepsilon_i$ when $\hil \geq 0$. The right handside also reduces to $ - \capl{i} $ because the bubble that remains is an identity, using the degree conditions.
Let us prove it for $\hil < 0$. 
On the one hand, using the relation of invertibility, we have 
\begin{align*}
\tdcrossrld{i}{i} & = \sum\limits_{n=0}^{- \hil -1} 
\sum\limits_{r \geq 0}{\stda{i}{n+1}{-n-r-2}{r}} - \diddownupdl{i}{i} \underset{(\ref{posbub})}{=}  \sum\limits_{n=0}^{- \hil -1} 
\sum\limits_{r=0}^{- \hil -1}{\stda{i}{n+1}{-n-r-2}{r}} -\diddownupdl{i}{i} \\ & =  \sum\limits_{n=1}^{- \hil} \sum\limits_{r=0}^{- \hil -1}{\stda{i}{n}{-n-r-1}{r}} - \diddownupdl{i}{i}  = \sum\limits_{n=1}^{- \hil} \sum\limits_{r=0}^{- \hil}{\stda{i}{n}{-n-r-1}{r}} - \diddownupdl{i}{i}
\end{align*}
The last equality above is due to the fact that $ \sum\limits_{n=1}^{- \hil} {\stdahil{i}{n}{-n +\hil -1}{- \hil}} = 0$ since $n > 0$, using $(\ref{posbub})$. On the other hand, we can make the dot go down using the upward KLR relations:
\begin{align*} 
\tdcrossrld{i}{i} & = \tdcrossrldd{i}{i} - \raisebox{2mm}{$\acapl{i}$}  + \raisebox{3mm}{$\cuprb{i}$} \underset{(\ref{neg})}{=}  \sum\limits_{n=0}^{- \hil - 1} \sum\limits_{r \geq 0}{\stda{i}{n}{-n-r-2}{r+1}} - \diddownupdl{i}{i} - \acapl{i} + \raisebox{3mm}{$\cuprb{i}$} \\ & \underset{(\ref{everything})}{=} \sum\limits_{n=0}^{- \hil - 1} \sum\limits_{r \geq 0}{\stda{i}{n}{-n-r-2}{r+1}}  + \sum\limits_{n=0}^{- \hil} \lemmec{i}{n} + \raisebox{3mm}{$\cuprb{i}$} - \diddownupdl{i}{i} \\
&  \underset{(\ref{posbub})}{=} \sum\limits_{n=0}^{- \hil - 1} \sum\limits_{r = 0}^{- \hil -1}{\stda{i}{n}{-n-r-2}{r+1}} + \sum\limits_{n=0}^{- \hil} \lemmec{i}{n} + \raisebox{3mm}{$\cuprb{i}$} - \diddownupdl{i}{i} \\ & \underset{(\ast)}{=}  \sum\limits_{n=0}^{- \hil } \sum\limits_{r = 0}^{- \hil -1}{\stda{i}{n}{-n-r-2}{r+1}} + \sum\limits_{n=0}^{- \hil} \lemmec{i}{n} + \raisebox{3mm}{$\cuprb{i}$} - \diddownupdl{i}{i} \\ 
& = \sum\limits_{n=0}^{- \hil } \sum\limits_{r = 0}^{- \hil }{\stda{i}{n}{-n-r-1}{r}} + \raisebox{3mm}{$\cuprb{i}$} - \diddownupdl{i}{i},
\end{align*}
where the $(\ast)$ equality is due to the fact the term in $-\hil$ in the first summand is zero by the degree conditions.
Thus, the two expressions obtained have to be equal, and so we must have
\begin{align*}
\sum\limits_{r=0}^{- \hil} \raisebox{3mm}{$\lemmee{i}{r}$} + \raisebox{3mm}{$\cuprb{i}$} = 0.
\end{align*} 
Using the bilinearity of the vertical composition in the linear~$(2,2)$-category $\Ag$,we obtain the result.
\end{proof}

\subsection{The linear~$(3,2)$-polygraph $\mathcal{KLR}$}
\label{SS:LinPolKLRb}
We recall following \cite{BRU15,KL3} a presentation of the linear~$(2,2)$-category $\widehat{\Ag}$ by a linear~$(3,2)$-polygraph, which we will prove quasi-terminating and confluent modulo its subpolygraph of isotopies.
\begin{definition}
Let $\mathcal{KLR}$ be the linear~$(3,2)$-polygraph defined by:
\begin{enumerate}[{\bf i)}]
\item the elements of $\mathcal{KLR}_0$ are the weights $\lambda \in X$ of the Kac-Moody algebra;
\item the elements of $\mathcal{KLR}_1$ are given by
\[ 1_{\lambda '} \mathcal{E}_{\varepsilon_1 i_1} \dots \mathcal{E}_{\varepsilon_m i_m} 1_{\lambda} \] for any signed sequence of vertices $(\varepsilon_1 i_1, \dots, \varepsilon_m i_m)$ in $\text{SSeq} := \coprod\limits_{\mathcal{V} \in \mathbb{N}[I]} \text{SSeq}(\mathcal{V})$, and $\lambda$,$\lambda '$ in $X$. Such a $1$-cell has for $0$-source $\lambda$ and $0$-target $\lambda '$, and 
\[ 1_{\lambda ''} \mathcal{E}_{\varepsilon'_1 j_1} \dots \mathcal{E}_{\varepsilon_l j_l} 1_{\lambda'} \star_0 1_{\lambda '} \mathcal{E}_{\varepsilon_1 i_1} \dots \mathcal{E}_{\varepsilon_m i_m} 1_{\lambda} = 1_{\lambda ''} \mathcal{E}_{\varepsilon'_1 j_1} \dots \mathcal{E}_{\varepsilon_m i_m} 1_{\lambda} \]
\item the elements of $\mathcal{KLR}_2$ are the following generating $2$-cells: for any $i$ in $I$ and $\lambda '$ in $X$,
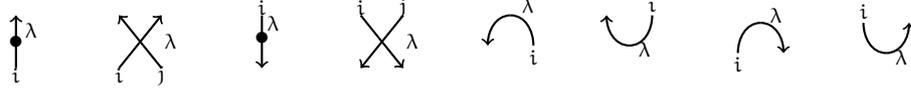
\begin{align*}
\udott{i} \qquad \crossup{i}{j} \qquad \ddott{i} \qquad \crossdn{i}{j} \qquad
\capl{i} \qquad  \cupl{i} \qquad \capr{i} \qquad \cupr{i}
\end{align*}
\item $\mathcal{KLR}_3$ consists of the following $3$-cells:
\begin{itemize}
\item[1)] The $3$-cells of the linear~$(3,2)$-polygraph $\textrm{KLR}$ for both upward and downward orientations of all strands. For any $3$-cell $\delta$ in $\text{KLR}_3$, we denote by $\delta^{\lambda,+}$ (resp. $\delta^{\lambda,-}$) the corresponding $3$-cell in $\KLRb$ with upward (resp. downward) oriented strands and the rightmost region of the diagram being labelled by $\lambda$.  
\item[2)] The isotopy $3$-cells: for any $i \in I$ and $\lambda \in X$
\begin{align}
\label{E:IsotopyCells}
& \mathord{
\begin{tikzpicture}[baseline = 0]
  \draw[->,thick,black] (0.3,0) to (0.3,.4);
	\draw[-,thick,black] (0.3,0) to[out=-90, in=0] (0.1,-0.4);
	\draw[-,thick,black] (0.1,-0.4) to[out = 180, in = -90] (-0.1,0);
	\draw[-,thick,black] (-0.1,0) to[out=90, in=0] (-0.3,0.4);
	\draw[-,thick,black] (-0.3,0.4) to[out = 180, in =90] (-0.5,0);
  \draw[-,thick,black] (-0.5,0) to (-0.5,-.4);
   \node at (-0.5,-.6) {$\scriptstyle{i}$};
   \node at (0.5,0) {$\scriptstyle{\lambda}$};
\end{tikzpicture}
} \overset{i_1^0}{\Rrightarrow}
\mathord{\begin{tikzpicture}[baseline=0]
  \draw[->,thick,black] (0,-0.4) to (0,.4);
   \node at (0,-.6) {$\scriptstyle{i}$};
   \node at (0.2,0) {$\scriptstyle{\lambda}$};
\end{tikzpicture}
}, \quad
\mathord{
\begin{tikzpicture}[baseline = 0]
  \draw[->,thick,black] (0.3,0) to (0.3,-.4);
	\draw[-,thick,black] (0.3,0) to[out=90, in=0] (0.1,0.4);
	\draw[-,thick,black] (0.1,0.4) to[out = 180, in = 90] (-0.1,0);
	\draw[-,thick,black] (-0.1,0) to[out=-90, in=0] (-0.3,-0.4);
	\draw[-,thick,black] (-0.3,-0.4) to[out = 180, in =-90] (-0.5,0);
  \draw[-,thick,black] (-0.5,0) to (-0.5,.4);
   \node at (-0.5,.6) {$\scriptstyle{i}$};
   \node at (0.5,0) {$\scriptstyle{\lambda}$};
\end{tikzpicture}
}
\overset{i_3^0}{\Rrightarrow}
\mathord{\begin{tikzpicture}[baseline=0]
  \draw[<-,thick,black] (0,-0.4) to (0,.4);
   \node at (0,.6) {$\scriptstyle{i}$};
   \node at (0.2,0) {$\scriptstyle{\lambda}$};
\end{tikzpicture}
}, \quad 
\mathord{
\begin{tikzpicture}[baseline = 0]
  \draw[-,thick,black] (0.3,0) to (0.3,-.4);
	\draw[-,thick,black] (0.3,0) to[out=90, in=0] (0.1,0.4);
	\draw[-,thick,black] (0.1,0.4) to[out = 180, in = 90] (-0.1,0);
	\draw[-,thick,black] (-0.1,0) to[out=-90, in=0] (-0.3,-0.4);
	\draw[-,thick,black] (-0.3,-0.4) to[out = 180, in =-90] (-0.5,0);
  \draw[->,thick,black] (-0.5,0) to (-0.5,.4);
   \node at (0.3,-.6) {$\scriptstyle{i}$};
   \node at (0.5,0) {$\scriptstyle{\lambda}$};
\end{tikzpicture}
}
\overset{i_4^0}{\Rrightarrow}
\mathord{\begin{tikzpicture}[baseline=0]
  \draw[->,thick,black] (0,-0.4) to (0,.4);
   \node at (0,-.6) {$\scriptstyle{i}$};
   \node at (0.2,0) {$\scriptstyle{\lambda}$};
\end{tikzpicture}
}, \quad
\mathord{
\begin{tikzpicture}[baseline = 0]
  \draw[-,thick,black] (0.3,0) to (0.3,.4);
	\draw[-,thick,black] (0.3,0) to[out=-90, in=0] (0.1,-0.4);
	\draw[-,thick,black] (0.1,-0.4) to[out = 180, in = -90] (-0.1,0);
	\draw[-,thick,black] (-0.1,0) to[out=90, in=0] (-0.3,0.4);
	\draw[-,thick,black] (-0.3,0.4) to[out = 180, in =90] (-0.5,0);
  \draw[->,thick,black] (-0.5,0) to (-0.5,-.4);
   \node at (0.3,.6) {$\scriptstyle{i}$};
   \node at (0.5,0) {$\scriptstyle{\lambda}$};
\end{tikzpicture}
}
\overset{i_2^0}{\Rrightarrow}
\mathord{\begin{tikzpicture}[baseline=0]
  \draw[<-,thick,black] (0,-0.4) to (0,.4);
   \node at (0,.6) {$\scriptstyle{i}$};
   \node at (0.2,0) {$\scriptstyle{\lambda}$};
\end{tikzpicture}
},
\end{align}
\begin{align}
& \suld{i} \overset{i_2^1}{\Rrightarrow} \dpd{i}, \quad \surdd{i} \overset{i_1^1}{\Rrightarrow} \upd{i}, \quad \sdrd{i} \overset{i_3^1}{\Rrightarrow} \dpd{i}, \quad \sdld{i} \overset{i_4^1}{\Rrightarrow} \upd{i},
\end{align}
\begin{align}
\label{E:IsotopyCells2}
 & \cuprdl{i}{}  \overset{i_1^2}{\Rrightarrow} \cuprdr{i}{}, \quad \caprdl{i}{} \overset{i_3^2}{\Rrightarrow} \caprdr{i}{}, \quad 
\cupldl{i}{} \overset{i_2^2}{\Rrightarrow} \cupldr{i}{}, \quad \capldl{i}{} \overset{i_4^2}{\Rrightarrow} \capldr{i}{} 
\end{align}

\item[3)] The $3$-cells coming from the new generators in $\Ag$: for any $i,j \in I$, $\lambda \in X$
\begin{equation} 
\label{E:NewGenerators}
\tfishdlpf{i}{- \hil } \overset{D_{i,\lambda}^{-}}{\Rrightarrow} \cupl{i} \quad \text{for $\hil \leq 0$}, \hspace{0.5cm}
\tfishulpf{i}{ \hil } \overset{B_{i,\lambda}^{+}}{\Rrightarrow} - \capl{i} \quad \text{for $\hil \geq 0$}
\end{equation}
\item[4)] The $3$-cells for the degree conditions on bubbles: for every $i \in I$, $\lambda \in X$
\begin{equation}
\label{E:ReductionPosBub}
\posbubd{i}{n} \underset{b_{i,\lambda}^{0,n}}{\overset{b_{i,\lambda}^1}{\Rrightarrow}}  \left\{
    \begin{array}{ll}
        1_{1_\lambda} & \mbox{if } n = \hil - 1 \\
        0 & \mbox{if } n < \hil - 1 
    \end{array}
\right. 
\end{equation}
\begin{equation}
\label{E:ReductionNegBub} 
\negbubd{i}{n} \underset{c_{i,\lambda}^{0,n}}{\overset{c_{i,\lambda}^1}{\Rrightarrow}} \left\{
    \begin{array}{ll}
        1_{1_\lambda} & \mbox{if } n = - \hil - 1 \\
        0 & \mbox{if } n < - \hil - 1 
    \end{array}
\right.
\end{equation}
\item[5)] The Infinite-Grassmannian $3$-cells: for any $i \in I$, $\lambda \in X$ and $\alpha > 0$,
\begin{align*}
\posbubdff{i}{\hil - 1 + \alpha } \overset{\text{ig}_\alpha}{\Rrightarrow} - \sum\limits_{l=1}^{ \alpha } {\dbbubff{i}{\hil - 1 + \alpha - l}{- \hil -1 +l}}
\end{align*}
\item[6)] Bubble-slide $3$-cells: for any $i,j$ in $I$ and any $\alpha \geq 0$,
 \begin{eqnarray*} \label{klr1} 
\hspace{-1.5cm}  \raisebox{-2mm}{$\posbubdfffsl{i}{ \hilx{j} - 1 + \alpha}$} \identu{j}
 & \overset{s_{i,j,\lambda,\alpha}^+}{\Rrightarrow} & \left\lbrace
\begin{array}{ccc}
  \sum\limits_{f=0}^{\alpha} (\alpha +1 -f) \identdotsufsl{i}{\alpha -f} \posbubdffr{i}{\hil -1 + f} \;  & \qquad & \text{if $i=j$, } \\ \\
  \identusl{j} \posbubdffr{i}{\hil -1 + \alpha}
  & &
 \text{if $i \cdot j=0$, }
  \\    \\
 \identdotusl{j} \posbubdffr{i}{- \hil + \alpha -2} + 
  \quad + \quad
    \identusl{j} \posbubdffr{i}{\hil -1 + \alpha}
 & &
 \text{if $i \cdot j=-1$. }
\end{array}
\right.
\end{eqnarray*}

and for any $i,j$ in $I$ and any $\alpha \geq 0$,
 \begin{eqnarray*} \label{klr1} 
\hspace{-2.5cm}  \raisebox{-2mm}{$\negbubdfffsl{i}{ - \hilx{j} - 1 + \alpha}$} \identu{j}
 & \overset{s_{i,j,\lambda,\alpha}^-}{\Rrightarrow} & \left\lbrace
\begin{array}{cc}
  \identdotsusl{i}{2} \negbubdfff{i}{-\hil + \alpha -3} - 2 \identdotusl{i} \negbubdfff{i}{- \hil  + \alpha -2} + \identusl{i} \negbubdfff{i}{- \hil -1 + \alpha}    & \hspace{-0.4cm} \text{if $i=j$, }
  \\    \\
 \identusl{i} \negbubdfff{i}{- \hil -1 + \alpha} & \hspace{-0.3cm} \text{if $i \cdot j = 0$. }
  \\     \\
\sum\limits_{f=0}^{\alpha} (-1)^f \identdotsusl{j}{f} \negbubdfff{i}{- \hil -1 + \alpha -f}
 & \hspace{-0.3cm}
 \text{if $i \cdot j=-1$. }
\end{array}
\right.
\end{eqnarray*}
so as their reflections $r_{i,j,\lambda, \alpha}^+$ and $r_{i,j,\lambda, \alpha}^-$ through a horizontal axis, allowing to make a bubble go through a downward strand. These reflexions correspond to the images of these relations via the symmetry $\widetilde{\psi}$ defined by Khovanov and Lauda in \cite[Section 3.3]{KL3}. Note that these relations were originally proved by Khovanov and Lauda in \cite[Propositions 3.3 $\&$ 3.4]{KL3}, and are added to this presentation to reach confluence modulo as it will be explained later.
\item[7)] The invertibility $3$-cells: for any $i,j \in I$ and $\lambda \in X$
\begin{align*}
\tdcrosslr{i}{j} & \overset{F_{i,j,\lambda}}{\Rrightarrow}  \mathord{
\begin{tikzpicture}[baseline = 0]
	\draw[<-,thick,black] (0.08,-.3) to (0.08,.4);
	\draw[->,thick,black] (-0.28,-.3) to (-0.28,.4);
   \node at (-0.28,-.4) {$\scriptstyle{i}$};
   \node at (0.08,.5) {$\scriptstyle{j}$};
   \node at (.3,.05) {$\scriptstyle{\lambda}$};
\end{tikzpicture}
}, \qquad
\tdcrossrl{i}{j}  \overset{E_{i,j,\lambda}}{\Rrightarrow} \mathord{
\begin{tikzpicture}[baseline = 0]
	\draw[->,thick,black] (0.08,-.3) to (0.08,.4);
	\draw[<-,thick,black] (-0.28,-.3) to (-0.28,.4);
   \node at (-0.28,.5) {$\scriptstyle{i}$};
   \node at (0.08,-.4) {$\scriptstyle{j}$};
   \node at (.3,.05) {$\scriptstyle{\lambda}$};
\end{tikzpicture}
} \\
\tdcrosslr{i}{i} 
& \overset{F_{i,\lambda}}{\Rrightarrow} -\mathord{
\begin{tikzpicture}[baseline = 0]
	\draw[<-,thick,black] (0.08,-.3) to (0.08,.4);
	\draw[->,thick,black] (-0.28,-.3) to (-0.28,.4);
   \node at (-0.28,-.4) {$\scriptstyle{i}$};
   \node at (0.08,.5) {$\scriptstyle{i}$};
   \node at (.3,.05) {$\scriptstyle{\lambda}$};
\end{tikzpicture}
} +
\!\!\sum_{n=0}^{\langle h_i, \lambda \rangle-1}
\sum_{r \geq 0}
\stdb{i}{r}{-n-r-2}{n}, \\
\tdcrossrl{i}{i}
&\overset{E_{i,\lambda}}{\Rrightarrow} -\mathord{
\begin{tikzpicture}[baseline = 0]
	\draw[->,thick,black] (0.08,-.3) to (0.08,.4);
	\draw[<-,thick,black] (-0.28,-.3) to (-0.28,.4);
   \node at (-0.28,.5) {$\scriptstyle{i}$};
   \node at (0.08,-.4) {$\scriptstyle{i}$};
   \node at (.3,.05) {$\scriptstyle{\lambda}$};
\end{tikzpicture}
} + 
\!\!\sum_{n=0}^{-\langle h_i, \lambda \rangle-1}
\sum_{r \geq 0}
\stda{i}{n}{-n-r-2}{r}.
\end{align*}
\item[8)] The $3$-cells corresponding to the $\msl$ relations: for any $i \in I$ and $\lambda \in X$
\begin{align*}
\tfishur{i}
&\overset{C_{i,\lambda}}{\Rrightarrow}
\sum_{n=0}^{\langle h_i, \lambda\rangle}
\rulec{i}{n}; \qquad
  \tfishdr{i}
\overset{A_{i,\lambda}}{\Rrightarrow}
-\sum_{n=0}^{-\langle h_i, \lambda\rangle}
\rulea{i}{n}; \\
\tfishul{i} & \overset{B_{i,\lambda}}{\Rrightarrow} - \sum_{n=0}^{- \langle h_i, \lambda\rangle}
\ruleb{i}{n}; \qquad
\tfishdl{i}  \overset{D_{i,\lambda}}{\Rrightarrow} \sum_{n=0}^{\langle h_i, \lambda\rangle}
\ruled{i}{n}.
\end{align*}
\end{itemize}
\end{enumerate}
\end{definition}

\begin{remark}
The $3$-cells defining the new caps and cups generators in $3)$ are redundant in this presentation since they can be recovered using the $\mathfrak{sl}_2$ relations of $8)$, the degree condition relations on bubbles of $4)$ and the KLR relations of $1)$: for instance, we have the following rewriting sequence in $\mathcal{KLR}$: for $\hil > 0$,
$$ \tfishulpf{i}{\hil} \Rrightarrow \mathord{
\begin{tikzpicture}[baseline = 0,scale=0.7]
\draw[<-,thick,black] (0.8,0.5) to[out=90, in=0] (0.5,0.9);
	\draw[-,thick,black] (0.5,0.9) to[out = 180, in = 90] (0.2,0.5);

	\draw[-,thick,black] (0.2,.5) to (-0.3,-.5);
	\draw[-,thick,black] (-0.2,.2) to (0.2,-.3);
        \draw[-,thick,black] (0.2,-.3) to[out=130,in=180] (0.5,-.5);
        \draw[-,thick,black] (0.5,-.5) to[out=0,in=270] (0.8,.5);
        \draw[-,thick,black] (-0.2,.2) to[out=130,in=0] (-0.5,.5);
        \draw[->,thick,black] (-0.5,.5) to[out=180,in=-270] (-0.8,-.5);
        \node at (1,0.5) {$\scriptstyle{\lambda}$};
         \node at (0.3,0.8) {$\color{black}\bullet$};
   \node at (-0.5,0.85) {$\scriptstyle{\hil}$};
\end{tikzpicture}}   \Rrightarrow  \mathord{
\begin{tikzpicture}[baseline = 0,scale=0.7]
\draw[<-,thick,black] (0.8,0.5) to[out=90, in=0] (0.5,0.9);
	\draw[-,thick,black] (0.5,0.9) to[out = 180, in = 90] (0.2,0.5);

	\draw[-,thick,black] (0.2,.5) to (-0.3,-.5);
	\draw[-,thick,black] (-0.2,.2) to (0.2,-.3);
        \draw[-,thick,black] (0.2,-.3) to[out=130,in=180] (0.5,-.5);
        \draw[-,thick,black] (0.5,-.5) to[out=0,in=270] (0.8,.5);
        \draw[-,thick,black] (-0.2,.2) to[out=130,in=0] (-0.5,.5);
        \draw[->,thick,black] (-0.5,.5) to[out=180,in=-270] (-0.8,-.5);
        \node at (1,0.5) {$\scriptstyle{\lambda}$};
         \node at (-0.2,-0.4) {$\color{black}\bullet$};
   \node at (-0.3,-0.7) {$\scriptstyle{\hil}$};
\end{tikzpicture}} - \sum\limits_{a+b= \hil -1} \mathord{
\begin{tikzpicture}[baseline=0]
	\draw[-,thick,black] (0.3,-.45) to[out=90, in=0] (0,-0.05);
	\draw[->,thick,black] (0,-0.05) to[out = 180, in = 90] (-0.3,-.45);
    \node at (-0.3,-.55) {$\scriptstyle{i}$};
   \node at (0.25,-0.25) {$\color{black}\bullet$};
   \node at (0.4,-.25) {$\color{black}\scriptstyle{a}$};
  \draw[<-,thick,black] (0,0.45) to[out=180,in=90] (-.2,0.25);
  \draw[-,thick,black] (0.2,0.25) to[out=90,in=0] (0,.45);
 \draw[-,thick,black] (-.2,0.25) to[out=-90,in=180] (0,0.05);
  \draw[-,thick,black] (0,0.05) to[out=0,in=-90] (0.2,0.25);
 \node at (-.25,.45) {$\scriptstyle{i}$};
   \node at (0.45,0) {$\scriptstyle{\lambda}$};
   \node at (-0.2,0.25) {$\color{black}\bullet$};
   \node at (-0.4,0.25) {$\color{black}\scriptstyle{b}$};
\end{tikzpicture}} \Rrightarrow 0 - \capl{i}    $$

Similarly, one proves that the relations (\ref{lemmepos}) - (\ref{lemmeneg}) can be recovered with this presentation, so the corresponding $3$-cells can be removed from the presentation. We still denote by $\KLRb$ the linear~$(3,2)$-polygraph defined as above but with the $3$-cells of $3)$ removed.
\end{remark}

Following \cite{KL3,BRU15}, the $3$-cells in $\KLRb$ are sufficient to recover all the relations in $\Ag$, so that we have the following result:
\begin{proposition}
The linear~$(3,2)$-polygraph $\KLRb$ presents the linear~$(2,2)$-category $\widehat{\Ag}$.
\end{proposition}

\subsubsection{Convergent splitting of $\mathcal{KLR}$}
\label{SSS:SplittingKLR}
We define a convergent splitting $(E,R)$ of the linear~$(3,2)$-polygraph $\mathcal{KLR}$ as follows: the linear~$(3,2)$-polygraph $E$ has the same $0$-cells and $1$-cells than $\mathcal{KLR}$, its generating $2$-cells are given by the six following $2$-cells 
\[ \udott{i} \qquad \ddott{i} \qquad
\capl{i} \qquad  \cupl{i} \qquad \capr{i} \qquad \cupr{i} \]
and the $3$-cells of $E$ are the isotopy $3$-cells of $\mathcal{KLR}$ given in (\ref{E:IsotopyCells}) -- (\ref{E:IsotopyCells2}). Note that following \cite{DUP19}, the linear~$(3,2)$-polygraph $E$ is convergent. The linear $(3,2)$-polygraph $R$ is then defined by $R_i = \mathcal{KLR}_i$ for $0 \leq i \leq 2$ and $R_3 = \mathcal{KLR}_3 \backslash E_3$. In the sequel, we will consider rewriting with respect to the linear~$(3,2)$-polygraph $S:= \ER$, and we will prove the following result:
\begin{theorem}
\label{T:Quasi-ConvModKLR}
The linear~$(3,2)$-polygraph modulo $(R,E,\ER)$ is quasi-terminating and confluent modulo $E$.
\end{theorem}

\subsection{Quasi-termination of $\ER$}
\label{SS:QuasiTermKLR}

\subsubsection{Quasi-reduced monomials}
\label{SSS:Quasi-RedMonomials}
Recall that Alleaume enlighted in \cite{AL16} that linear~$(2,2)$-categories admitting relations making bubbles go through strands cannot be equiped with a monomial order, and thus can not be presented by terminating but rather quasi-terminating rewriting system. This is the case in this setting because of the bubble slide relations creating rewriting cycles, as for instance:
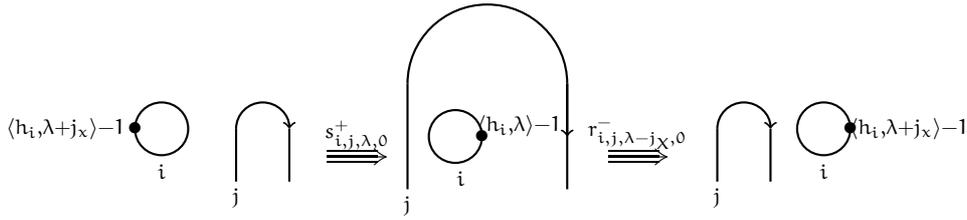
\[ 
\xymatrix@C=2em{ \raisebox{-7mm}{$\bubcapb{j}{i}{\la h_i, \lambda + j_x \ra -1}$} \quad \ar@3 [r] ^-{s_{i,j,\lambda,0}^+} & \raisebox{-8mm}{$\nestcapb{j}{i}{\hil -1}$} \quad \ar@3 [r] ^-{r_{i,j,\lambda - j_X,0}^-} & \quad \raisebox{-7mm}{$\capbubb{j}{i}{\la h_i, \lambda + j_x \ra -1}$} }
\] 
for any $i$ and $j$ such that $i \cdot j = 0$, and where the last equality is due to the exchange relation (\ref{E:ExchangeRel}) of the linear~$(2,2)$-category $\Ag$. Note that there are the same kind of cyclic rewriting sequences in $\KLRb$ for different labels $i$ and $j$, different orientations of bubbles and different number of dots decorating them. There also are the same kind of relations with caps replaced by cups, these relations are not drawn here.

 However, following \cite{AL16}, we say that a monomial in $\Ag$ is \emph{quasi-reduced} if we can only apply to it one of the rewriting sequences given above. 

\begin{remark}
Note that rewriting with respect to the linear~$(3,2)$-polygraph modulo $\ER$ brings additional loops coming from indexed diagrams of the form
\begin{equation} \label{E:CycleIndexIso}
\isoindexa{i}, \qquad \isoindexap{i}.
\end{equation}
using the dot move $3$-cells $i_j^2$ for $1 \leq j \leq 4$, where $k$ is a $2$-cell in $R_2^\ast$. Note that when $k$ is a $2$-cell built of a $\star_0$ and $\star_1$ composite of dots, cups and caps $2$-cells, the diagram in (\ref{E:CycleIndexIso}) is irreducible by $R$, and thus by $\ERE$. When $k$ is built with crossings, one checks that there there are cycles of the following form:
\begin{equation}
\label{E:PivotalCycle}
 \isocbadot{i}{j} \raisebox{-5mm}{$\tfl$} \isocbadotb{i}{j} \raisebox{-5mm}{$- \delta_{i,j} \isocbres{i}$} \raisebox{-5mm}{$\equiv_E$} \isocbadotc{i}{j} \! \raisebox{-5mm}{$- \delta_{i,j} \isocbres{i}$} \raisebox{-5mm}{$\tfl$} \isocbadot{i}{j}\! \raisebox{-5mm}{$- \delta_{i,j} \isocbres{i} + \delta_{i,j} \isocbres{i} $} 
\end{equation}
and from the same diagram closed on its right by a rightward cap and a leftward cup. Similarly, if for $k \geq0$ we denote by \[ \isocbak{i}{j} \]
the diagram obtained as the superposition of $2k$ composable crossings, closed on the left using a cap and a cup, there are cycles in $\ER$ given by:
\[ \isocbakdota{i}{j} \raisebox{-5mm}{$\tfl$}  \isocbakdota{i}{j}  \]
and similarly for a superposition of $2k$ upward oriented crossings closed on its right by a rightward cap and a leftward cup, and for downward oriented crossings. However, one can always exit the cycles of the form (\ref{E:PivotalCycle}) using the $3$-cells $\beta_{i,j}^+$ or $\beta_{i,j}^-$ when the dot is not inside a double crossing, so that we do not take these cycles into account when considering quasi-reduced monomials.
\end{remark}

\subsubsection{Termination without bubble slide $3$-cells}
\label{SSS:TermWithoutBubbleSlide}

Before proving that $\ER$ is quasi-terminating, let us at first prove the following result stating that, without the bubble slide $3$-cells, the linear~$(3,2)$-polygraph $R$ given above is terminating. 

\begin{lemma}
\label{L:TermWithoutBubbleSLide}
	The linear~$(3,2)$-polygraph $R' = \left( R_0,R_1,R_2,R_3 \; \backslash \{ s_{i,j,\lambda}^+,s_{i,j,\lambda}^- \} \right)$ is terminating.
\end{lemma}

\begin{proof}
We proceed into three steps. 
\begin{enumerate}[{\bf i)}]
\item At first, let us extend the derivation $d$ defined in \ref{SSS:TerminationKLR} by keeping the same value on crossings and dots, no matter the orientation of strands, and by setting the value on caps and cups $2$-cells as $0$. Using this derivation, we get that $d( s_2(\delta) ) > d( t_2(\delta) )$ for any $3$-cell $\delta$ coming from the linear~$(3,2)$-polygraph $\text{KLR}$. As a consequence, one gets that if the linear~$(3,2)$-polygraph $R''$ defined as $R'$ minus every KLR $3$-cell terminates, then so does $R'$. Indeed, otherwise there would exist an infinite reduction sequence $(f_n)_{n \in \N}$ in $R'$ and thus, an infinite decreasing sequence $(d(f_n))_{n \in \N}$ of natural numbers. Moreover, this sequence would be strictly decreasing at each step that is generated by any $\text{KLR}$ $3$-cell. Thus, after some natural number $p$, this sequence would be generated by the other $3$-cells only. This would yield an infinite reduction sequence $(f_n)_{n \geq p}$ in $R''$, which is impossible by assumption.
\item Let us prove that $R''$ is terminating in the two remaining steps. First of all, let us consider the derivation $|| \cdot ||_{\{ \tau_{i,j}^+ , \tau_{i,j}^- \}_{i,j \in I}}$ into the trivial modulo $M_{\ast, \ast, \Z}$, counting the number of crossing generators in a given $2$-cell. Then for any $3$-cell $\delta$ belonging to $\{ \Ail, \Bil, \Cil, \Dil, E_{i,j,\lambda}, F_{i,j,\lambda} \}$, we get that $d( s_2 ( \delta)) > d(t_2 (\delta))$, and we prove in a same way that if the linear~$(3,2)$-polygraph $R'''$ defined as $R$ with only all $3$-cells implying bubbles as $3$-cells is terminating, then so is $R'$.
\item To prove that $R'''$ is terminating, we consider the derivation $d'$ into the trivial module $M_{\ast, \ast, \Z}$ defined for any $2$-cell $u$ in $\KLRb_2$ by
\[ d'(u) = \left\{
    \begin{array}{lll}
       \# \lbrace \text{bubbles in $u$} \rbrace + \sum\limits_{\text{$\pi$ clockwise oriented bubble in $u$}} \text{deg}(\pi) & \mbox{if } \text{$u$ contains bubbles}, \\
        0 & \mbox{if $u$ has no bubbles}, \\
        - \infty & \text{if $u=0$}.
    \end{array}
\right. \]
One then easily checks that
\[ d'( s_2(b_{i,\lambda}^1)) = d'( s_2 ( b_{i,\lambda}^{0,n})) = 1 + 2 (1 - \hil + n) > 0 = \text{max} \left(  d'( t_2(b_{i,\lambda}^1)), d'( t_2 ( b_{i,\lambda}^{0,n})) \right) \]
\[ d'( s_2(c_{i,\lambda}^1)) = d'( s_2 ( c_{i,\lambda}^{0,n})) = 1 > 0 = \text{max} \left(  d'( t_2(c_{i,\lambda}^1)), d'( t_2 ( c_{i,\lambda}^{0,n})) \right) \]
\[ d' (s_2 (\text{ig}_{\alpha} )) = d' \left(  \posbubdff{}{\hil -1 + \alpha} \right) = 1 + \alpha \; i \cdot i > 2 + (\alpha - l) i \cdot i = d' \left( \dbbubff{}{\hil -1 +\alpha -l}{- \hil -1 + l} \right) \]
since $l \geq 1$ and $i \cdot i = 2$.
\end{enumerate}
\end{proof}

\subsubsection{Weight functions}
\label{SSS:WeightFunctions}
Let $\mathcal{C}$ a linear~$(2,2)$-category. Recall from \cite{AL16} that a \emph{weight function} on $\mathcal{C}$ is a function $\tau$ from $\mathcal{C}_2$ to $\N$ such that
\begin{enumerate}[{\bf i)}]
\item $\tau (u \star_i v) = \tau (u) + \tau (v)$ for $i=0,1$ for any $i$-composable $2$-cells $u$ and $v$,
\item for each $2$-cell $u$ in $\mathcal{C}_2$,
\[ \tau (u) = \text{max} \{ \tau(u_i) \; | \; u_i \in \text{Supp}(u) \} \]
\end{enumerate}
Note that when $\mathcal{C}$ is presented by a linear~$(3,2)$-polygraph $P$, such a weight function is uniquely and entirely determined by its values on the generating $2$-cells of $P_2$. This enables to define a \emph{quasi-ordering} $\qord$ on $\KLRb_2^\ell$, that is a transitive and reflexive binary relation $\qord$ on $\KLRb_2^\ell$ \cite{DER87}, by $ u \qord v$ if $\tau(u) \geq \tau(v)$.

We define a weight function on $\KLRb_2^\ell$ by its following values on generating $2$-cells:
\vspace{-0.5cm}
\[ \tau (\: \raisebox{-1mm}{$\caplsl{}$} ) = \tau (\caprsl{} \: ) = \tau (  \cuprsl{} \: ) = \tau (\: \raisebox{-1mm}{$\cuplsl{}$} ) = 0, \; \tau ( \raisebox{2mm}{$\udottsl{}$} ) = \tau ( \raisebox{2mm}{$\ddottsl{}$} ) = 0, \;
\tau ( \raisebox{-3mm}{$\scalebox{0.6}{\crossingup{}{}}$} ) = \tau ( \raisebox{-3mm}{$\scalebox{0.6}{\crossingdn{}{}}$} ) = 3.
 \]
 Note that for any $3$-cell $\alpha$ in $E_3$, we have $\tau (s_2 (\alpha)) = \tau (t_2 (\alpha))$ so that the isotopy $3$-cells preserve this weight function. Then, starting with a monomial $u$ of $\KLRb_2^\ell$:
 \begin{itemize}
 \item[-] While $u$ can be rewritten with respect to $\ER$ into a $2$-cell $u'$ such that
 $\tau(u') < \tau(u)$, then assign $u$ to $u'$.
\item[-] While $u$ can be rewritten with respect to $\ER$ into a $2$-cell $u'$ without any of the $3$-cells depicted in \ref{SSS:Quasi-RedMonomials}, then assign $u$ to $u'$.
 \end{itemize}
 
From Lemma \ref{L:TermWithoutBubbleSLide} and well-foundedness of the quasi-ordering $\qord$, this procedure terminates and returns a linear combination of monomials in $\KLRb_2^\ell$ which are quasi-reduced, proving that $\ER$ is quasi-terminating.

\subsection{Confluence modulo}
\label{SS:ConfluenceModuloKLR}
We prove that $\ER$ is confluent modulo $E$ by proving that it is decreasing modulo $E$, following \cite[Theorem 2.3.8]{DUP19}. To prove that it is decreasing, we use \cite[Proposition 2.4.4]{DUP19}, and prove that all critical branchings of the form $(f,g)$, where $f$ is a positive $3$-cell in $\So$ and $g$ is a positive $3$-cell in $\Ro$ are decreasingly confluent for the quasi-normal form labelling $\psiqnf$. Let us provide an exhaustive list of such critical branchings, and prove that these are all confluent modulo $E$, and decreasing with respect to $\psiqnf$. Let us at first notice that the branching implying $3$-cells $b_{i,\lambda}^{k,n}$, $b_{i,\lambda}^{k,n}$ and $I_{\alpha}$ for $k=0,1$ and $\alpha > 0$ are trivially confluent by definition of bubbles with a negative number of dots and the Infinite Grassmanian relation. Notice also that the bubble slide $3$-cells does not overlap with the other $3$-cells implying bubbles since their sources are bubbles with positive degrees by definition. Let us now study the remaining critical branchings, that we split into two sets: those implying the KLR $3$-cells and the remaining branchings between $3$-cells $A_{i,\lambda}$-$F_{i,\lambda}$.

\subsubsection{Critical branchings from KLR relations}
First of all, we have to consider all the the critical branchings of the linear~$(3,2)$-polygraph $\text{KLR}$ presenting the KLR algebra for both downward and upward orientation of strands. These are all confluent from \ref{SSS:CriticalBranchingsKLRAlgebra} and Appendix \ref{A:KLRCriticalBranchings}.
 These $3$-cells provide the following critical branchings of $\ER$ modulo $E$:
\[
(A_{i,\lambda}, \alpha_{i,\lambda}^{L,+}),
\; (B_{i,\lambda}, i_4^2 \cdot \alpha_{i,\lambda}^{L,+}), \;
(C_{i,\lambda}, i_3^2, \alpha_{i,\lambda}^{R,+}), \;
(D_{i,\lambda}, \alpha_{i,\lambda}^{R,+}), \;
(E_{i,\lambda}, \alpha_{i,\lambda}^{L,+}), \;
(F_{i,\lambda}, \alpha_{i,\lambda}^{R,+}).
\]
for any value of $\hil$, of respective sources 
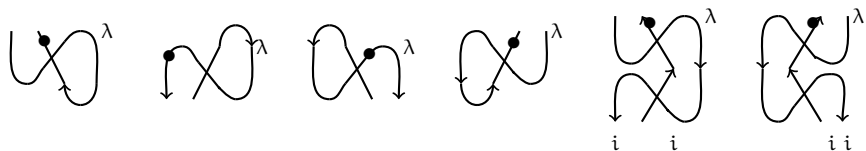
\[ \raisebox{-2mm}{$\cbadot{i}$} \quad \raisebox{-4mm}{$\cbbdot{i}$} \quad \raisebox{-4mm}{$\cbcdot{i}$} \quad \raisebox{-2mm}{$\cbddot{i}$} \quad \cbedot{i}{i} \quad \cbfdot{i}{i}  \]
There are also critical branchings coming from isotopy given by 
\[ (\beta_{i,j}^{\lambda,+}, (i_1^0 \star_2 i_4^0)^- \cdot F_{i,j,\lambda}), \quad (\alpha_{i,\lambda}^{R,+}, (i_1^0 \star_2 i_4^0)^- \star_2 i_3^2 \star_2 i_1^2 \cdot F_{i,j,\lambda}), \quad
	( \gamma_{j,i,j}^{\lambda,+} , (i_1^0 \star_2 i_4^0)^- \cdot F_{i,j,\lambda})
\]
of respective sources 
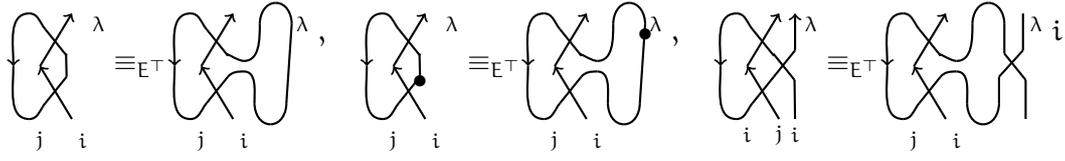
\[ \isocba{i}{j} \raisebox{-4mm}{$\equiv_{\tck{E}}$} \isocbb{i}{j}, \quad \isocbadot{i}{j} \raisebox{-4mm}{$\equiv_{\tck{E}}$} \isocbbdot{i}{j}, \quad \isocbacr{i}{j}{i} \raisebox{-4mm}{$\equiv_{\tck{E}}$} \isocbbcr{i}{j}{i}   \]
Similarly, there are critical branchings of the form
\[ (\beta_{i,j}^{\lambda,+} , (i_1^0 \star_2 i_4^0)^- \cdot E_{i,j,\lambda}), \quad (\alpha_{i,\lambda}^{L,+}, (i_1^0 \star_2 i_4^0)^- \star_2 (i_2^2 \star_2 i_4^2)^- \cdot E_{i,j,\lambda}).
\]
All these branchings are proved confluent modulo $E$ with respect to $\ER$ in \ref{A:BranchingsDots}. Besides, it is clear that each rewriting step drawn in the confluence diagrams in \ref{A:BranchingsDots} make the distance to a quasi-normal form decrease by $1$, proving decreasing confluence of these critical branchings for $\psiqnf$.

\subsubsection{Critical branchings between $3$-cells $A-F$}
\label{SSS:ClassificationOfBranchingsAF}
 Let us now classify critical branchings between the $3$-cells $A_{i,\lambda}$, $B_{i,\lambda}$, $C_{i,\lambda}$. We denote at first that if $i,j \in I$ with $i \ne j$, there are two critical branchings given by $(E_{i,j,\lambda},F_{i,j,\lambda})$ and $(F_{i,j,\lambda},E_{i,j,\lambda})$ which are trivially confluent.
When both strands are labelled by the same vertex $i$, the $3$-cells $E_{i,\lambda}$ and $F_{i,\lambda}$ overlap with the $\msl$ $3$-cells, and we describe below a way to list these overlappings, depending on the notion of \emph{type} of a $2$-cell.

\begin{definition}
For any $2$-cell $u$ in $\mathcal{KLR}_2$, we define the \textit{type} of $u$ as follows:
\begin{enumerate}[{\bf i)}]
\item If $u$ has a $1$-source (resp. $1$-target) $\mathcal{E}$ and an identity $1$-cell as target (resp. source), that is if $u$ is represented by a closed diagram at its top (resp. at its bottom), we set the type of $D$ to be $$\text{sgn}(\mathcal{E})^d \quad (\text{resp. $\text{sgn}(\mathcal{E})^u$}),$$
where $\text{sgn}(\mathcal{E})$ depicts the sequence of signs appearing in $\mathcal{E}$.
\item If $u$ is a $2$-cell in $\mathcal{KLR}_2$ between two non-identity $1$-cells, then the type of $u$ is given by two elements $\text{sgn}(\mathcal{E})^d$ and $\text{sgn}(\mathcal{F})^u$
\end{enumerate}
For instance, the following diagrams have respectively for type $(+,-)^d$ and $(-,+)^d,(-,+)^u:= (-,+)^{u,d}$:
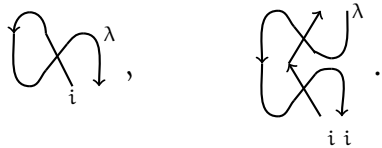
\begin{align*}
\raisebox{2mm}{$\tfishur{i}$}, \qquad \qquad \raisebox{5mm}{$\tdcrosslr{i}{i}$} .
\end{align*}
\end{definition}

Moreover, all the $3$-cells named by a letter $A$ have the same type $(-,+)^u$, we thus call it type $A$. We do the same thing for the other $3$-cells and we recover the different types for our $3$-cells in an array: \\
\begin{center}
\begin{tabular}{|c|c|}
\hline 
Type of the $3$-cell & Type of the diagram \\
\hline
$A$ & $(-,+)^u$ \\
\hline
$B$ & $(-,+)^d$ \\
\hline
$C$ & $(+,-)^d$ \\ 
\hline
$D$  & $(+,-)^u$ \\
\hline 
$E$ & $(-,+)^{d,u}$ \\
\hline
$F$ & $(+,-)^{d,u}$ \\
\hline
\end{tabular}
\end{center}

There is a critical branching between two such relations if and only if they overlap on an element $\tleftcross{}{}$ or $\trightcross{}{}$. Thus, we can notice that there is a branching only between $3$-cells of opposed type, that is in which we reverse all the signs and we change the orientation. For instance, there is a branching between $A$ and $C$ whose source is:
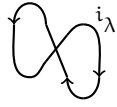
\begin{align*}
\bcritac{i}
\end{align*}

\noindent Following this observation, the pairs of $3$-cells that lead to a critical branching are:
\[ (\Cil, \Ail), \: (\Fil, \Ail), \: (\Bil, \Dil), \: (\Bil, \Fil), \: (\Cil, \Eil), \: (\Eil, \Dil), \: (\Eil, \Fil), \: (\Fil, \Eil) \]
for any $i$ in $I$, any $\lambda$ in $X$ and any possible value of $\hil$.
We check that all these critical branchings are confluent modulo $E$, all the drawings are given in the appendix \ref{A:IsoCriticalBranchings}.

\subsection{Categorification of quantum groups}
\label{SS:CategorificationQuantumGroups}
In this section, we prove using rewriting that the generating set that Khovanov and Lauda expected to be a basis is actually a basis, relating this generating set to a set of quasi-normal forms for the linear~$(3,2)$-polygraph $\ER$ defined from $\KLRb$. As an immediate consequence of the results of \cite{KL3}, we obtain that the linear~ $(2,2)$-category $\Ug$ is a categorification of the quantum group $\dot{\mathbf{U}}_q (\mathfrak{g})$ associated with a symmetrizable Kac-Moody algebra $\mathfrak{g}$ whose Dynkin diagram $\Gamma$ is a simply-laced graph.

\subsubsection{Khovanov-Lauda's generating set}
In \cite{KL3}, Khovanov and Lauda described a general generating set for the vector space $\Ug (E_{\textbf{i}} \one , E_{\textbf{j}} \one)$, for any $\textbf{i}$ and $\textbf{j}$ in $\SSv$, and $\lambda$ in $X$. To define this set, consider $m$ points (resp. $n$ points) on the lower (resp. upper) boundary $\R \times \{ 0 \}$ (resp. $\R \times \{ 1 \}$) of the planar strip $\R \times [0,1]$, with $m+n$ even, and choose an immersion of $\frac{n+m}{2}$ strands into the strip $\R \times [0,1]$ having these points as endpoints. We say that a strand is a \emph{through strand} if it links an endpoint of $R \times \{ 0 \}$ to an endpoint of $\R \times \{ 1 \}$. We fix and orientation and a label for each of this strands, so that any endpoint inherit a label from the strand he is linked to, and a sign which is $+$ if the strand is upward oriented when reaching the endpoint, $-$ otherwise. These orientations and labels on the upper (resp. the lower) endpoints then define signed sequences $\textbf{i}$ and $\textbf{j}$ in $\SSv$. These immersions between $\textbf{i}$ and $\textbf{j}$ are defined modulo boundary-preserving homotopies, and are called \emph{$(\textbf{i}, \textbf{j})$-pairings}. We will consider \emph{minimal} $(\textbf{i},\textbf{j})$-pairings, that is such pairings in which strands have no self-intersections and any two strands intersect at most once. 

Any $(\textbf{i},\textbf{j})$-pairing has a minimal diagram, and we denote by $p (\textbf{i},\textbf{j})$ a set of fixed minimal $(i,j)$-pairing $\tilde{u}$ for any $(\textbf{i},\textbf{j})$-pairing $u$. Let us also denote by $\Pi_{\lambda}$ the set of $2$-cells $\Ug (\one, \one)$ containing all products of clockwise and counterclockwise oriented bubbles with exterior region labelled by $\lambda$, having an arbitrary number of dots on it and such that the degree of each bubble is positive. Following \cite{KL3}, we consider a set $\mathcal{B}_{\mathbf{i},\mathbf{j},\lambda}$ consisting of the union, over all $u$ in $p (\textbf{i} , \textbf{j})$, of diagrams built out of $u$ by fixing a choice of an interval on each strand, away from the intersections, and placing an arbitrary number of dots on each of this intervals, and placing any diagram representing a monomial in $\Pi_{\lambda}$ to the right of this new diagram. Khovanov and Lauda proved that this space spans the $\mathbb{K}$-vector space $\Ug (E_{\textbf{i}} \one , E_{\textbf{j}} \one)$.

\subsubsection{Monomials in quasi-normal form}
\label{SSS:MonQuasiNF}
In this section, we will fix a particular set of monomials in quasi-normal forms for the linear~$(3,2)$-polygraph $\ER$. Before defining this set, let us expand a few remarks on reductions of $2$-cells using rewriting modulo with respect to $\ER$, allowing to change a diagram up to isotopy to apply $3$-cells of $\KLRb$.
\begin{enumerate}[{\bf a)}]
\item Note that a $2$-cell $u$ can be reduced into a linear combination of diagrams in which all $2$-cells have positive degree, using the infinite Grassmannian $3$-cell and the degree condition $3$-cells.
\item A $2$-cell $u$ containing bubbles can be reduced into a linear combination of $2$-cells $u'$ in which all the bubbles moved to the rightmost region using the bubble slide relations.
\item If a $2$-cell $u$ contains a strand that intersect twice with another strand, one can use isotopies and $3$-cells $\Eil$, $\Fil$ or $\beta_{i,j,\lambda}^{\pm}$ to remove these intersections. As a consequence, two different strands can intersect at most once.
\item If a $2$-cell contains a non through strand that intersect with itself, one can use isotopies and $3$-cells $\Ail$ (or $\Bil, \Cil, \Dil$) on the part of the diagram next to the intersection to remove this intersection. 
\item If a $2$-cell contains a through strand with dots on it, the dots can be moved to the bottom of the strand using the KLR $3$-cells $\alpha_{i,\lambda}^{L, \pm}$.
\item If a $2$-cell contains a non through strand with a dot on it, and this strand does not intersect with another strand, the dot can be placed anywhere. Taking the normal form will respect to $E$ will then make the dot move to the right.
\item If this non-through strand intersect with another strand, we are in one of the following situations:
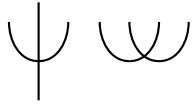
\[ \cupstrand \quad \cupcup \]
or the mirror image of it through the anti-involution $\T$, for any orientation and labels on strands. In the first case, if the dot is placed on the left of the cup, it can be moved to the right using isotopy and the $3$-cell $\alpha_{i,j,\lambda}^{L,\pm}$. In the second situation, if the dot is placed on the leftmost cup (resp. on the rightmost cup), it can be reduced with the KLR $3$-cell $\alpha_{i,j}^{L,\pm}$ (or just an identity if the dot is already in the good position) in 
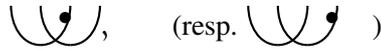
\[ \cupcupdota, \qquad \text{(resp. $\cupcupdotb$ \; )} \]
\end{enumerate}

As a consequence, one can choose a set of $E$-normal forms of quasi-normal forms with respect to $\ER$ containing $2$-cells in $\KLRb_2$ having: all bubbles placed in the rightmost region and all dots placed to the right of a bubble, a minimal number of crossings and crossings moved as far as possible to the right using the Yang-Baxter $3$-cells $\gamma_{i,j,\lambda}^{\pm}$, no strands with self-intersection and no double intersections between two different strands, dots placed on the bottom on every through strand and on the rightmost part of every non-through strand. This choice of set of quasi-normal forms correspond to a particular set $\mathcal{B}_{\mathbf{i},\mathbf{j},\lambda}$ of Khovanov and Lauda. As a consequence of \cite[Theorem 2.5.6]{DUP19}, we get the following result:
\begin{theorem}
\label{T:BasisKLCategory}
The set $\mathcal{B}_{i,j,\lambda}$ defined above is a linear basis of $\Ug (E_{\mathbf{i}} \one, E_{\mathbf{j}} \one)$.
\end{theorem}

\subsubsection{Categorification of quantum groups}
In \cite{KL3}, Khovanov and Lauda defined a map $\gamma$ between Lusztig's idempotent and integral form $\dot{\mathbf{U}} (\mathfrak{g})$ defined in \cite{LUS10} of the quantum group $\mathbf{U}_q(\mathfrak{g})$ associated with a symmetrizable Kac-Moody algebra and the Grothendieck group of the (additive) linear~$(2,2)$-category $\Ug$. They established that this map is surjective for any Kac-Moody algebra $\mathfrak{g}$ and any field $\K$. However, the injectivity of $\gamma$ holds if and only if the graphical calculus they introduce is non-degenerate, which is equivalent to the fact that the generating set $\mathcal{B}_{\mathbf{i},\mathbf{j},\lambda}$ is a linear basis of the $\mathbb{K}$-vector space of $2$-cells $\Ug (E_{\mathbf{i}} \one, E_{\mathbf{j}} \one)$ for any $\mathbf{i}$ and $\mathbf{j}$ in $\SSv$. From Theorem \ref{T:BasisKLCategory}, this is true for any Kac-Moody algebra $\mathfrak{g}$ defined from a simply-laced Cartan datum, namely for any Kac-Moody algebra having a simply-laced Dynkin Diagram, so we obtain as a corollary the following result:
\begin{corollary}
For a Kac-Moody algebra $\mathfrak{g}$ defined by a simply-laced Cartan datum, the linear~$(2,2)$-category $\Ug$ is a categorification of $\dot{\mathbf{U}} (\mathfrak{g})$.
\end{corollary}

%

\begin{small}
\renewcommand{\refname}{\Large\textsc{References}}
\bibliographystyle{plain}
\bibliography{Bibliography}
\end{small}

\clearpage

\csname @addtoreset\endcsname{section}{part} 
\csname @addtoreset\endcsname{subsection}{part}

\appendix
\setcounter{secnumdepth}{2}

\newgeometry{top=2.3cm, bottom=2.2cm, left=2.5cm , right=2.5cm}

\section{Critical branchings for the linear~$(3,2)$-polygraph $\textrm{KLR}$}
\label{A:KLRCriticalBranchings}
In this section, we will draw all the diagram corresponding to the given list of critical branchings for the linear~$(3,2)$-polygraph KLR.

\begin{itemize}
	\item[A)] \textbf{Crossings with two dots}
\end{itemize}
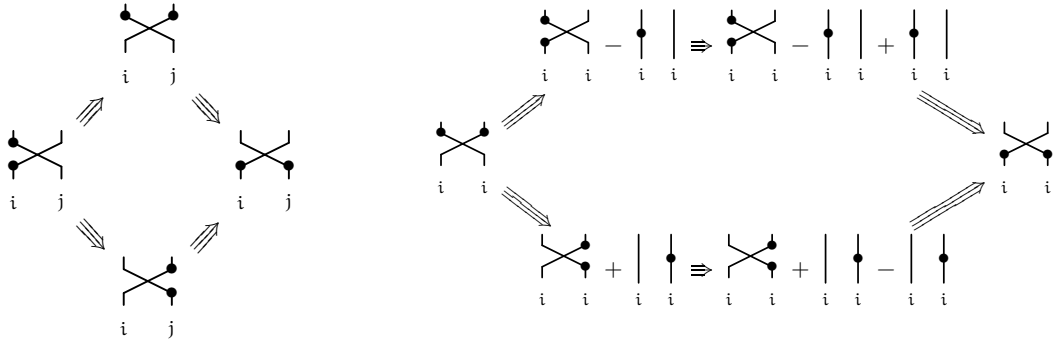
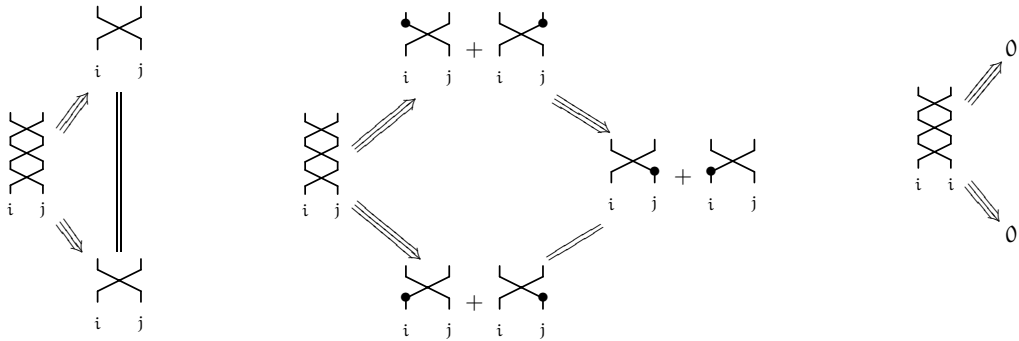
\[
\scalebox{0.9}{\xymatrix@-3ex@R=0.8em{ & \raisebox{-6mm}{$\ducross{i}{j}$} \ar@3[dr] &               \\
		\raisebox{-6mm}{$\dcrossudl{i}{j}$} \ar@3[ur] \ar@3[dr] &  &  \raisebox{-6mm}{$\ddcross{i}{j}$} \\
		&  \raisebox{-6mm}{$\dcrossudr{i}{j}$}  \ar@3[ur] & } \hspace{1.5cm} \scalebox{0.9}{\xymatrix@-3ex@R=0.8em{ & \raisebox{-6mm}{$\dcrossudl{i}{i}$} - \raisebox{-6mm}{$\didl{i}{i}$} \ar@3[r] & \raisebox{-6mm}{$\dcrossudl{i}{i}$} - \raisebox{-6mm}{$\didl{i}{i}$} + \raisebox{-6mm}{$\didl{i}{i}$} \ar@3[dr] &               \\
		\raisebox{-6mm}{$\ducross{i}{i}$} \ar@3[ur] \ar@3[dr] & & &  \raisebox{-6mm}{$\ddcross{i}{i}$} \\
		&   \raisebox{-6mm}{$\dcrossudr{i}{i}$} + \raisebox{-6mm}{$\didr{i}{i}$} \ar@3[r] &  \raisebox{-6mm}{$\dcrossudr{i}{i}$} + \raisebox{-6mm}{$\didr{i}{i}$} - \raisebox{-6mm}{$\didr{i}{i}$} \ar@3[ur] &  }}} \]

\begin{itemize}
	\item[B)] \textbf{Triple crossings}
\end{itemize}
\vskip-5pt
\[ \scalebox{0.811}{\xymatrix@-3ex@R=0.7em{ & \raisebox{-5mm}{$\dcrossnodot{i}{j}$} \ar@{=}[dd] \\
		\raisebox{-7mm}{$\triplecross{i}{j}$} \ar@3[ur] \ar@3[dr] &  \\
		& \raisebox{-5mm}{$\dcrossnodot{i}{j}$} } \hspace{2cm}
	\xymatrix@-2ex@R=0.7em{  & \raisebox{-6mm}{$\dcrossul{i}{j}$} + \raisebox{-6mm}{$\dcrossur{i}{j}$} \ar@3[dr] &  \\
		\raisebox{-7mm}{$\triplecross{i}{j}$} \ar@3[ur] \ar@3[dr] & & \raisebox{-6mm}{$\dcrossdr{i}{j}$} + \raisebox{-6mm}{$\dcrossdl{i}{j}$} \\
		& \raisebox{-6mm}{$\dcrossdl{i}{j}$} + \raisebox{-6mm}{$\dcrossdr{i}{j}$} \ar@{=}[ur] &  } \hspace{2cm}
	\xymatrix@-2ex@R=0.7em{ & 0  \\
		\raisebox{-8mm}{$\triplecross{i}{i}$} \ar@3[ur] \ar@3[dr] & \\    & 0 }} \]
respectively when $i \cdot j = 0$, $i \cdot j = -1$ and $i \ne j$.

\begin{itemize}
	\item[C)] \textbf{Double crossings with dots}
\end{itemize}
\vskip-5pt
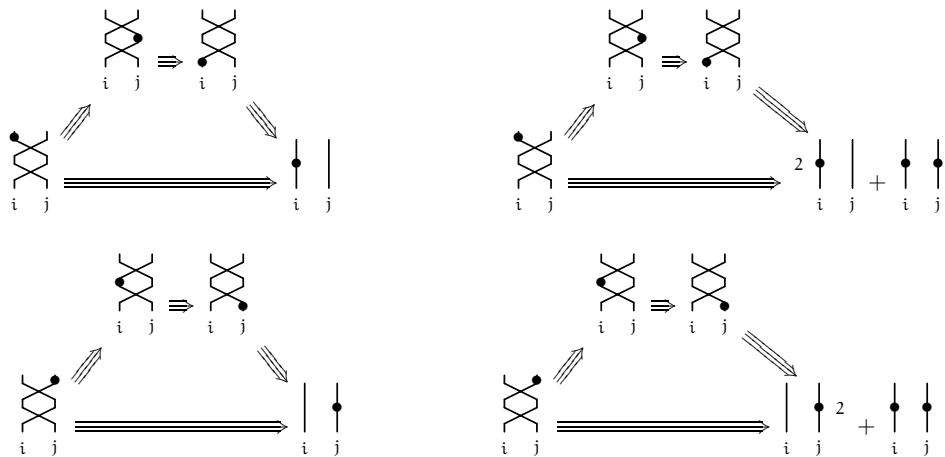
\[ \scalebox{0.811}{\xymatrix@R=0.7em@C=1em{ & \raisebox{-5mm}{$\tcrossmr{i}{j}$} \ar@3[r] &  \raisebox{-5mm}{$\tcrossdl{i}{j}$} \ar@3[dr] & \\
		\raisebox{-5mm}{$\tcrossul{i}{j}$} \ar@3[ur] \ar@3[rrr] &   &    &  \raisebox{-5mm}{$\didl{i}{j}$}  }} \hspace{2cm} 
	\scalebox{0.811}{\xymatrix@C=1em@R=0.7em{ & \raisebox{-5mm}{$\tcrossmr{i}{j}$} \ar@3[r] &  \raisebox{-5mm}{$\tcrossdl{i}{j}$} \ar@3[dr] & \\
		\raisebox{-5mm}{$\tcrossul{i}{j}$} \ar@3[ur] \ar@3[rrr] &   &    & \raisebox{-5mm}{$\didldot{i}{j}{2}$} + \raisebox{-5mm}{$\didlr{i}{j}$} }} \]
		
\[ \scalebox{0.811}{\xymatrix@R=0.7em@C=1em{ & \raisebox{-5mm}{$\tcrossml{i}{j}$} \ar@3[r] &  \raisebox{-5mm}{$\tcrossdr{i}{j}$} \ar@3[dr] & \\
		\raisebox{-5mm}{$\tcrossur{i}{j}$} \ar@3[ur] \ar@3[rrr] &   &    &  \raisebox{-5mm}{$\didr{i}{j}$}  } \hspace{2cm} 
	\xymatrix@C=1em@R=0.7em{ & \raisebox{-5mm}{$\tcrossml{i}{j}$} \ar@3[r] &  \raisebox{-5mm}{$\tcrossdr{i}{j}$} \ar@3[dr] & \\
		\raisebox{-5mm}{$\tcrossur{i}{j}$} \ar@3[ur] \ar@3[rrr] &   &    &  \raisebox{-5mm}{$\didrdot{i}{j}{2}$} + \raisebox{-5mm}{$\didlr{i}{j}$} }} \]
when $i \ne j$ and $i \cdot j = 0$ or $i \cdot j = -1$ respectively. When $i = j$, we have the following situation:
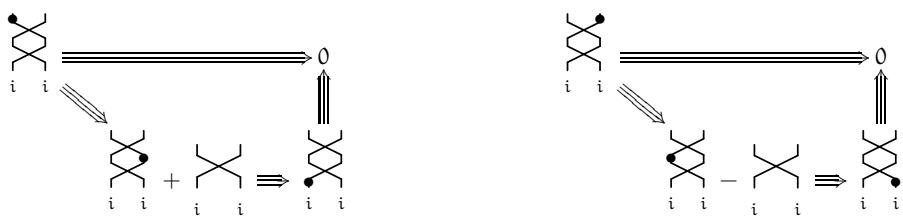
\[ \scalebox{0.811}{\xymatrix@-2ex@R=0.7em{ 
		\raisebox{-6mm}{$\tcrossul{i}{i}$} \ar@3[rr] \ar@3[dr] & & 0  \\
		& \raisebox{-5mm}{$\tcrossmr{i}{i}$} + \raisebox{-6mm}{$\dcrossnodot{i}{i}$} \ar@3[r] & \raisebox{-5mm}{$\tcrossdl{i}{i}$}  \ar@3 [u]} \hspace{3cm} \xymatrix@-2ex@R=0.7em{ 
		\raisebox{-6mm}{$\tcrossur{i}{i}$} \ar@3[rr] \ar@3[dr] & & 0  \\
		& \raisebox{-5mm}{$\tcrossml{i}{i}$} - \raisebox{-6mm}{$\dcrossnodot{i}{i}$} \ar@3[r] & \raisebox{-5mm}{$\tcrossdr{i}{i}$} \ar@3 [u] }}   \]

\begin{itemize}                                                                    
	\item[D)] \textbf{Double Yang-Baxter} \\
\end{itemize}
\vskip-9pt
The form of this critical branching depends on the labels on the three strands and the value of the bilinear form $\cdot$ between them.
\begin{itemize}
	\item[{\bf i)}] First of all, we consider the case where two consecutive vertices are equal: for instance $i=j \ne k$. The other cases would provide the same discussion.
\end{itemize}
\vskip-4pt
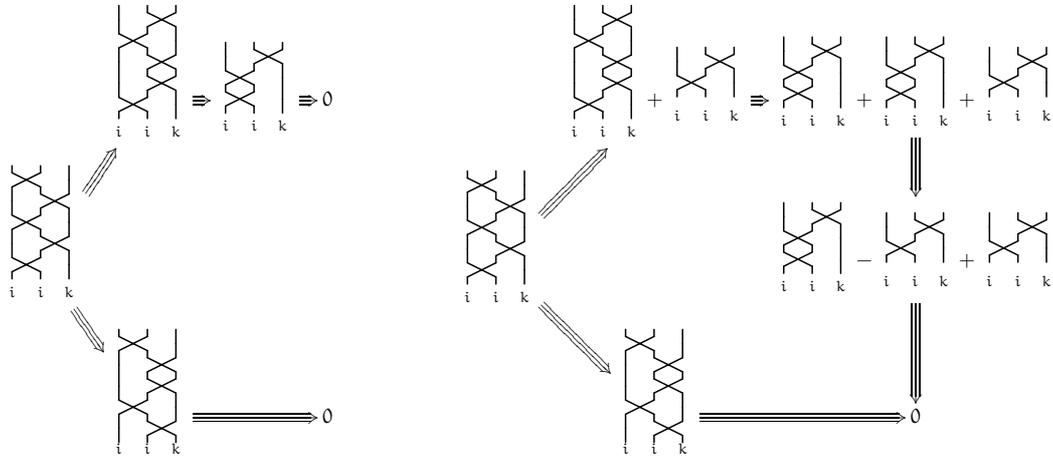
\[ \scalebox{0.711}{ \xymatrix@-3ex@R=0.5em{ & \raisebox{-7mm}{$\ybdcdcg{i}{i}{k}$} \ar@3[r] & \raisebox{-6mm}{$\cddcg{i}{i}{k}$} \ar@3[r] & 0           \\
		\raisebox{-7mm}{$\doubleybg{i}{i}{k}$} \ar@3[ur] \ar@3[dr] & & &    \\
		& \raisebox{-7mm}{$\cgcdybd{i}{i}{k}$} \ar@3[rr] &  & 0 } \hspace{2cm} \xymatrix@-3ex@R=0.5em{ & \raisebox{-7mm}{$\ybdcdcg{i}{i}{k}$} + \raisebox{-4mm}{$\cdcg{i}{i}{k}$} \ar@3[r] & \raisebox{-5mm}{$\cddcg{i}{i}{k}$} + \raisebox{-5mm}{$\cddcg{i}{i}{k}$} + \raisebox{-4mm}{$\cdcg{i}{i}{k}$} \ar@3[d] \\
		\raisebox{-8mm}{$\doubleybg{i}{i}{k}$} \ar@3[ur] \ar@3[dr] & & \raisebox{-6mm}{$\cddcg{i}{i}{k}$} - \raisebox{-5mm}{$\cdcg{i}{i}{k}$} + \raisebox{-5mm}{$\cdcg{i}{i}{k}$} \ar@3[d] \\
		& \raisebox{-7mm}{$\cgcdybd{i}{i}{k}$} \ar@3[r] &  0   }}
\]
when $i \cdot k = 0$ and $i \cdot k = -1$ respectively.
\vskip-5pt
\begin{itemize}
	\item[{\bf ii)}] When three vertices are distinct: we have to distinguish 6 cases according the values of $i \cdot j$, $j \cdot k$ and $i \cdot k$. We focus on the case $i \cdot j=i \cdot k=j \cdot k=-1$, the other forms are proved confluent similarly.
\end{itemize}
\vskip-4pt%
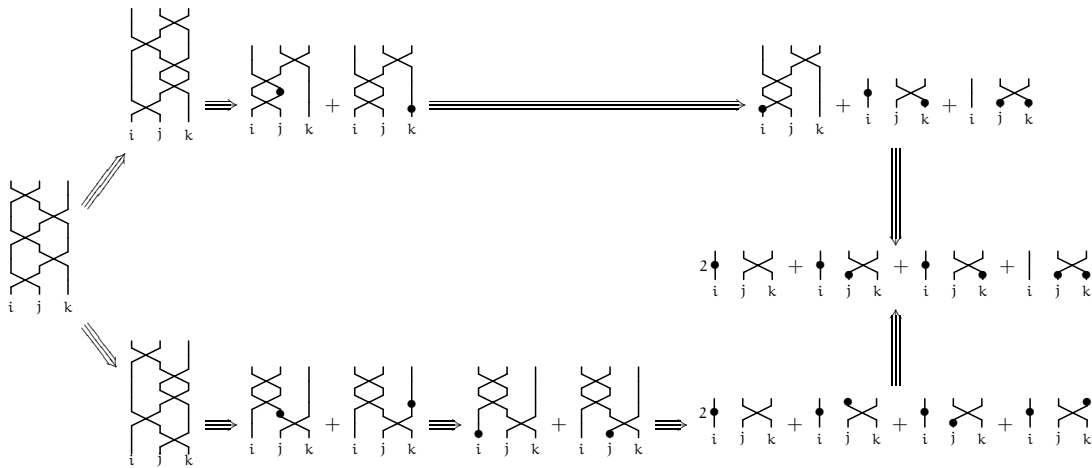
\[ \scalebox{0.711}{\xymatrix@C=1.5em@R=0.5em{ & \raisebox{-7mm}{$\ybdcdcg{i}{j}{k}$} \ar@3[r] & \raisebox{-6mm}{$\cddcgmmm{i}{j}{k}$} + \raisebox{-6mm}{$\cddcgdr{i}{j}{k}$} \ar@3[rr] & & \raisebox{-6mm}{$\cddcgdl{i}{j}{k}$} +    \raisebox{-4mm}{$\cdlr{i}{j}{k}$} + \raisebox{-4mm}{$\cddotsmr{i}{j}{k}$}  \ar@3[d]  \\
		\raisebox{-9mm}{$ \doubleybg{i}{j}{k}$} \ar@3[ur] \ar@3[dr] & & & &   \raisebox{-6mm}{$\cdll{i}{j}{k}{2}$} + \raisebox{-6mm}{$\cddotslm{i}{j}{k}$} + \raisebox{-6mm}{$\cdlr{i}{j}{k}$} + \raisebox{-6mm}{$\cddotsmr{i}{j}{k}$} \\
		&  \raisebox{-8mm}{$\cgcdybd{i}{j}{k}$} \ar@3[r] & \raisebox{-6mm}{$\dcgcdmm{i}{j}{k}$} + \raisebox{-6mm}{$\dcgcdur{i}{j}{k}$} \ar@3[r] & \raisebox{-6mm}{$\dcgcddl{i}{j}{k}$} + \raisebox{-6mm}{$\dcgcddm{i}{j}{k}$} \ar@3[r] & \raisebox{-4mm}{$\cdll{i}{j}{k}{2}$} + \raisebox{-4mm}{$\cddotslul{i}{j}{k}$} + \raisebox{-4mm}{$\cddotslm{i}{j}{k}$} + \raisebox{-4mm}{$\cddotslur{i}{j}{k}$} \ar@3[u] }} \]

\vskip-10pt
\begin{itemize}                    
	\item[{\bf iii)}] Let us consider the case $i=k$:
\end{itemize}
\hspace{-1cm}
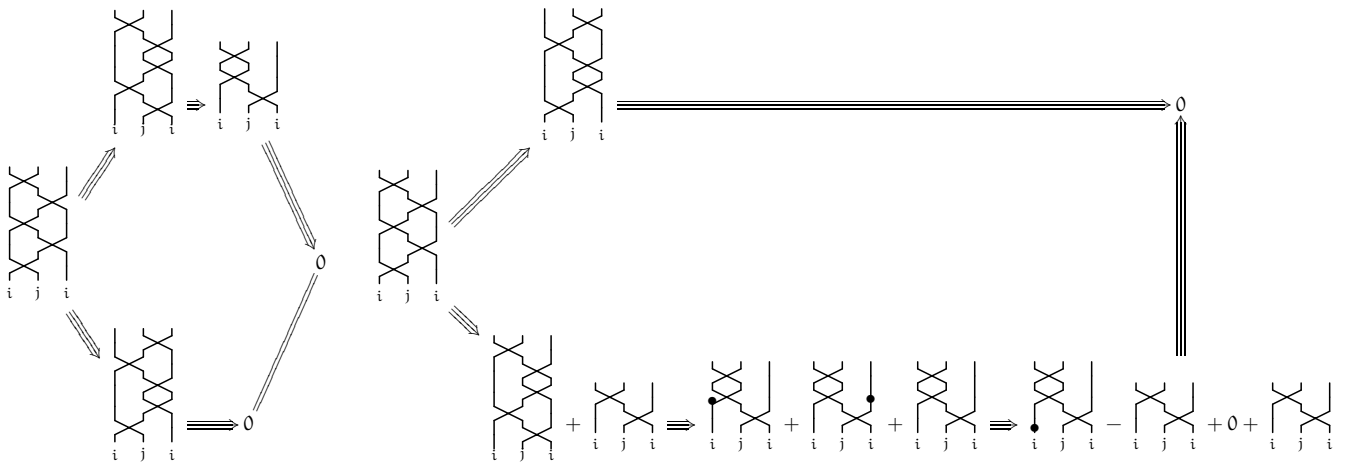
$$ \scalebox{0.711} { \xymatrix@-3ex@R=0.5em{ & \raisebox{-6mm}{$\cgcdybd{i}{j}{i}$} \ar@3[r] & \raisebox{-5mm}{$\dcgcd{i}{j}{i}$} \ar@3[dr] &  \\
		\raisebox{-7mm}{$\doubleybg{i}{j}{i}$} \ar@3[ur] \ar@3[dr] & & & 0 \\
		& \raisebox{-7mm}{$\ybdcdcg{i}{j}{i}$} \ar@3[r] & 0 \ar@{=}[ur] &  } \hspace{0.5cm}
	\xymatrix@-2ex@R=0.5em{ 
		& \raisebox{-7mm}{$\ybdcdcg{i}{j}{i}$} \ar@3[rr] &  & 0 \\
		\raisebox{-7mm}{$\doubleybg{i}{j}{i}$} \ar@3[ur] \ar@3[dr] &  &  &  &     \\
		& \raisebox{-7mm}{$\cgcdybd{i}{j}{i}$} + \raisebox{-5mm}{$\cgcd{i}{j}{i}$} \ar@3[r] & \raisebox{-5mm}{$\dcgcdml{i}{j}{i}$} + \raisebox{-5mm}{$\dcgcdur{i}{j}{i}$} + \raisebox{-5mm}{$\dcgcd{i}{j}{i}$}  \ar@3[r] & \raisebox{-5mm}{$\dcgcddl{i}{j}{i}$} - \raisebox{-5mm}{$\cgcd{i}{j}{i}$} + 0 + \raisebox{-5mm}{$\cgcd{i}{j}{i}$} \ar@3[uu] &       }} $$
	\newgeometry{top=2.2cm, bottom=1.9cm, left=2.5cm , right=2.5cm}	
		
when $i \cdot j = 0$ and $i \cdot j = -1$ respectively.

\begin{itemize}                                                                        
	\item[E)] \textbf{Yang-Baxter + Crossings} \\
\end{itemize}
\vskip-10pt
\begin{itemize}
	\item[{\bf i)}] We treat at first the case when two consecutive vertices are equal. For instance if $i=j$ or $i=k$, we have respectively:
\end{itemize}
\vskip-8pt
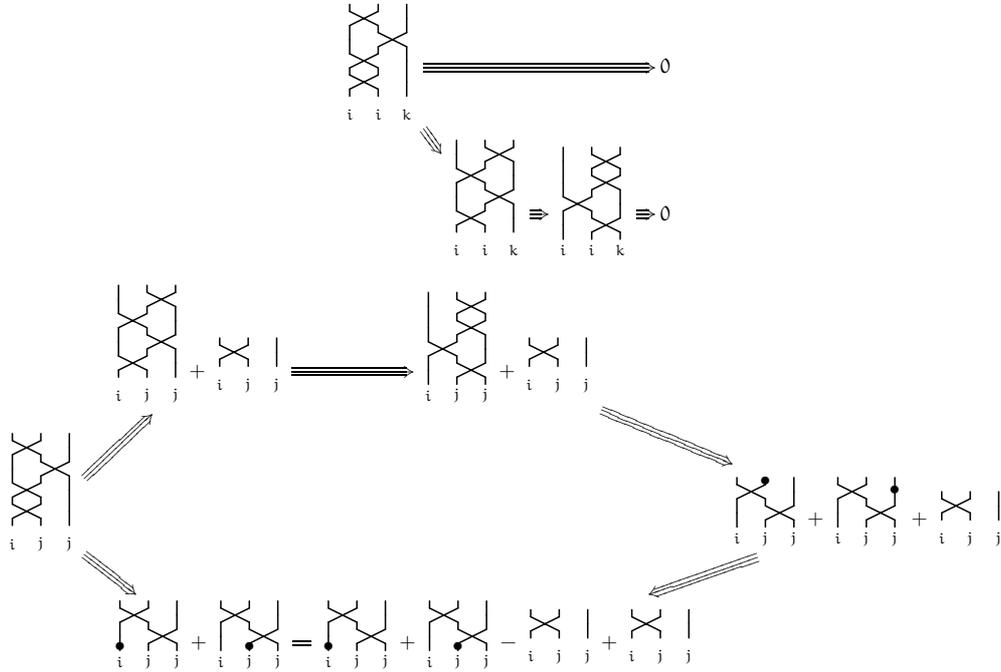
\[ \scalebox{0.711}{\xymatrix@-3ex@R=0.6em{ \raisebox{-10mm}{$\ybgcrossd{i}{i}{k}$} \ar@3[rrr] \ar@3[dr] & & &  0\\ 
		& \raisebox{-8mm}{$\doublecdcg{i}{i}{k}$} \ar@3[r] & \raisebox{-8mm}{$\tauybd{i}{i}{k}$} \ar@3[r] & 0 }} \]
\vskip-7pt
\[  \scalebox{0.711}{\xymatrix@-3ex@R=0.6em{  & \raisebox{-6mm}{$\doublecdcg{i}{j}{j}$} + \raisebox{-4mm}{$\cg{i}{j}{j}$} \ar@3[r] & \raisebox{-6mm}{$\tauybd{i}{j}{j}$} + \raisebox{-4mm}{$\cg{i}{j}{j}$} \ar@3[dr] &  \\ 
		\raisebox{-6mm}{$\ybgcrossd{i}{j}{j}$} \ar@3[ur] \ar@3[dr] &   & & \raisebox{-5mm}{$\cgcdum{i}{j}{j}$} + \raisebox{-5mm}{$\cgcdur{i}{j}{j}$} + \raisebox{-5mm}{$\cg{i}{j}{j}$} \ar@3[dl] \\
		&  \raisebox{-5mm}{$\cgcddl{i}{j}{j}$} + \raisebox{-5mm}{$\cgcddm{i}{j}{j}$} \ar@{=} [r] & \raisebox{-5mm}{$\cgcddl{i}{j}{j}$}  + \raisebox{-5mm}{$\cgcddm{i}{j}{j}$} - \raisebox{-4mm}{$\cg{i}{j}{j}$} + \raisebox{-4mm}{$\cg{i}{j}{j}$}  &    }} \]
when $i \cdot j = -1$.
\vskip-8pt
\begin{itemize}
	\item[{\bf ii)}] We check the case where all the vertices are different: one can check that the critical branching only depends on the value of $i \cdot k$:
\end{itemize}
\vskip-6pt
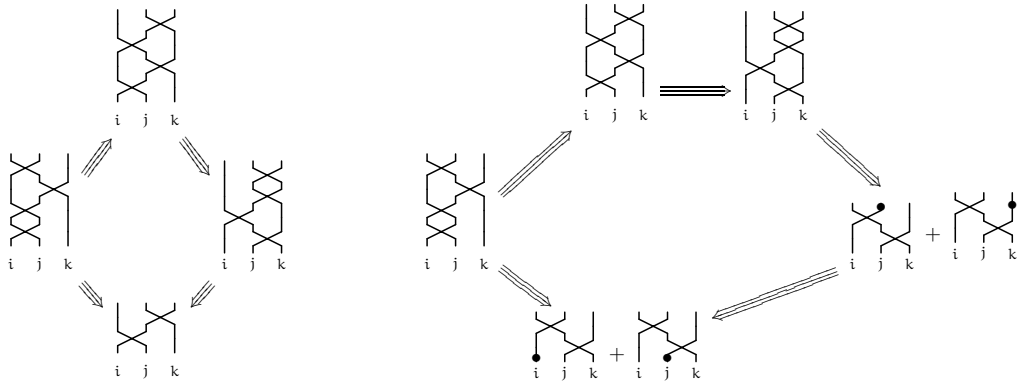
\[ \scalebox{0.711}{ \xymatrix@-3ex@R=0.4em{ & \raisebox{-7mm}{$\doublecdcg{i}{j}{k}$} \ar@3[dr] &  \\
		\raisebox{-7mm}{$\ybgcrossd{i}{j}{k}$} \ar@3[ur] \ar@3[dr] &  & \raisebox{-7mm}{$\tauybd{i}{j}{k}$} \ar@3[dl] \\
		& \raisebox{-4mm}{$\cdcg{i}{j}{k}$} & } \hspace{2cm} \xymatrix@-3ex@R=0.4em{ & \raisebox{-6mm}{$\doublecdcg{i}{j}{k}$} \ar@3[r] & \raisebox{-6mm}{$\tauybd{i}{j}{k}$} \ar@3[dr] &   \\
		\raisebox{-7mm}{$\ybgcrossd{i}{j}{k}$} \ar@3[ur] \ar@3[dr] & & & \raisebox{-7mm}{$\cgcdum{i}{j}{k}$} + \raisebox{-5mm}{$\cgcdur{i}{j}{k}$} \ar@3[dll] \\
		& \raisebox{-5mm}{$\cgcddl{i}{j}{k}$} + \raisebox{-5mm}{$\cgcddm{i}{j}{k}$} & &   } } \]
when $ i \cdot k = 0$ and $i \cdot k = -1$ respectively.
\begin{itemize}
\vskip-10pt
	\item[iii)] When the bottom sequence is $iji$, we focus on the case $ i \cdot j = -1$ and the other case would be similar:
\end{itemize} 
\vskip-7pt
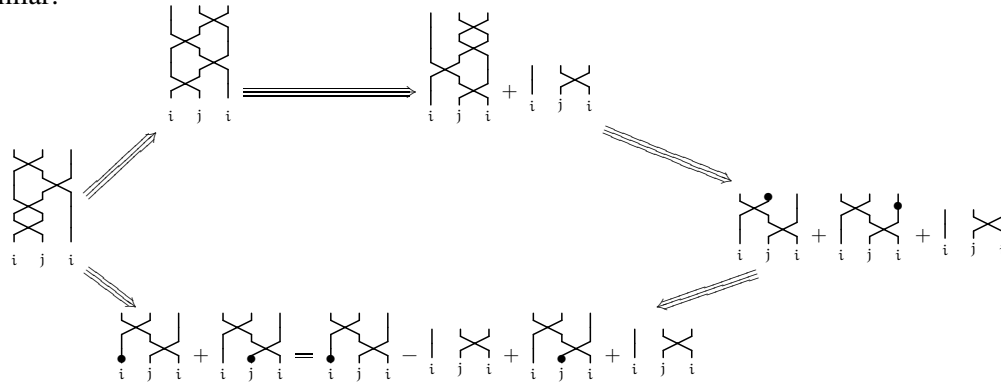
\[ \scalebox{0.711}{ \xymatrix@-3ex@R=0.4em{ & \raisebox{-6mm}{$\doublecdcg{i}{j}{i}$} \ar@3[r] & \raisebox{-6mm}{$\tauybd{i}{j}{i}$} + \raisebox{-4mm}{$\cd{i}{j}{i}$} \ar@3[dr] &  \\
		\raisebox{-6mm}{$\ybgcrossd{i}{j}{i}$} \ar@3[ur] \ar@3[dr] & & & \raisebox{-5mm}{$\cgcdum{i}{j}{i}$} + \raisebox{-5mm}{$\cgcdur{i}{j}{i}$} + \raisebox{-4mm}{$\cd{i}{j}{i}$} \ar@3[dl] \\
		& \raisebox{-5mm}{$\cgcddl{i}{j}{i}$} + \raisebox{-5mm}{$\cgcddm{i}{j}{i}$} \ar@{=}[r] & \raisebox{-5mm}{$\cgcddl{i}{j}{i}$} - \raisebox{-4mm}{$\cd{i}{j}{i}$} + \raisebox{-5mm}{$\cgcddm{i}{j}{i}$} + \raisebox{-4mm}{$\cd{i}{j}{i}$} & }} \]

We study the confluence diagrams of all the forms of the branching $\raisebox{-7mm}{$\scalebox{0.711}{\tauybg{}{}{}}$}$ in the same way.
 
\begin{itemize} 
	\item[F)] \textbf{Yang-Baxter with dots} \\
\end{itemize}
\begin{itemize}
\vskip-10pt
	\item[i)] When the three vertices are disctinct, the diagrams do not depend on the values of the bilinear pairing.
\end{itemize}
\vskip-5pt
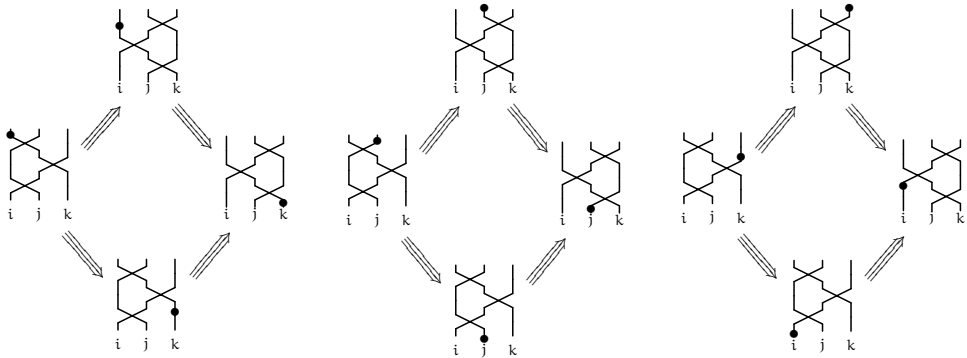
\[ \scalebox{0.711}{ \xymatrix@-3ex{ & \raisebox{-4mm}{$\ybdul{i}{j}{k}$} \ar@3[dr] & \\
		\raisebox{-4mm}{$\ybgul{i}{j}{k}$} \ar@3[dr] \ar@3[ur] & & \raisebox{-4mm}{$\ybddr{i}{j}{k}$} \\
		& \raisebox{-4mm}{$\ybgdr{i}{j}{k}$} \ar@3[ur] & } \hspace{0.5cm} \xymatrix@-3ex{ & \raisebox{-4mm}{$\ybdum{i}{j}{k}$} \ar@3[dr] & \\
		\raisebox{-4mm}{$\ybgum{i}{j}{k}$} \ar@3[dr] \ar@3[ur] & & \raisebox{-4mm}{$\ybddm{i}{j}{k}$} \\
		& \raisebox{-4mm}{$\ybgdm{i}{j}{k}$} \ar@3[ur] & }  \hspace{0.5cm} \xymatrix@-3ex{ & \raisebox{-4mm}{$\ybdur{i}{j}{k}$} \ar@3[dr] & \\
		\raisebox{-4mm}{$\ybgur{i}{j}{k}$} \ar@3[dr] \ar@3[ur] & & \raisebox{-4mm}{$\ybddl{i}{j}{k}$} \\
		& \raisebox{-4mm}{$\ybgdl{i}{j}{k}$} \ar@3[ur] & }} \]

\begin{itemize}
	\item[ii)] When two consecutive vertices are equal, for instance if $i = j \ne k$, if a dot is placed on the left strand, then it will go down in the diagram without creating any additive term because there will be no crossing with two strands with the same label, so that the branching is trivially confluent. For the other cases, the same process applies. Let us prove the confluence when there is a dot on the rightmost strand:
\end{itemize}
\vskip-5pt
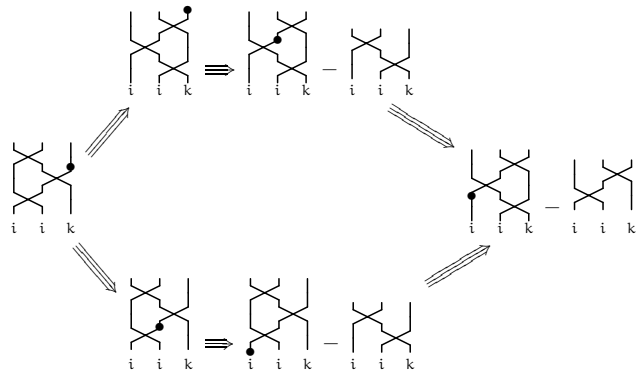
\[ \scalebox{0.711}{\xymatrix@-2ex{ & \raisebox{-5mm}{$\ybdur{i}{i}{k}$} \ar@3[r] & \raisebox{-5mm}{$\ybdmm{i}{i}{k}$} - \raisebox{-5mm}{$\cgcd{i}{i}{k}$} \ar@3[dr] & \\
		\raisebox{-5mm}{$\ybgur{i}{i}{k}$} \ar@3[ur] \ar@3[dr] & & & \raisebox{-5mm}{$\ybddl{i}{i}{k}$} - \raisebox{-5mm}{$\cdcg{i}{i}{k}$} \\
		& \raisebox{-5mm}{$\ybgmm{i}{i}{k}$} \ar@3[r] & \raisebox{-5mm}{$\ybgdl{i}{i}{k}$} - \raisebox{-5mm}{$\cgcd{i}{i}{k}$} \ar@3[ur] & }} \]
One may apply the same process for the case $i \ne j = k$ with a dot placed on the up of the leftmost (or middle) strand.

\begin{itemize}
	\item[iii)] When the bottom sequence is $iji$, the way to make a dot go down is the same no matter where the dot is placed at the beginning, we only check confluence for a dot placed on the leftmost strand. It would provide the same diagram for the other cases.
\end{itemize}
\vskip-5pt
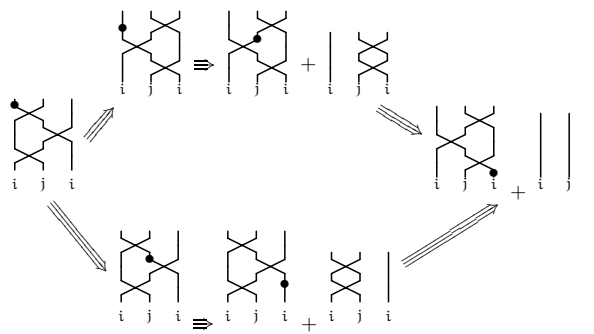
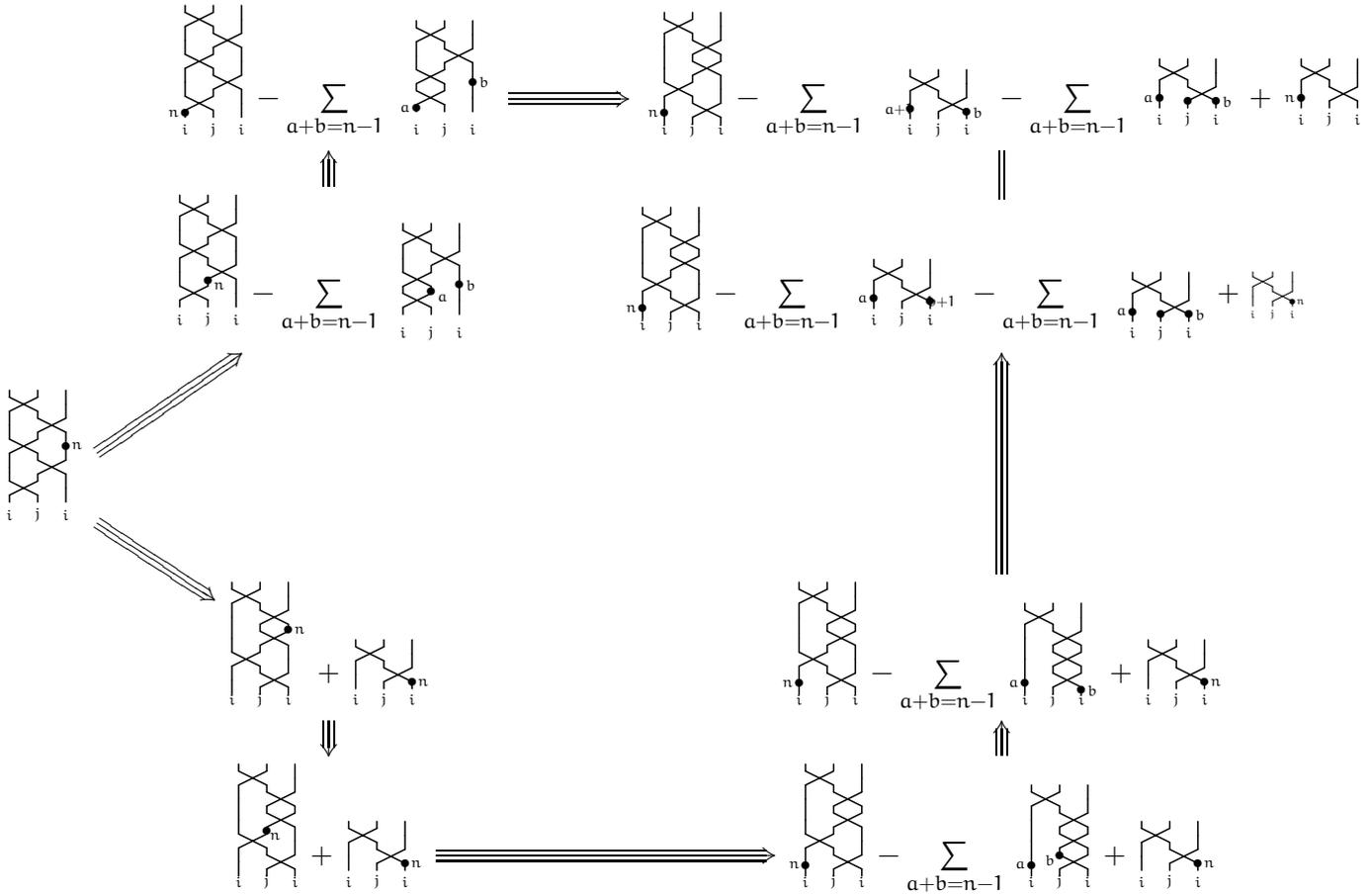
\[ \scalebox{0.711}{ \xymatrix@-3ex{ & \raisebox{-6mm}{$\ybdul{i}{j}{i}$} \ar@3[r] & \raisebox{-6mm}{$\ybdmm{i}{j}{i}$} + \raisebox{-6mm}{$\dcr{i}{j}{i}$} \ar@3[dr] & \\
		\ybgul{i}{j}{i} \ar@3[ur] \ar@3[dr] & & & \ybddr{i}{j}{i} + \tid{i}{j}{i} \\
		& \ybgmmm{i}{j}{i} \ar@3[r] & \ybgdr{i}{j}{i} + \dcl{i}{j}{i}  \ar@3[ur] & }} \]
		
	\newgeometry{top=2.2cm, bottom=1.7cm, left=2.5cm , right=2.5cm}
		
\[ \scalebox{0.711}{ \xymatrix@-3ex@R=0.5em{ & \raisebox{-5mm}{$\ybdul{i}{j}{i}$} + \raisebox{-5mm}{$\tidl{i}{j}{i}$} \ar@3[r] & \raisebox{-5mm}{$\ybdmmm{i}{j}{i}$} + \raisebox{-5mm}{$\tidl{i}{j}{i}$} + \raisebox{-5mm}{$\dcr{i}{j}{i}$} \ar@3[dr] & \\
		\raisebox{-6mm}{$\ybgul{i}{j}{i}$} \ar@3[ur] \ar@3[dr] &  & & \raisebox{-6mm}{$\ybddr{i}{j}{i}$} + \raisebox{-5mm}{$\tidl{i}{j}{i}$} + \raisebox{-5mm}{$\tidm{i}{j}{i}$} + \raisebox{-5mm}{$\tidr{i}{j}{i}$} \\
		& \raisebox{-6mm}{$\ybgmmm{i}{j}{i}$} \ar@3[r] & \raisebox{-6mm}{$\ybgdr{i}{j}{i}$} + \raisebox{-6mm}{$\dcl{i}{j}{i}$} \ar@3[ur] &   }} \]
		\vskip-5pt
when $i \cdot j = 0$ and $i \cdot j = -1$ respectively.

\vskip-15pt

\begin{itemize}
	\item[G)] \textbf{Indexed critical branchings}
\end{itemize}
	Let us prove that the indexed critical branchings of the form (\ref{indexyb}) given in Section \ref{SSS:CriticalBranchingsKLRAlgebra} are confluent, in the following two cases:
plug in (\ref{indexyb}) is given by the following $2$-cells:
\begin{enumerate}[{\bf i)}]
\item $\identdots{i}{n}$ for every $n \in N$,
\item \raisebox{-6mm}{$\scalebox{1}{\dcrossdldot{i}{l}{n}}$}  \: for all $n \in \N$ and for any $l$ in $I$.
\end{enumerate}

		For the first case, the instance for $n=0$ was already checked in the Double Yang-Baxter family of critical branchings.
		Let us prove the confluence of this indexed critical branchings in the particular case when $i=k$ and $i \cdot j=-1$. This is the "most complicated" case in the sense that it is the one that creates the most additive terms.
		
Let us denote by $\alpha_{i,j}^{L,n}$ and $\alpha_{i,j}^{R,n}$ the $3$-cells 		
\vskip-7pt		
 \[ \alpha_{i,j}^{L,n} = \underbrace{ \alpha_{i,j}^L \star_{2} \alpha_{i,j}^L \dots \star_2 \alpha_{i,j}^L}_{n \quad \textrm{times}} \qquad \textrm{(resp. } \alpha_{i,j}^{R,n} = \underbrace{ \alpha_{i,j}^R \star_{2} \alpha_{i,j}^R \dots \star_2 \alpha_{i,j}^R}_{n \quad \textrm{times}} ) \] depicted by
\vskip-7pt
		\[ \scalebox{0.811}{\xymatrix{ \raisebox{-6mm}{$\dcrossuldot{i}{j}{n}$} \ar@3[r] ^-{\alpha_{i,j}^{L,n}} & \raisebox{-6mm}{$\dcrossdrdot{i}{j}{n}$} }}, \qquad \scalebox{0.811}{ \xymatrix@+2ex{  \raisebox{-6mm}{$\dcrossuldot{i}{i}{n}$} \; \ar@3[r] ^-{\alpha_{i,i}^{L,n}} &  \raisebox{-6mm}{$\dcrossdrdot{i}{i}{n}$} + \sum\limits_{a+b = n-1}{ \raisebox{-6mm}{$\didlrdot{i}{i}{a}{b}$} } }}, \]
\vskip-5pt		
		\[ \scalebox{0.811}{\xymatrix@+2ex{ \raisebox{-6mm}{$\dcrossurdot{i}{j}{n}$} \ar@3[r]^{\alpha_{i,j}^{R,n}} & \raisebox{-6mm}{$\dcrossdldot{i}{j}{n}$} }}, \qquad \scalebox{0.811}{\xymatrix{  \raisebox{-6mm}{$\dcrossurdot{i}{i}{n}$} \; \ar@3[r]^{\alpha_{i,i}^{R,n}} &  \raisebox{-6mm}{$\dcrossdldot{i}{i}{n}$} - \sum\limits_{a+b = n-1}{ \raisebox{-6mm}{$\didlrdot{i}{i}{a}{b}$} } }}. \]

\vskip-10pt
	
	Thus, we have:
	\begin{eqnarray*} 
		\hspace{-2cm}
		\xymatrix @R=1.2em @C=1.8em { &   \raisebox{-5mm}{$\scalebox{0.711}{\doubleybgdotc{i}{j}{i}{n}}$} - \sum\limits_{a+b=n-1}{ \raisebox{-5mm}{$\scalebox{0.711}{\ybgcrossdotsb{i}{j}{i}{a}{b} }$}} \ar@3[rr] & & \raisebox{-5mm}{$\scalebox{0.711}{\indexc{i}{j}{i}{n}}$} - \sum\limits_{a+b=n-1}{\raisebox{-5mm}{$\scalebox{0.711}{ \cgcdot{i}{j}{i}{a+1}{b} }$}} - \sum\limits_{a+b=n-1}{\scalebox{0.711}{ \raisebox{-5mm}{$\cgcdots{i}{j}{i}{a}{}{b}$} }} +  \scalebox{0.711}{\raisebox{-5mm}{$\cgcddld{i}{j}{i}{n}$}} \ar@{=}[d]   \\
			& \scalebox{0.711}{\raisebox{-6mm}{$\doubleybgdotb{i}{j}{i}{n}$}} - \sum\limits_{a+b=n-1}{\raisebox{-6mm}{$\scalebox{0.711}{ \ybgcrossdotsa{i}{j}{i}{a}{b} }$}} \ar@3[u] & &  \hspace{-1cm} \scalebox{0.711}{\raisebox{-7mm}{$\indexc{i}{j}{i}{n}$}} - \sum\limits_{a+b=n-1}{\scalebox{0.711}{ \raisebox{-6mm}{$\cgcdot{i}{j}{i}{a}{b+1}$} }} - \sum\limits_{a+b=n-1}{\raisebox{-6mm}{$\scalebox{0.711}{\cgcdots{i}{j}{i}{a}{}{b}}$}} +  \scalebox{0.5}{\raisebox{-6mm}{$\cgcddrd{i}{j}{i}{n}$}} \\
			\scalebox{0.711}{\raisebox{-6mm}{$\doubleybgdot{i}{j}{i}{n}$}}  \ar@3[ur] \ar@3[dr]  & & & &  \\
			&   \scalebox{0.711}{\raisebox{-6mm}{$\indexa{i}{j}{i}{n}$}} +  \scalebox{0.711}{\raisebox{-6mm}{$\cgcddrd{i}{j}{i}{n}$}} \ar@3[d] &  & \scalebox{0.711}{\raisebox{-6mm}{$\indexc{i}{j}{i}{n}$}} - \sum\limits_{a+b=n-1} {\scalebox{0.711}{\raisebox{-6mm}{$\cgtcdotsb{i}{j}{i}{a}{b}$}}} +  \scalebox{0.711}{\raisebox{-6mm}{$\cgcddrd{i}{j}{i}{n}$}} \ar@3[uu]  \\
			&   \scalebox{0.711}{\raisebox{-6mm}{$\indexb{i}{j}{i}{n}$}} +  \scalebox{0.711}{\raisebox{-6mm}{$\cgcddrd{i}{j}{i}{n}$}} \ar@3[rr] & & \scalebox{0.711}{\raisebox{-6mm}{$\indexc{i}{j}{i}{n}$}} - \sum\limits_{a+b=n-1} {\scalebox{0.711}{\raisebox{-6mm}{$\cgtcdotsa{i}{j}{i}{a}{b}$}}} +  \scalebox{0.711}{\raisebox{-6mm}{$\cgcddrd{i}{j}{i}{n}$}} \ar@3[u]  }
	\end{eqnarray*}


		For the second indexation, one remarks that the fourth vertex of the sequences does not matter in the reductions. We consider the case where the bottom sequence is $ijik$ with $i \cdot j =0$. Let us at first consider this indexation for $n=0$: 
		\begin{eqnarray} 
		\label{indexnul}
		\scalebox{0.8}
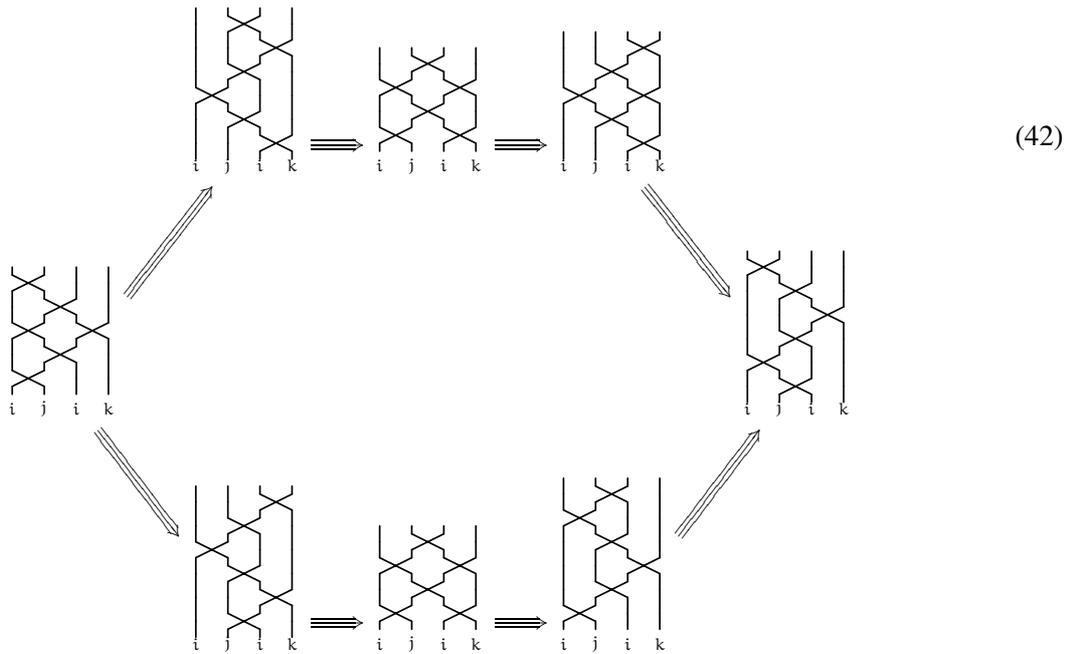
{\xymatrix{ & \raisebox{-5mm}{$\lastbc{i}{j}{i}{k}$} \ar@3[r] & \raisebox{-5mm}{$\lasta{i}{j}{i}{k}$} \ar@3[r] & \raisebox{-5mm}{$\lastab{i}{j}{i}{k}$} \ar@3[dr] & \\
				\raisebox{-6mm}{$\last{i}{j}{i}{k}$} \ar@3[ur] \ar@3[dr] & & & & \raisebox{-6mm}{$\lastac{i}{j}{i}{k}$} \\
				& \raisebox{-5mm}{$\lastb{i}{j}{i}{k}$} \ar@3[r] & \raisebox{-5mm}{$\lasta{i}{j}{i}{k}$} \ar@3[r] & \raisebox{-5mm}{$\lastba{i}{j}{i}{k}$} \ar@3[ur] & }}
		\end{eqnarray}
		This diagram was given in \cite{GM09} for the indexation of \raisebox{-4mm}{$\scalebox{0.6}{\crossing{}{}}$} in the double Yang-Baxter diagram. When $i \cdot j= -1$, it is the same branching except that it creates an extra term \[ \raisebox{-7mm}{$\scalebox{0.7}{\triplecg{i}{j}{i}{k}}$} \] in both reducing paths. For $n > 0$, the bottom line of (\ref{indexnul}) defines a $3$-cell $$\gamma_{ijik}: \raisebox{-6mm}{$\scalebox{0.7}{\last{i}{j}{i}{k}}$} \Rrightarrow  \raisebox{-7mm}{$\scalebox{0.7}{\lasta{i}{j}{i}{k}}$}.$$ As we started reducing only the bottom part on the diagram, we can apply the same reductions on the diagram \[ \scalebox{0.7}{\lastdota{i}{j}{i}{k}{n}} \] since the dot $2$-cell never appears in the source of any reduction.
		This enables us to define, for any $n \in \N$, a $3$-cell $$ \gamma_n:  \raisebox{-7mm}{$\scalebox{0.7}{\lastdota{i}{j}{i}{k}{n}}$} \Rrightarrow \raisebox{-7mm}{$\scalebox{0.7}{\lastbdota{i}{j}{i}{k}{n}}$} + \raisebox{-7mm}{$\scalebox{0.7}{\triplecgdot{i}{j}{i}{k}{n}}$} $$
		Then we have:
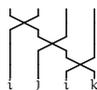
		\begin{eqnarray*}
			\hspace{-2cm}
			\xymatrix @R=1em @C=0.2em{ & & \scalebox{0.611}{\raisebox{-7mm}{$\lastdota{i}{j}{i}{k}{n}$}} \ar@3[dll] \ar@3[drr]^{\gamma_n} & & \\
				\scalebox{0.611}{\raisebox{-7mm}{$\lastdotb{i}{j}{i}{k}{n}$}} - \sum\limits_{a+b=n-1}{\scalebox{0.611}{\raisebox{-7mm}{$\sansnomadmr{i}{j}{i}{k}{a}{b}$}}} \ar@3[d] & & & & \scalebox{0.611}{\raisebox{-7mm}{$\lastbdota{i}{j}{i}{k}{n}$}} + \scalebox{0.6}{\raisebox{-7mm}{$\triplecgdot{i}{j}{i}{k}{n}$}} 
				\ar@3[d] \\
				\scalebox{0.711}{\raisebox{-7mm}{$\lastdotc{i}{j}{i}{k}{n}$}} - \sum\limits_{a+b=n-1}{\scalebox{0.611}{\raisebox{-7mm}{$\sansnomadlr{i}{j}{i}{k}{a}{b}$}}} \ar@3[d] & & & & \scalebox{0.611}{\raisebox{-7mm}{$\lastbdotb{i}{j}{i}{k}{n}$}} + \scalebox{0.611}{\raisebox{-7mm}{$\triplecgdot{i}{j}{i}{k}{n}$}} \ar@3[d] \\
				\scalebox{0.611}{\raisebox{-7mm}{$\lastdotc{i}{j}{i}{k}{n}$}} - \sum\limits_{a+b=n-1}{\scalebox{0.611}{\raisebox{-7mm}{$\triplecgddots{i}{j}{i}{k}{a}{}{b}$}}} - \sum\limits_{a+b=n-1}{\scalebox{0.611}{\raisebox{-7mm}{$\triplecgdots{i}{j}{i}{k}{a+1}{b}$}}} \ar@3[dd]^{\gamma_0} & & & & \scalebox{0.611}{\raisebox{-7mm}{$\lastbdotc{i}{j}{i}{k}{n}$}} + \scalebox{0.611}{\raisebox{-7mm}{$\triplecgdot{i}{j}{i}{k}{n}$}} - \sum\limits_{a+b=n-1}{\scalebox{0.611}{\raisebox{-7mm}{$\sansnombdotlm{i}{j}{i}{k}{a}{b}$}}} \ar@3[d] \\
				& & & & \scalebox{0.611}{\raisebox{-7mm}{$\lastbdotc{i}{j}{i}{k}{n}$}} + \scalebox{0.611}{\raisebox{-7mm}{$\triplecgdot{i}{j}{i}{k}{n}$}} - \sum\limits_{a+b=n-1}{\scalebox{0.611}{\raisebox{-7mm}{$\sansnombdotlr{i}{j}{i}{k}{a}{b}$}}} \ar@3[d] \\
				\scalebox{0.611}{\raisebox{-7mm}{$\lastbdotc{i}{j}{i}{k}{n}$}} + \scalebox{0.611}{\raisebox{-7mm}{$\triplecgdotl{i}{j}{i}{k}{n}$}} \ar@{=}[rrrr] & & & & \scalebox{0.611}{\raisebox{-7mm}{$\lastbdotc{i}{j}{i}{k}{n}$}} + \scalebox{0.6}{\raisebox{-7mm}{$\triplecgdot{i}{j}{i}{k}{n}$}} \\
				\hspace{1cm} - \sum\limits_{a+b=n-1}{ \scalebox{0.611}{\raisebox{-7mm}{$\triplecgddots{i}{j}{i}{k}{a}{}{b}$}}}- \sum\limits_{a+b=n-1}{\scalebox{0.611}{\raisebox{-7mm}{$\triplecgdots{i}{j}{i}{k}{a+1}{b}$}}} & & & & \hspace{1cm} - \sum\limits_{a+b=n-1}{\scalebox{0.611}{\raisebox{-7mm}{$\triplecgddots{i}{j}{i}{k}{a}{}{b}$}}} - \sum\limits_{a+b=n-1}{\scalebox{0.611}{\raisebox{-7mm}{$\triplecgdots{i}{j}{i}{k}{a}{b+1}$}}}  }
		\end{eqnarray*}

\section{Critical branchings modulo of the linear~$(3,2)$-polygraph modulo}

\subsection{Further $3$-cells in $\KLRb$}
\label{SS:FurtherCells}
In this subsection, we define some additional $3$-cells in $\KLRb_3$, which we will use to prove the confluence modulo of the linear~$(3,2)$-polygraph modulo $\ER$. First of all, using the degree conditions on bubbles on the terms
\[ \sum\limits_{r \geq 0} {\stda{i}{n}{n-r-2}{r}} \: ; \hspace{1cm} \text{(resp. $\sum\limits_{r \geq 0} {\stdb{i}{r}{n-r-2}{n}}$ \: )}, \]
when $r > - \hil -1$ (resp. $r \leq \hil -1$), then $n-r-2 < - \hil -1$ (resp. $n-r-2 < \hil -1$ and then the bubble reduces to $0$. We then denote by $b'_{i, \lambda}$ and $c'_{i, \lambda}$ the following $3$-cells in $KLR$ obtained by application of the $3$-cells $b_{i,\lambda}^0$ and $c_{i,\lambda}^0$:
\[ \sum\limits_{r \geq 0} {\stda{i}{n}{n-r-2}{r}} \overset{b'_{i,\lambda}}{\Rrightarrow} \sum\limits_{r = 0}^{-\hil -1} {\stda{i}{n}{n-r-2}{r}} \: ; \hspace{1cm}  \sum\limits_{r \geq 0} {\stdb{i}{r}{n-r-2}{n}} \overset{c'_{i,\lambda}}{\Rrightarrow} \sum\limits_{r = 0}^{\hil -1} {\stdb{i}{r}{n-r-2}{n}} \] 

We also define the $3$-cell $A'_{i,\lambda}$ for $\hil \geq 0$ having as $2$-source \[ \tfishdrp{i}{n} \] and as $2$-target either $0$ if $n < \hil$ or $- \cupr{i}$ if $n = \hil$ as the following composite of rewriting steps in $\ER$:
\[ \xymatrix@C=2.8em{ \tfishdrpr{i}{n} \ar@3 [r] ^-{(i_3^2)^{-} \cdot \alpha_{i,\lambda}^{R,+}} & \tfishdrpl{i}{n} - \sum\limits_{a+b=n-1} {\drulea{i}{a}{b}} \ar@3 [r] ^-{(i_1^2)^- \cdot A_{i,\lambda}} & 0 - \sum\limits_{a+b=n-1} {\drulea{i}{a}{b}} \ar @3[r]^{b_{i,\lambda}} & - \delta_{n,\hil} \cupr{i}  } \] 
where:
\begin{itemize}
	\item the $3$-cell $(i_3^2)^{-} \cdot \alpha_{i,\lambda}^{R,+}$ is the rewriting step of $\ER$ given by 
	\[ \tfishdrpr{i}{n} = \tfishdrpra{i}{n} \sim \tfishdrprb{i}{n} \overset{\alpha_{i,\lambda}^{R,+}}{\Rrightarrow}  \tfishdrpl{i}{n} - \sum\limits_{a+b=n-1} {\drulea{i}{a}{b}} \]
	\item the $3$-cell $b_{i,\lambda}$ is defined by successive applications of the cells $b_{i,\lambda}^{0,b}$ since $\posbubd{i}{b}$ reduces to $0$ unless $n = \hil$ and $a = 0$,$b= \hil -1$, and in that case $\posbubdf{i}{\hil -1}$ reduces to $1_{1_{\lambda}}$ by $b_{i,\lambda}^{1,\hil-1}$.
\end{itemize}

We define in a similar fashion $3$-cells
\[ \hspace{-2.1cm}
\tfishulp{i}{n} \overset{B'_{i,\lambda}}{\Rrightarrow}  \left\{
\begin{array}{ll}
- \capl{i} & \mbox{if } n = \hil  \\
0 & \mbox{if } n < \hil 
\end{array}
\right. ; \;  \tfishurp{i}{n} \overset{C'_{i,\lambda}}{\Rrightarrow}  \left\{
\begin{array}{ll}
\capr{i} & \mbox{if } n = - \hil  \\
0 & \mbox{if } n < - \hil 
\end{array}
\right. ; \; \tfishdlp{i}{n} \overset{D'_{i,\lambda}}{\Rrightarrow}  \left\{
\begin{array}{ll}
\cupl{i} & \mbox{if } n = - \hil  \\
0 & \mbox{if } n < - \hil 
\end{array}
\right. \]
for $\hil \geq 0$ for $B'_{i,\lambda}$ and $hil \leq 0$ for $C'_{i,\lambda}$ and $D'_{i,\lambda}$.

\subsection{Branchings from $KLR$ relations}
\label{A:BranchingsDots}
\subsubsection{Critical branchings $(A_{i,\lambda}, \alpha_{i,\lambda}^{L,+})$}
For any $i$ in $I$ and $\lambda$ in $X$ the weight lattice, and for any value of $\hil$, the critical branchings $(A_{i,\lambda}, \alpha_{i,\lambda}^{L,+})$ are confluent modulo $E$ as follows:
\[
\xymatrix@C=4.5em@R=3em{ \cbadot{i} 
	\ar@3[rrr] ^-{A_{i,\lambda}}
	\ar@3 [d] _-{\rotatebox{90}{=}} & & & - \sum\limits_{n=0}^{- \hil} \druleaf{i}{n+1}{-n-1} \ar@3 [d] ^-{(i_1^2)^-}  \\
	\cbadot{i} \ar@3 [r] _-{\alpha_{i,\lambda}^{L,+}} & \cbadotb{i} + \ruleawdot{i} \ar@3[r] _-{i_2^2 \star_2 (i_3^2)^- \cdot \alpha_{i,\lambda}^{R,+}} & \cbadotc{i}  \ar@3[r] _-{(i_1^2)^- \cdot A_{i,\lambda}} & - \sum\limits_{n=0}^{- \hil} \ruleabul{i}{n} 
} \]

\subsubsection{Critical branchings $(B_{i,\lambda}, i_4^2 \cdot \alpha_{i,\lambda}^{L,+})$}

\[
\xymatrix@C=4.5em@R=3em{ \cbbdot{i} 
	\ar@3[rrr] ^-{B_{i,\lambda}}
	\ar@3 [d] _-{\rotatebox{90}{=}} & & & - \sum\limits_{n=0}^{- \hil} \rulebbul{i}{n} \ar@3 [d] ^-{i_4^2} \\
	\cbbdot{i} \ar@3 [r] ^-{i_4^2 \cdot \alpha_{i,\lambda}^{L,+}} & \cbbdotb{i} + \rulebwdot{i} \ar@3 [r] _-{i_2^2 \star_2 (i_3^2)^- \cdot \alpha_{i,\lambda}^{R,+}} & \cbbdotc{i} \ar@3[r] _-{B_{i,\lambda}} & - \sum\limits_{n=0}^{- \hil} \druleb{i}{n+1}{-n-1} }   \]

\subsubsection{Critical branchings $(i_3^2 \cdot C_{i,\lambda}, \alpha_{i,\lambda}^{R,+})$}
\[
\xymatrix@C=4.5em@R=3em{ \cbcdot{i} 
	\ar@3[rrr] ^-{i_3^2 \cdot C_{i,\lambda}}
	\ar@3 [d] _-{\rotatebox{90}{=}} & & &  \sum\limits_{n=0}^{ \hil} \drulec{i}{n+1}{-n-1} \ar@3 [d] ^-{(i_3^2)^-} \\
	\cbcdot{i} \ar@3 [r] ^-{\alpha_{i,\lambda}^{R,+}} & \cbcdotb{i} - \rulecwdot{i} \ar@3 [r] _-{(i_1^2)^- \star_2 i_4^2 \cdot \alpha_{i,\lambda}^{L,+}} & \cbcdotc{i} \ar@3[r] _-{C_{i,\lambda}} &  \sum\limits_{n=0}^{ \hil} \rulecbul{i}{n} }   \]

\subsubsection{Critical branchings $(D_{i,\lambda}, \alpha_{i,\lambda}^{R,+})$}
\[
\xymatrix@C=5.5em@R=3em{ \cbddot{i} 
	\ar@3[rrr] ^-{D_{i,\lambda}}
	\ar@3 [d] _-{\rotatebox{90}{=}} & & &  \sum\limits_{n=0}^{ \hil} \druled{i}{n+1}{-n-1} \ar@3 [d] ^-{(i_2^2)^-} \\
	\cbddot{i} \ar@3 [r] ^-{\alpha_{i,\lambda}^{R,+}} & \cbddotb{i} - \ruledwdot{i} \ar@3[r] _{(i_1^2)^- \star_2 i_4^2 \cdot \alpha_{i,\lambda}^{L,+}} & \cbddotc{i} \ar@3 [r] _-{D_{i,\lambda}} &  \sum\limits_{n=0}^{ \hil} \ruledbul{i}{n} }   \]

\subsubsection{Critical branchings $(E_{i,\lambda}, \alpha_{i,\lambda}^{L,+})$ and $(F_{i,\lambda}, \alpha_{i,\lambda}^{R,+})$}
Let us prove that for any $i$ in $I$ and $\lambda$ in $X$, and for any value of $\hil$, the critical branching $(E_{i,\lambda}, \alpha_{i,\lambda}^{L,+})$ is confluent modulo $E$. The proof of confluence modulo of this branching follows the proof scheme of Lemma \ref{L:FurtherRelations}, and we prove the confluence of the critical branching $(F_{i,\lambda}, \alpha_{i,\lambda}^{L,+})$ similarly.
Let us denote by $\alpha_i$ the following composition of $3$-cells of $\ER$:
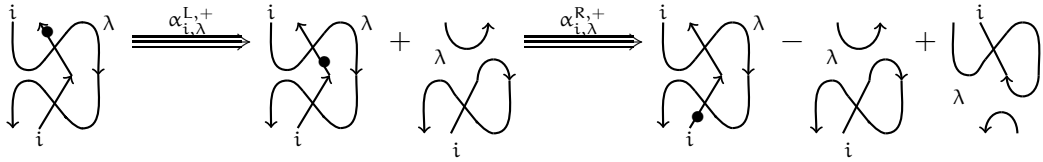
\[  \xymatrix@C=4em{ \tdcrossrldtest{i}{i} \ar@3[r] ^-{\alpha_{i,\lambda}^{L,+}} & \tdcrossrldm{i}{i} + \cuprb{i} \ar@3[r] ^-{\alpha_{i,\lambda}^{R,+}} & \tdcrossrldd{i}{i} - \cuprb{i} + \acapl{i} } \]

\begin{enumerate}[{\bf i)}]
	\item For $\hil > 0$,
	
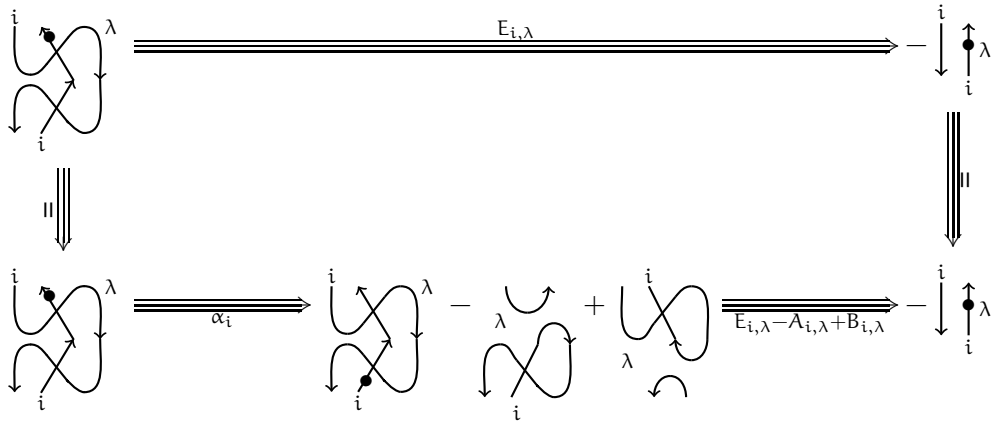
\[ \xymatrix@C=6em@R=2.5em{ 
		\tdcrossrldtest{i}{i} \ar@3 [rr] ^-{E_{i,\lambda}}
		\ar@3 [d] _-{\rotatebox{90}{=}}   & & - \diddownupdl{i}{i} \ar @3 [d] ^-{\rotatebox{90}{=}} \\
		\tdcrossrldtest{i}{i} \ar@3 [r] _-{\alpha_i} & \tdcrossrldd{i}{i} - \cuprb{i} + \acapl{i} \ar@3 [r] _-{E_{i,\lambda} - A_{i,\lambda} + B_{i,\lambda}} & - \diddownupdl{i}{i} } \]  
	using that for $\hil > 0$, $A_{i,\lambda}$ and $B_{i,\lambda}$  admit $0$ as $2$-target, and where the $3$-cell $E_{i,\lambda} - A_{i,\lambda} + B_{i,\lambda}$ is actually a composite of three rewriting steps of $\ER$.
	\item For $\hil = 0$, the $2$-cells
	\[ \cuprb{i} \qquad \raisebox{-3mm}{$\text{and}$} \qquad \acapl{i}  \] 
	both rewrites with respect to $\ER$ into
	\[ - \stdaempty{i}{-1} \]
	so that the $2$-target of the $3$-cell $E_{i,\lambda} - A_{i,\lambda} + B_{i,\lambda}$ is unchanged, which proves the confluence of the branching.
	\item For $\hil < 0$,
	\[ \xymatrix@R=3em@C=2.3em{
		\tdcrossrldtest{i}{i} 
		\ar@3[r] ^-{\Ceil{E}}
		\ar@3 [d] _-{\rotatebox{90}{=}}  & - \diddownupdl{i}{} + \sum\limits_{n=0}^{- \hil -1} \sum\limits_{r \geq 0} \stda{i}{n+1}{-n-r-2}{r} \ar@3[r] ^-{\Ceil{b'}} &  - \diddownupdl{i}{} + \sum\limits_{n=0}^{- \hil -1} \sum\limits_{r =0}^{- \hil} \stda{i}{n+1}{-n-r-2}{r} \ar@3[d] ^-{\rotatebox{90}{=}} \\
		\tdcrossrldtest{i}{i} \ar@3[r] _-{\alpha_i} & \tdcrossrldd{i}{i} - \cuprb{i} + \acapl{i} \ar@3[r] _-{\gamma}  & 
		- \diddownupdl{i}{i} +  \sum\limits_{n=1}^{-\hil} \sum\limits_{r=0}^{-\hil} \stda{i}{n}{-n-r-1}{r}  } \]
	where the $3$-cell $\gamma$ is defined as the following composite of $3$-cells of $(\ER)_3^\ell$:
	\begin{align*}
	\hspace{-1.4cm} \tdcrossrldd{i}{i} - \cuprb{i} + \acapl{i} & \overset{\Ceil{E} - \Ceil{A} + \Ceil{B}} \Rrightarrow  \: \sum\limits_{n=0}^{- \hil -1} \sum\limits_{r \geq 0} \stda{i}{n}{-n-r-2}{r+1} - \sum\limits_{n=0}^{-\hil} \lemmecb{i}{n} + \sum\limits_{n=0}^{-\hil} \lemmec{i}{n} \\
	& \overset{\Ceil{b'}}{\Rrightarrow} \sum\limits_{n=0}^{- \hil -1} \sum\limits_{r = 0}^{-\hil -1} \stda{i}{n}{-n-r-2}{r+1} - \sum\limits_{n=0}^{-\hil} \lemmecb{i}{n} + \sum\limits_{n=0}^{-\hil} \lemmec{i}{n} \\
	& =  \sum\limits_{n=0}^{- \hil -1} \sum\limits_{r = -1}^{-\hil -1} \stda{i}{n}{-n-r-2}{r+1} - \sum\limits_{n=0}^{-\hil} \lemmecb{i}{n} + \ruleahil{i}{\hil} \\
	& = \sum\limits_{n=1}^{-\hil} \sum\limits_{r=0}^{-\hil} \stda{i}{n}{-n-r-1}{r} 
	\end{align*}
	where the equalities are obtained from the linear structure using reindexations of sums.
\end{enumerate}


\subsubsection{Critical branchings $(\beta_{i,j}^{\lambda,+}, (i_1^0 \star_2 i_4^0)^- \cdot F_{i,j,\lambda})$}
\begin{enumerate}[{\bf i)}]
	\item First of all, let us consider the case where $i = j$, and thus the source of this branching rewrites to $0$ using $\beta_{i,\lambda^+}$. The other side of this critical branching is given by the following scheme of rewritings with respect to $\ER$:
	\begin{align*}
	\isocbb{i}{i} & \overset{F_{i,\lambda}}{\tfl} - \identu{i} \raisebox{-2mm}{$\begin{tikzpicture}[baseline = 0]
		\draw[-,thick,black] (0,0.4) to[out=180,in=90] (-.2,0.2);
		\draw[->,thick,black] (0.2,0.2) to[out=90,in=0] (0,.4);
		\draw[-,thick,black] (-.2,0.2) to[out=-90,in=180] (0,0);
		\draw[-,thick,black] (0,0) to[out=0,in=-90] (0.2,0.2);
		\node at (0,-.1) {$\scriptstyle{i}$};
		\end{tikzpicture}$} + \sum\limits_{n=0}^{\hil -1} \sum\limits_{r \geq 0} \stdacl{i}{r}{-n-r-2}{n} \\
	& \overset{b'_{i,\lambda}}{\tfl} - \identu{i} \raisebox{-2mm}{$\begin{tikzpicture}[baseline = 0]
		\draw[-,thick,black] (0,0.4) to[out=180,in=90] (-.2,0.2);
		\draw[->,thick,black] (0.2,0.2) to[out=90,in=0] (0,.4);
		\draw[-,thick,black] (-.2,0.2) to[out=-90,in=180] (0,0);
		\draw[-,thick,black] (0,0) to[out=0,in=-90] (0.2,0.2);
		\node at (0,-.1) {$\scriptstyle{i}$};
		\end{tikzpicture}$} + \sum\limits_{n=0}^{\hil -1} \sum\limits_{r = 0}^{\hil -1} \stdacl{i}{r}{-n-r-2}{n} \\
	& \equiv_E - \identu{i} \raisebox{-2mm}{$\begin{tikzpicture}[baseline = 0]
		\draw[-,thick,black] (0,0.4) to[out=180,in=90] (-.2,0.2);
		\draw[->,thick,black] (0.2,0.2) to[out=90,in=0] (0,.4);
		\draw[-,thick,black] (-.2,0.2) to[out=-90,in=180] (0,0);
		\draw[-,thick,black] (0,0) to[out=0,in=-90] (0.2,0.2);
		\node at (0,-.1) {$\scriptstyle{i}$};
		\end{tikzpicture}$} +  \sum\limits_{n=0}^{\hil -1} \sum\limits_{r = 0}^{\hil -1} \begin{tikzpicture}[baseline = 0,scale=1.2]
	\draw[<-,thick,black] (0,0.4) to[out=180,in=90] (-.2,0.2);
	\draw[-,thick,black] (0.2,0.2) to[out=90,in=0] (0,.4);
	\draw[-,thick,black] (-.2,0.2) to[out=-90,in=180] (0,0);
	\draw[-,thick,black] (0,0) to[out=0,in=-90] (0.2,0.2);
	\node at (0,-.1) {$\scriptstyle{i}$};
	\node at (-0.2,0.2) {$\color{black}\bullet$};
	\node at (-0.9,0.2) {$\color{black}\scriptstyle{-n-r-2}$};
	\end{tikzpicture}  \identdotsuf{i}{n+r} 
	\end{align*}
	Each summand in the above sum rewrites using the bubble slide $3$-cells as follows:
	\begin{align*}
	\begin{tikzpicture}[baseline = 0,scale=1.2]
	\draw[-,thick,black] (0,0.4) to[out=180,in=90] (-.2,0.2);
	\draw[->,thick,black] (0.2,0.2) to[out=90,in=0] (0,.4);
	\draw[-,thick,black] (-.2,0.2) to[out=-90,in=180] (0,0);
	\draw[-,thick,black] (0,0) to[out=0,in=-90] (0.2,0.2);
	\node at (0,-.1) {$\scriptstyle{i}$};
	\node at (-0.2,0.2) {$\color{black}\bullet$};
	\node at (-0.9,0.2) {$\color{black}\scriptstyle{-n-r-2}$};
	\end{tikzpicture}  \identdotsuf{i}{n+r}  & \overset{s_{i,\lambda}^-}{\tfl} \begin{tikzpicture}[baseline = 0,scale=0.8]
	\draw[->,thick,black] (0,-0.4) to (0,0.6);
	\node at (0,-0.6) {$\scriptstyle{i}$};
	\node at (0,0.1) {$\bullet$};
	\node at (-0.7,0.1) {$\scriptstyle{n+r+2}$};
	\end{tikzpicture} \negbubdf{i}{-n-r-2} - 2 \begin{tikzpicture}[baseline = 0,scale=0.8]
	\draw[->,thick,black] (0,-0.4) to (0,0.6);
	\node at (0,-0.6) {$\scriptstyle{i}$};
	\node at (0,0.1) {$\bullet$};
	\node at (-0.7,0.1) {$\scriptstyle{n+r+1}$};
	\end{tikzpicture} \negbubdf{i}{-n-r-1} + \begin{tikzpicture}[baseline = 0,scale=0.8]
	\draw[->,thick,black] (0,-0.4) to (0,0.6);
	\node at (0,-0.6) {$\scriptstyle{1}$};
	\node at (0,0.1) {$\bullet$};
	\node at (-0.6,0.1) {$\scriptstyle{n+r}$};
	\end{tikzpicture} \negbubdf{i}{-n-r}
	\end{align*}
	and we easily check that the above sums are telescopic, so that it remains the $2$-cell
	\[ 
	\sum\limits_{r=0}^{\hil -1} \left[ \identdotsusl{i}{r} \negbubd{i}{-r} - \identdotsufsl{i}{r+1} \negbubdf{i}{-r-1}
	+ \identdotsuffsl{i}{\hil + r +1} \negbubdff{i}{- \hil - r -1} - \identdotsuffsl{i}{\hil + r} \negbubdff{i}{- \hil - r}  \right] \]
	After simplification, it only remains \[ \identu{i} \raisebox{-2mm}{$\begin{tikzpicture}[baseline = 0]
		\draw[->,thick,black] (0,0.4) to[out=180,in=90] (-.2,0.2);
		\draw[-,thick,black] (0.2,0.2) to[out=90,in=0] (0,.4);
		\draw[-,thick,black] (-.2,0.2) to[out=-90,in=180] (0,0);
		\draw[-,thick,black] (0,0) to[out=0,in=-90] (0.2,0.2);
		\node at (0,-.1) {$\scriptstyle{i}$};
		\end{tikzpicture}$} \] and thus the starting diagram reduces to $0$, and this critical branching is confluent modulo $E$.
	\item Now, let us consider the case where $i \ne j$ and $i \cdot j = 0$. Let us at first notice that in that case, we have the following rewriting step given by a bubble slide $3$-cell:
	\[ \raisebox{-2mm}{$\begin{tikzpicture}[baseline = 0, scale = 1.2]
		\draw[-,thick,black] (0,0.4) to[out=180,in=90] (-.2,0.2);
		\draw[->,thick,black] (0.2,0.2) to[out=90,in=0] (0,.4);
		\draw[-,thick,black] (-.2,0.2) to[out=-90,in=180] (0,0);
		\draw[-,thick,black] (0,0) to[out=0,in=-90] (0.2,0.2);
		\node at (0,-.1) {$\scriptstyle{i}$};
		\end{tikzpicture}$} \identu{j} = 
	\begin{tikzpicture}[baseline = 0,scale=1.2]
	\draw[-,thick,black] (0,0.4) to[out=180,in=90] (-.2,0.2);
	\draw[->,thick,black] (0.2,0.2) to[out=90,in=0] (0,.4);
	\draw[-,thick,black] (-.2,0.2) to[out=-90,in=180] (0,0);
	\draw[-,thick,black] (0,0) to[out=0,in=-90] (0.2,0.2);
	\node at (0,-.1) {$\scriptstyle{i}$};
	\node at (-0.2,0.2) {$\color{black}\bullet$};
	\node at (-1.4,0.2) {$\color{black}\scriptstyle{- < h_i, \lambda + j_x > - 1 + \alpha}$}; 
	\end{tikzpicture} \identu{j} \overset{s_{i,\lambda}^-}{\tfl} \identu{j} \raisebox{-2mm}{$\begin{tikzpicture}[baseline = 0, scale=1.2]
		\draw[->,thick,black] (0,0.4) to[out=180,in=90] (-.2,0.2);
		\draw[-,thick,black] (0.2,0.2) to[out=90,in=0] (0,.4);
		\draw[-,thick,black] (-.2,0.2) to[out=-90,in=180] (0,0);
		\draw[-,thick,black] (0,0) to[out=0,in=-90] (0.2,0.2);
		\node at (0,-.1) {$\scriptstyle{i}$};
		\end{tikzpicture}$}  \]
	where $\alpha = < h_i, \lambda + j_x > +1$. Hence, the decreasing confluence of this critical branching is given by the following diagram:
	\[ \xymatrix@R=2.5em@C=3em{
		\isocba{j}{i} \ar@3[r] ^-{\beta_{i,j,\lambda}^+} \ar@3 [d] _-{} & \begin{tikzpicture}[baseline = 0, scale = 1.2]
		\draw[-,thick,black] (0,0.4) to[out=180,in=90] (-.2,0.2);
		\draw[->,thick,black] (0.2,0.2) to[out=90,in=0] (0,.4);
		\draw[-,thick,black] (-.2,0.2) to[out=-90,in=180] (0,0);
		\draw[-,thick,black] (0,0) to[out=0,in=-90] (0.2,0.2);
		\node at (0,-.1) {$\scriptstyle{i}$};
		\end{tikzpicture} \identu{j} \ar@3 [r] ^-{s_{i,\lambda}^-} & \identu{j} \raisebox{-2mm}{$\begin{tikzpicture}[baseline = 0, scale=1.2]
			\draw[->,thick,black] (0,0.4) to[out=180,in=90] (-.2,0.2);
			\draw[-,thick,black] (0.2,0.2) to[out=90,in=0] (0,.4);
			\draw[-,thick,black] (-.2,0.2) to[out=-90,in=180] (0,0);
			\draw[-,thick,black] (0,0) to[out=0,in=-90] (0.2,0.2);
			\node at (0,-.1) {$\scriptstyle{i}$};
			\end{tikzpicture}$} \ar@3 [d] ^-{\fleq} \\
		\isocbb{j}{i} \ar@3 [rr] _-{F_{i,j,\lambda}} & & \identu{j} \raisebox{-2mm}{$\begin{tikzpicture}[baseline = 0, scale=1.2]
			\draw[->,thick,black] (0,0.4) to[out=180,in=90] (-.2,0.2);
			\draw[-,thick,black] (0.2,0.2) to[out=90,in=0] (0,.4);
			\draw[-,thick,black] (-.2,0.2) to[out=-90,in=180] (0,0);
			\draw[-,thick,black] (0,0) to[out=0,in=-90] (0.2,0.2);
			\node at (0,-.1) {$\scriptstyle{i}$};
			\end{tikzpicture}$} } \]
	\item Let us now consider the last case where $i \ne j$ and $i \cdot j = -1$. In that case, we have the following rewriting step in $\ER$:
	\[ 
	\isocba{j}{i} \overset{\beta_{i,j,\lambda}^+}{\tfl} \raisebox{-2mm}{$\begin{tikzpicture}[baseline = 0, scale = 1.2]
		\draw[-,thick,black] (0,0.4) to[out=180,in=90] (-.2,0.2);
		\draw[->,thick,black] (0.2,0.2) to[out=90,in=0] (0,.4);
		\draw[-,thick,black] (-.2,0.2) to[out=-90,in=180] (0,0);
		\draw[-,thick,black] (0,0) to[out=0,in=-90] (0.2,0.2);
		\node at (0,-.1) {$\scriptstyle{i}$};
		\node at (0.2,0.2) {$\bullet$}; 
		\end{tikzpicture}$} \identu{j} + \raisebox{-2mm}{$\begin{tikzpicture}[baseline = 0, scale = 1.2]
		\draw[-,thick,black] (0,0.4) to[out=180,in=90] (-.2,0.2);
		\draw[->,thick,black] (0.2,0.2) to[out=90,in=0] (0,.4);
		\draw[-,thick,black] (-.2,0.2) to[out=-90,in=180] (0,0);
		\draw[-,thick,black] (0,0) to[out=0,in=-90] (0.2,0.2);
		\node at (0,-.1) {$\scriptstyle{i}$};
		\end{tikzpicture}$} \identdotu{j} \]
	Using the bubble slide $3$-cells, the first summand (resp. the second summand) rewrites into 
	\[ \sum\limits_{f = 0}^{\hil + 1} (-1)^f \identdotsusl{j}{f} \negbubd{i}{-f} \quad \left( \text{  resp. $\sum\limits_{f = 0}^{\hil + 1} (-1)^f \identdotsufsl{j}{f+1} \negbubd{i}{-f-1}$ } \right)  \]
	so that the sum is equal to \[ \identu{j} \raisebox{-2mm}{$\begin{tikzpicture}[baseline = 0, scale=1.2]
		\draw[->,thick,black] (0,0.4) to[out=180,in=90] (-.2,0.2);
		\draw[-,thick,black] (0.2,0.2) to[out=90,in=0] (0,.4);
		\draw[-,thick,black] (-.2,0.2) to[out=-90,in=180] (0,0);
		\draw[-,thick,black] (0,0) to[out=0,in=-90] (0.2,0.2);
		\node at (0,-.1) {$\scriptstyle{i}$};
		\end{tikzpicture}$},  \]
	and this critical branching is confluent modulo $E$.
\end{enumerate}

\subsubsection{Critical branchings $(\alpha_{i,\lambda}^{R,+}, (i_1^0 \star_2 i_4^0)^- \star_2 i_3^2 \star_2 i_1^2 \cdot F_{i,j,\lambda})$}
When $i \ne j$ and $i \cdot j = 0$:
\[
\xymatrix@R=2.5em@C=3em{\isocbadot{i}{j} \ar@3 [d] _-{i_1^0 \star_2 i_4^0} \ar@3 [r] ^-{\gamma_{r,+}} & \isocbadotbbslb{i}{j} \ar@3 [r] ^-{\alpha_{i,j}^{R,+}  } &   \raisebox{-2mm}{$\begin{tikzpicture}[baseline = 0, scale=1.2]
  \draw[->,thick,black] (0,0.4) to[out=180,in=90] (-.2,0.2);
  \draw[-,thick,black] (0.2,0.2) to[out=90,in=0] (0,.4);
 \draw[-,thick,black] (-.2,0.2) to[out=-90,in=180] (0,0);
  \draw[-,thick,black] (0,0) to[out=0,in=-90] (0.2,0.2);
  \node at (0.2,0.2) {$\bullet$};
  \node at (0,-0.2) {$\scriptstyle{j}$};
 \end{tikzpicture}$} \identu{i}  \ar@3 [rr] ^-{s^-_{i,j,\lambda, - \la h_i, \lambda + j_x \ra + 2}} & & \identusl{i} \raisebox{-2mm}{$\begin{tikzpicture}[baseline = 0, scale=1.2]
  \draw[->,thick,black] (0,0.4) to[out=180,in=90] (-.2,0.2);
  \draw[-,thick,black] (0.2,0.2) to[out=90,in=0] (0,.4);
 \draw[-,thick,black] (-.2,0.2) to[out=-90,in=180] (0,0);
  \draw[-,thick,black] (0,0) to[out=0,in=-90] (0.2,0.2);
  \node at (0.2,0.2) {$\bullet$};
    \node at (0,-0.2) {$\scriptstyle{j}$};
  \node at (0.4,0.4) {$\scriptstyle{\lambda}$};
 \end{tikzpicture}$} \ar@3 [d] ^-{\fleq} \\ \isocbbdotsl{i}{j} \ar@3 [rrrr] _-{\Fil} & & & & \identusl{i} \raisebox{-2mm}{$\begin{tikzpicture}[baseline = 0, scale=1.2]
  \draw[->,thick,black] (0,0.4) to[out=180,in=90] (-.2,0.2);
  \draw[-,thick,black] (0.2,0.2) to[out=90,in=0] (0,.4);
 \draw[-,thick,black] (-.2,0.2) to[out=-90,in=180] (0,0);
  \draw[-,thick,black] (0,0) to[out=0,in=-90] (0.2,0.2);
  \node at (0.2,0.2) {$\bullet$};
      \node at (0,-0.2) {$\scriptstyle{j}$};
 \end{tikzpicture}$} }
\]
When $i \cdot j = -1$, we have
\[
\xymatrix@R=2.5em@C=3em{\isocbadot{i}{j} \ar@3 [d] _-{i_1^0 \star_2 i_4^0} \ar@3 [r] ^-{\gamma_{r,+}} & \isocbadotbbslb{i}{j} \ar@3 [r] ^-{\alpha_{i,j}^{R,+}  } &   \raisebox{-2mm}{$\begin{tikzpicture}[baseline = 0, scale=1.2]
  \draw[->,thick,black] (0,0.4) to[out=180,in=90] (-.2,0.2);
  \draw[-,thick,black] (0.2,0.2) to[out=90,in=0] (0,.4);
 \draw[-,thick,black] (-.2,0.2) to[out=-90,in=180] (0,0);
  \draw[-,thick,black] (0,0) to[out=0,in=-90] (0.2,0.2);
  \node at (0.2,0.2) {$\bullet$};
  \node at (0,-0.2) {$\scriptstyle{j}$};
 \end{tikzpicture}$} \identdotu{i}  + \raisebox{-2mm}{$\begin{tikzpicture}[baseline = 0, scale=1.2]
  \draw[->,thick,black] (0,0.4) to[out=180,in=90] (-.2,0.2);
  \draw[-,thick,black] (0.2,0.2) to[out=90,in=0] (0,.4);
 \draw[-,thick,black] (-.2,0.2) to[out=-90,in=180] (0,0);
  \draw[-,thick,black] (0,0) to[out=0,in=-90] (0.2,0.2);
  \node at (0.2,0.2) {$\bullet$};
  \node at (0.35,0.2) {$\scriptstyle{2}$};
  \node at (0,-0.2) {$\scriptstyle{j}$};
 \end{tikzpicture}$} \identu{i}  \\ \isocbbdotsl{i}{j} \ar@3 [rr] _-{\Fil} & & \identusl{i} \raisebox{-2mm}{$\begin{tikzpicture}[baseline = 0, scale=1.2]
  \draw[->,thick,black] (0,0.4) to[out=180,in=90] (-.2,0.2);
  \draw[-,thick,black] (0.2,0.2) to[out=90,in=0] (0,.4);
 \draw[-,thick,black] (-.2,0.2) to[out=-90,in=180] (0,0);
  \draw[-,thick,black] (0,0) to[out=0,in=-90] (0.2,0.2);
  \node at (0.2,0.2) {$\bullet$};
      \node at (0,-0.2) {$\scriptstyle{j}$};
 \end{tikzpicture}$} }
\]

Using the bubble slide $3$-cells $s_{i,j,\lambda, \hil + 1}^-$ and $s_{i,j,\lambda, \hil + 12}^-$ respectively, we get that
\[ 
\raisebox{-2mm}{$\begin{tikzpicture}[baseline = 0, scale=1.2]
  \draw[->,thick,black] (0,0.4) to[out=180,in=90] (-.2,0.2);
  \draw[-,thick,black] (0.2,0.2) to[out=90,in=0] (0,.4);
 \draw[-,thick,black] (-.2,0.2) to[out=-90,in=180] (0,0);
  \draw[-,thick,black] (0,0) to[out=0,in=-90] (0.2,0.2);
  \node at (0.2,0.2) {$\bullet$};
  \node at (0,-0.2) {$\scriptstyle{j}$};
 \end{tikzpicture}$} \identdotu{i}  \Rrightarrow \sum\limits_{f=0}^{\hil + 1} (-1)^f \identdotsufsl{j}{f+1} \negbubd{i}{-f} \quad \text{and} \quad
\raisebox{-2mm}{$\begin{tikzpicture}[baseline = 0, scale=1.2]
  \draw[->,thick,black] (0,0.4) to[out=180,in=90] (-.2,0.2);
  \draw[-,thick,black] (0.2,0.2) to[out=90,in=0] (0,.4);
 \draw[-,thick,black] (-.2,0.2) to[out=-90,in=180] (0,0);
  \draw[-,thick,black] (0,0) to[out=0,in=-90] (0.2,0.2);
  \node at (0.2,0.2) {$\bullet$};
  \node at (0.35,0.2) {$\scriptstyle{2}$};
  \node at (0,-0.2) {$\scriptstyle{j}$};
 \end{tikzpicture}$} \identu{i}  \Rrightarrow \sum\limits_{f=0}^{\hil + 2} (-1)^f \identdotsusl{j}{f} \negbubd{i}{1-f}  \] 
and one then proves the confluence of this critical branchings modulo using reindexations of the sums.

\noindent In the case $i=j$, we get the following situation:
\[
\xymatrix@R=2.5em@C=2em{\isocbadotsl{i}{i} \ar@3 [d] _-{i_1^0 \star_2 i_4^0} \ar@3 [r] ^-{\gamma_{r,+}} & \isocbadotbbsl{i}{i} - \raisebox{-4mm}{$\isocbadotbcsl{i}$} \ar@3 [r] ^-{\beta_{i}^+ - (i_1^0)^- \cdot \Cil } &  \sum\limits_{n=0}^{\hil} \ruleciso{i}{n}  \; \equiv_E \; \sum\limits_{n=0}^{\hil} \negbubdfsl{i}{-n-1}  \identdotsu{i}{n} \\ \isocbbdotsl{}{} \ar@3 [r] _-{\Fil}  & - \identusl{} \raisebox{-2mm}{$\begin{tikzpicture}[baseline = 0, scale=1.2]
  \draw[->,thick,black] (0,0.4) to[out=180,in=90] (-.2,0.2);
  \draw[-,thick,black] (0.2,0.2) to[out=90,in=0] (0,.4);
 \draw[-,thick,black] (-.2,0.2) to[out=-90,in=180] (0,0);
  \draw[-,thick,black] (0,0) to[out=0,in=-90] (0.2,0.2);
  \node at (0.2,0.2) {$\bullet$};
 \end{tikzpicture}$} 
 + \sum\limits_{n=0}^{\hil-1} \sum\limits_{r=0}^{\hil - 1} \stdbiso{i}{n+r+1}{-n-r-2} }
\]
Because of the degree conditions on bubbles $3$-cells, the last summand in the last term of the bottom line of this critical branching modulo is equal to $0$ whenever $ n + r > \hil -1$. As a consequence, it reduces to 
\[ \sum\limits_{n+r=0}^{\hil -1} \negbubdfsl{i}{-n-r-2}  \identdotsuf{i}{\; \; \; \; n+r+1} \]
and one then proves the confluence modulo of this branchings using a reindexation of this sum and the bubble slide $3$-cells as in the previous proof of confluence of critical branching.

\subsubsection{Critical branchings $(\gamma_{j,i,j}^{\lambda,+} , (i_1^0 \star_2 i_4^0)^- \cdot F_{i,j,\lambda})$}
\[ \xymatrix@R=2.5em@C=1em{
\raisebox{3mm}{$\isocbacr{i}{j}{i}$} \ar@3 [d] _-{(i_1^0)^- \star_2 (i_4^0)^-}
\ar@3[r] ^-{\beta^+} & \raisebox{3mm}{$\isocbacrb{i}{j}{i}$} + \delta_{i  \cdot j = -1} \raisebox{-2mm}{$\negbubsl{i}$} \identusl{j} \identu{i} \: \equiv_E \: \raisebox{3mm}{$\isocbbcrb{i}{i}{j}$} + \delta_{i \cdot j = -1} \raisebox{-2mm}{$\negbubsl{i}$} \identusl{j} \identu{i} \\
\raisebox{2mm}{$\isocbbcrsl{i}{}$} \ar@3 [r] _-{\Fil}  & \identusl{i} \isocbadotbcsl{j} \ar@3 [r] _-{(i_1^0)^- \cdot \Cil} &  \sum\limits_{n=0}^{\hil} \identusl{i} \: \ruleciso{j}{n} 
}
\]
 Using the $3$-cell $\Cil$, the term in the top line reduces to 
 \begin{equation} 
 \label{E:ReductionCriticalBranchingFYB} \sum\limits_{n=0}^{\hil} \raisebox{3mm}{$\isocbbcrcslc{i}{j}{i}{n}$} \equiv_E  
 \sum\limits_{n=0}^{\hil} \; \raisebox{-13mm}{$\isocbbcrcsld{i}{j}{n}$} \; \overset{\alpha_{i,j}^{L,n,+}}{\Rrightarrow}
\sum\limits_{n=0}^{\hil} \raisebox{-13mm}{$\isocbbcrcsle{i}{j}{n}$} 
\end{equation}
 When $i \cdot j =0$, this rewrites using $\beta_{i,j}^+$ to 
 \[ \sum\limits_{n=0}^{\hil} \raisebox{-13mm}{$\isocbbcrcslf{i}{j}{n}$} \]
 so that this branching is confluent modulo $E$. In the case $i \cdot j = 1$, this rewrites to
 \[ \sum\limits_{n=0}^{\hil} \raisebox{-13mm}{$\isocbbcrcslfb{i}{j}{n}$} + \sum\limits_{n=0}^{\hil} \raisebox{-13mm}{$\isocbbcrcslfc{i}{j}{n}{n+1}$}. \]
 Then note that 
 \[ \sum\limits_{n=0}^{\hil} \raisebox{-13mm}{$\isocbbcrcslfc{i}{j}{n}{n+1}$} + \raisebox{-2mm}{$\negbubsl{i}$} \identusl{j} \identu{i} = \sum\limits_{n=0}^{\hil-1} \raisebox{-13mm}{$\isocbbcrcslfd{i}{j}{n}{n}$} \] so that the top line of this branching rewrites to
 \[ \sum\limits_{n=0}^{\hil} \raisebox{-13mm}{$\isocbbcrcslfb{i}{j}{n}$} + \sum\limits_{n=0}^{\hil-1} \raisebox{-13mm}{$\isocbbcrcslfd{i}{j}{n}{n}$} \]
 and we check the confluence modulo of this branching using the bubble slide $3$-cells, the dots on the leftmost strand being cancelled by the $3$-cells $s_{i,j,\lambda,\alpha}^-$ for $i \cdot j = -1$.

\subsection{Branchings between isomorphism and $\mathfrak{sl_2}$ relations}
\label{A:IsoCriticalBranchings}

\subsubsection{Critical branchings between types $A$ and $C$}
We prove that for any $i \in I$ and $\lambda \in X$, and for any value of $\hil$, the critical branchings $(\Ail,\Cil)$ are confluent modulo $E$.

\begin{enumerate}[{\bf i)}]
	\item For $\hil < 0$,
	\[ \hspace{-1cm} \xymatrix{ \bcritacxy{i} 
		\ar@3[rrrr] ^-{\Cil}
		\ar@3 [d] _-{\fleq} & & & & 0 \ar@3 [d] ^-{\fleq} \\
		\bcritacxy{i} \ar@3 [r] _-{\Ail} & - \sum\limits_{n=0}^{- \hil} \odbbub{i}{n}{-n-1}  \ar@3 [r] _-{\Ceil{b}^{0,n}} & \odbbubf{i}{-\hil}{\hil -1} + \odbbubf{i}{\hil}{- \hil -1} \ar@3 [r] _-{} & \negbubdf{i}{-\hil} + \posbubdf{i}{\hil} \ar@3[r] _-{I_1} & 0 } \]
	
	\item For $\hil = 0$, 
	\[ \hspace{-1cm} 
	\xymatrix@C=4em{ \bcritacxy{i} 
		\ar@3 [r] ^-{\Cil}
		\ar@3[d] _-{\fleq} & \odbbubdup{i}{-1} \ar@3 [r] ^-{\Ceil{c}^{1,- \hil -1}} & \posbubwdot{i} \ar@3[r] ^-{I_1} & - \negbubwdot{i} \ar@3 [d] ^-{\fleq} \\
		\bcritacxy{i} \ar@3 [r] _-{\Ail} & - \odbbubdd{i}{-1}  \ar@3 [rr] _-{\Ceil{b}^{1, \hil -1}} & & - \negbubwdot{i} } \]
	\item For $\hil > 0$, the computation is similar to the case $\hil <0$, except that the source $2$-cell reduces to $0$ by $\Ail$ instead of $\Cil$.
\end{enumerate}

\subsubsection{Critical branchings between types $A$ and $F$}

\begin{enumerate}[{\bf i)}]
	\item For $\hil < 0$,
	
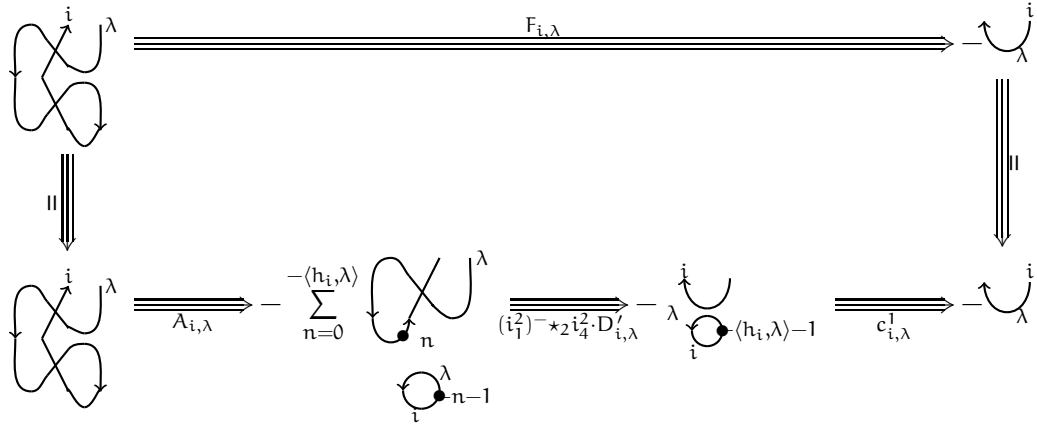
\[ \xymatrix@R=2.5em@C=4em{ \bcritaf{i} 
		\ar@3 [rrr] ^-{\Fil}
		\ar@3 [d] _-{\fleq} & & & - \cupl{i} \ar@3[d] ^-{\fleq}\\
		\bcritaf{i} \ar@3 [r] _-{\Ail} & - \sum\limits_{n=0}^{- \hil}  \tfishdlpbub{i}{n} \ar@3 [r] _-{(i_1^2)^- \star_2 i_4^2 \cdot \Ceil{D'}} & - \ruledhil{i}{-\hil -1} \ar@3 [r] _-{\Ceil{c}^1} & - \cupl{i} } \]
	where $\Ceil{D'}$ is a composite of $n$ positive $3$-cells of $(\ER)_3^\ell$, which represents the sum $D'_{i,\lambda,1} + \dots D'_{i,\lambda, - \hil}$, where the $3$-cell $D'_{i,\lambda,k}$ is defined for any $1 \leq k \leq - \hil$ in Appendix \ref{SS:FurtherCells}.
	
	\item For $\hil = 0$, 
	\[ \xymatrix@R=2.5em@C=4em{ \bcritaf{i} 
		\ar@3 [rrrr] ^-{\Fil}
		\ar@3 [d] _-{\fleq} & & & &  - \cupl{i} \ar@3[d] ^-{\fleq}\\
		\bcritaf{i} \ar@3 [r] _-{\Ail} & - \tfishdlpbubzero{i}{-1} \ar@3 [r] _-{\Ceil{b}^1} & - \tfishdl{i} \ar@3 [r] _-{\Dil} & - \ruledzero{i}{-1} \ar@3 [r] _-{\Ceil{c}^1} & - \cupl{i} } \]
	\item For $\hil > 0$,
	\[ \hspace{-1.5cm} \xymatrix@R=2.5em@C=1.5em{
		\bcritaf{i}
		\ar @3 [d] _-{\fleq}
		\ar@3 [r]^-{\Fil} & - \cupl{i} + \sum\limits_{n=0}^{\hil -1} \sum\limits_{r \geq 0} \stdbdown{i}{r}{-n-r-2}{n} \ar@3 [rr]^-{\Ceil{b}^{0,n}} & & - \cupl{i} + \sum\limits_{r \geq 0} \stdbdownf{i}{r}{ - \hil - r -1}{ \hil -1} \ar@3 [d] ^-{\fleq} \\
		\bcritaf{i} \ar@3[r] ^-{\Ail} & 0  & - \cupl{i} + \sum\limits_{r \geq 0} \ruledf{i}{r}{- \hil - r -1} \ar@3 [l] ^-{\Ceil{c}} & - \cupl{i} + \sum\limits_{r \geq 0} \stdbdownf{i}{r}{ - \hil - r -1}{ \hil -1} \ar@3 [l] ^-{\Ceil{b}^1}     } \]
	where the cell $\Ceil{c}$ is defined as the composite of rewriting steps of $\ER$ given by $\Ceil{c}^{1,-\hil-1} + \Ceil{c}^{0,-\hil-2} + \dots $, using degree condition $3$-cells on bubbles to prove that the only term remaining is for $r=0$, and is $\cupl{i}$.
\end{enumerate}

\subsubsection{Critical branchings between types $B$ and $D$}
\begin{enumerate}[{\bf i)}]
	\item For $\hil < 0$,
	
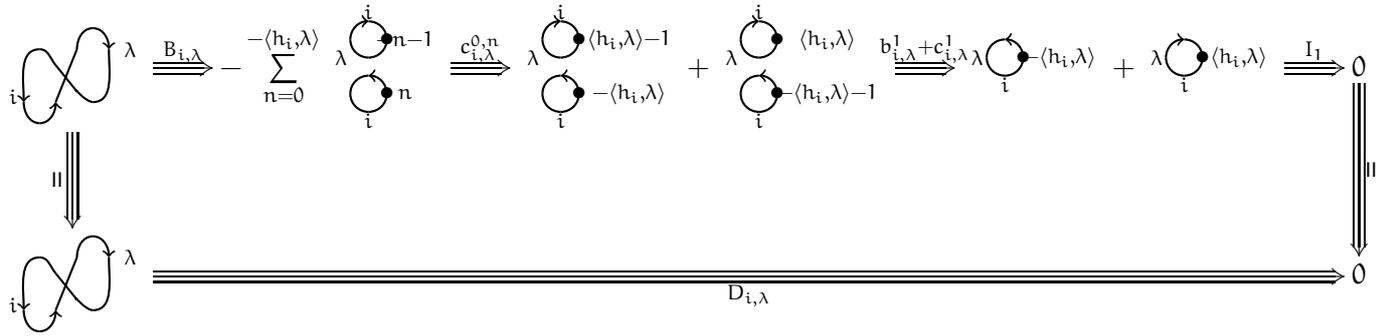
\[ \hspace{-1cm} \xymatrix@R=2.5em@C=2em{
		\bcritbd{i} \ar@3 [r] ^-{\Bil}
		\ar@3 [d] _-{\fleq} & - \somme{n=0}{- \hil} \odbbubpos{i}{-n-1}{n} \ar@3 [r] ^-{\Ceil{c}^{0,n}} & \odbbubposf{i}{\hil -1}{- \hil} + \odbbubposf{i}{\hil}{- \hil -1} \ar@3 [r] ^-{\Ceil{b}^1 + \Ceil{c}^1} & \negbubdf{i}{- \hil} + \posbubdf{i}{\hil} \ar@3 [r] ^-{I_1} & 0 \ar@3 [d] ^-{\fleq} \\
		\bcritbd{i} \ar@3 [rrrr] _-{\Dil} & & & & 0} \]
	\item For $\hil = 0$, 
	
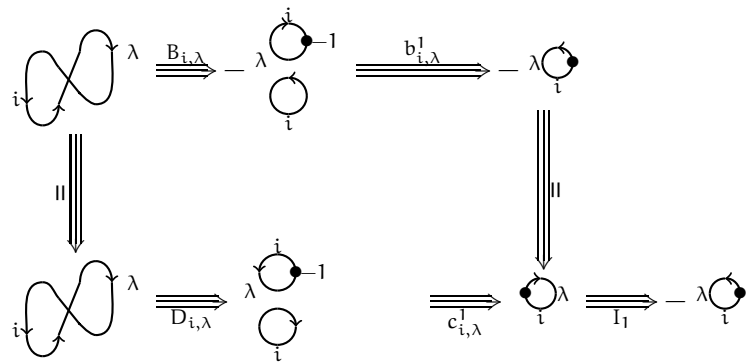
\[ \xymatrix@R=2.5em@C=2em{
		\bcritbd{i} \ar@3 [r] ^-{\Bil}
		\ar@3 [d] _-{\fleq} & - \odbbubposup{i}{-1} \ar@3 [rr] ^-{\Ceil{b}^1} & &  - \negbub{i}   \ar@3 [d] ^-{\fleq} \\
		\bcritbd{i} \ar@3 [r] _-{\Dil} & \odbbubdup{i}{-1} & \ar@3 [r] _-{\Ceil{c}^1}   & \posbub{i} \ar@3 [r] _-{I_1}  & - \negbub{i} } \]
	\item For $\hil > 0$,
	
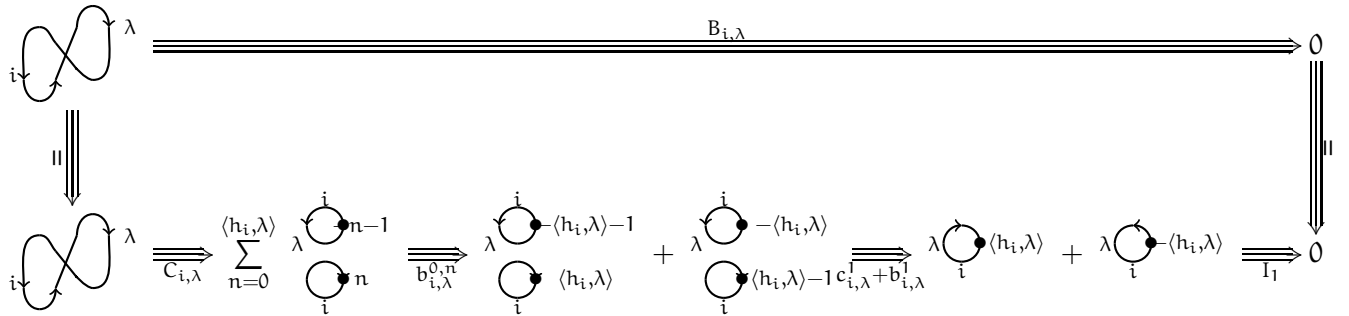
\[ \hspace{-1cm}
	\xymatrix@R=2.5em@C=2em{
		\bcritbd{i}
		\ar@3 [rrrr] ^-{\Bil} \ar@3 [d] _-{\fleq} & & & & 0 \ar@3 [d] ^-{\fleq} \\
		\bcritbd{i} 
		\ar@3 [r] _-{\Cil} & \somme{n=0}{\hil} \odbbub{i}{-n-1}{n} \ar@3 [r] _-{\Ceil{b}^{0,n}} &  \odbbubf{i}{- \hil -1}{\hil} + \odbbubf{i}{- \hil}{\hil -1} \ar@3 [r] _-{\Ceil{c}^1 + \Ceil{b}^1} & \posbubdf{i}{\hil} + \negbubdf{i}{- \hil} \ar@3 [r] _-{I_1} & 0 } \] 
\end{enumerate}

\subsubsection{Critical branchings between types $B$ and $F$}

\begin{enumerate}[{\bf i)}]
	\item For $\hil < 0$,
	
	
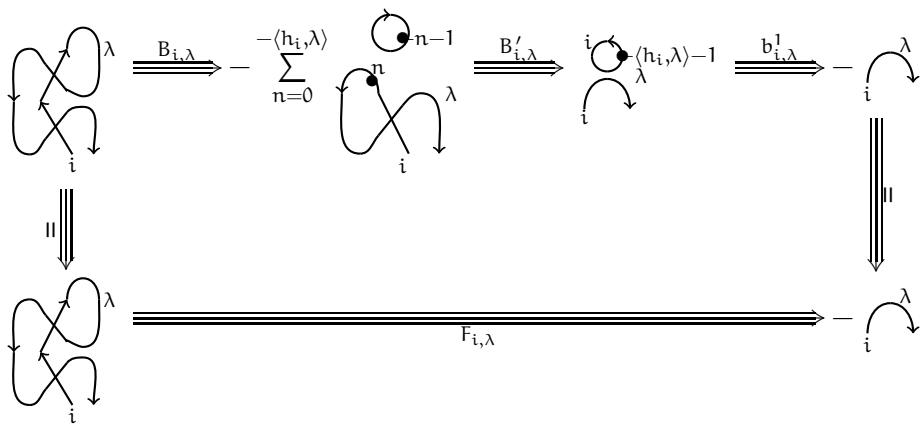
\[ \xymatrix@R=2.5em@C=3em{
		\bcritbf{i} \ar@3 [r] ^-{\Bil} \ar @3 [d] _-{\fleq} & - \sum\limits_{n=0}^{- \hil} \tfishurpbub{i}{n} \ar@3 [r] ^-{\Ceil{B'}} & \rulechil{i}{- \hil-1} \ar@3 [r] ^-{\Ceil{b}^1} & - \capr{i} \ar@3[d] ^-{\fleq} \\
		\bcritbf{i} \ar@3 [rrr] _-{\Fil} & & & - \capr{i}   } \]
	where $\Ceil{B'}$ is the positive $3$-cell of $(\ER)_3^\ell$ corresponding to $B'_{i,\lambda,0} + \dots + B'_{i ,\lambda, - \hil}$ where each $3$-cell $B'_{i,\lambda,k}$ for $0 \leq k \leq - \hil$ is defined in Appendix \ref{SS:FurtherCells}. 
	\item For $\hil = 0$,
	\[ \xymatrix@R=2.5em@C=3em{
		\bcritbf{i} \ar@3 [r] ^-{\Bil} \ar @3 [d] _-{\fleq} & - \tfishurpbubup{i}{-1} \ar @3[r] ^-{\Ceil{b}^1} & - \tfishur{i} \ar@3[r] ^-{\Cil} & - \rulechilb{i}{-1} \ar@3 [r] ^-{\Ceil{c}^1} & - \capr{i} \ar @3 [d] _-{\fleq} \\
		\bcritbf{i} \ar@3 [rrrr] _-{\Fil} & & & & - \capr{i} } \]
	\item For $\hil > 0$,
	\[ \hspace{-2cm} \xymatrix@R=2em@C=1.4em{
		\bcritbf{i} \ar@3[r] ^-{\Bil}
		\ar @3 [d] _-{\fleq} & 0 & - \capr{i} + \sum\limits_{n=0}^{\hil -1} \rulecf{i}{n}{-n-\hil-1} \ar@3 [l] _-{\Ceil{b}} & \sum\limits_{n=0}^{\hil -1} \sum\limits_{r=0}^{\hil -1} \stdbup{i}{r}{-n-r-2}{n}  \ar @3[d] ^-{\fleq} \ar@3 [l] _-{\Ceil{b}^{0,r}}  \\
		\bcritbf{i} \ar @3[r] _-{\Fil} & - \capr{i} + \sum\limits_{n=0}^{\hil-1} \sum\limits_{r \geq 0} \stdbup{i}{r}{-n-r-2}{n} \ar@3 [rr] _-{\Ceil{b'}} & &  \sum\limits_{n=0}^{\hil -1} \sum\limits_{r=0}^{\hil -1} \stdbup{i}{r}{-n-r-2}{n}   } \]
	where $\Ceil{b}$ is the $3$-cell of $(\ER)_3^\ell$ reducing each bubble by $\Ceil{b}^{0,-n-\hil -1}$ into $0$ when $ n \ne 0$ and by $\Ceil{b}^{1}$ into $1_{1_{\lambda}}$ when $n=0$.
\end{enumerate}

\subsubsection{Critical branchings between types $E$ and $D$}
\begin{enumerate}[{\bf i)}]
	\item For $\hil <0$, 
	
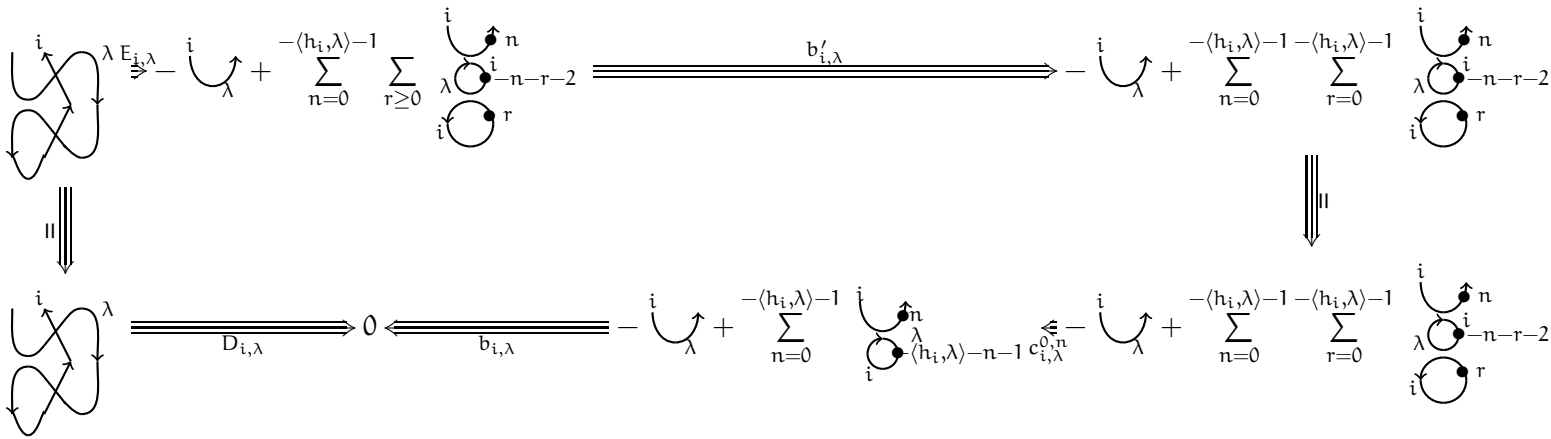
\[ \hspace{-2.8cm} \xymatrix@R=2em@C=0.5em{
		\bcritde{i} \ar @3 [r] ^-{\Eil} \ar@3 [d] _-{\fleq} & - \cupr{i} + \sum\limits_{n=0}^{- \hil -1} \sum\limits_{r \geq 0} \stdadown{i}{n}{-n-r-2}{r} \ar @3 [rr] ^-{\Ceil{b'}} & & - \cupr{i} +  \sum\limits_{n=0}^{- \hil -1} \sum\limits_{r = 0}^{- \hil -1} \stdadown{i}{n}{-n-r-2}{r} \ar@3 [d] ^-{\fleq} \\
		\bcritde{i} \ar @3 [r] _-{\Dil} & 0 & - \cupr{i} + \sum\limits_{n=0}^{-\hil -1} \ruleaf{i}{n}{-\hil -n-1} \ar @3 [l] ^-{\Ceil{b}} & - \cupr{i} +  \sum\limits_{n=0}^{- \hil -1} \sum\limits_{r = 0}^{- \hil -1} \stdadown{i}{n}{-n-r-2}{r} \ar @3 [l] ^-{\Ceil{c}^{0,n}} } \]
	where $\Ceil{b}$ is the $3$-cell of $(\ER)_3^\ell$ reducing each bubble by $\Ceil{b}^{0,-n- \hil -1}$ into $0$ when $ n \ne 0$ and by $\Ceil{b}^{1}$ into $1_{1_{\lambda}}$ when $n=0$.
	\item For $\hil=0$,
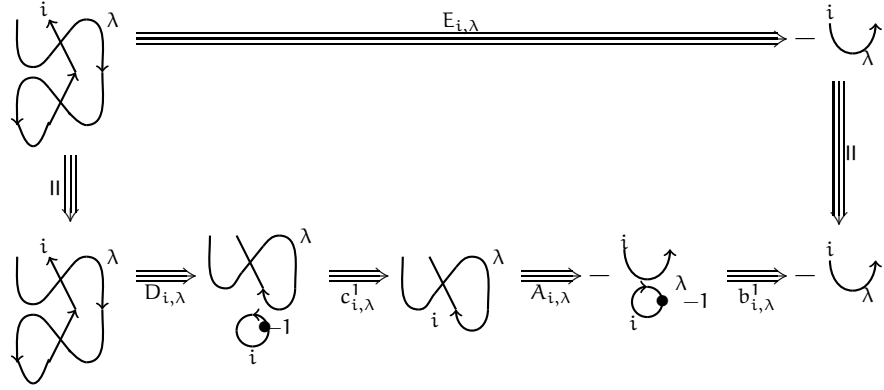
	\[ \raisebox{+4mm}{$\xymatrix@R=2.1em@C=2em{
		\bcritde{i} \ar@3 [rrrr] ^-{\Eil}
		\ar@3 [d] _-{\fleq} & & & & - \cupr{i} \ar@3 [d] ^-{\fleq} \\
		\bcritde{i} \ar@3 [r] _-{\Dil} & \tfishdrpbubhil{i}{-1} \ar@3 [r] _-{\Ceil{c}^1} & \tfishdr{i} \ar@3 [r] _-{\Ail} & - \ruleahilb{i}{\hspace{-0.3cm} -1} \ar@3 [r] _-{\Ceil{b}^1} & - \cupr{i} }$} \]
	\item For $\hil > 0$,
\[
	\raisebox{-4mm}{$
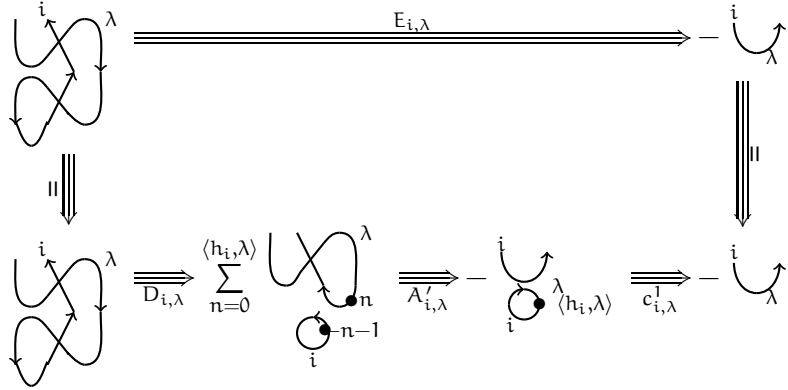
\xymatrix@R=2.1em@C=2em{
		\bcritde{i}
		\ar@3 [rrr] ^-{\Eil}
		\ar@3 [d] _-{\fleq}  & & & - \cupr{i} \ar@3 [d]^-{\fleq}  \\
		\bcritde{i} 
		\ar@3 [r] _-{\Dil} & \somme{n=0}{\hil} \tfishdrpbub{i}{n} \ar@3 [r] _-{\Ceil{A'}} & - \ruleahilb{i}{\hil} \ar@3 [r] _-{\Ceil{c}^1} &    - \cupr{i} 
	}$} \]
	where the $3$-cell $\Ceil{A'}$ is defined as the $3$-cell  $ A'_{i,\lambda ,0} + \dots + A'_{i,\lambda,\hil}$, where each $3$-cell $A'_{i,\lambda,k}$ for $0 \leq k \leq \hil$ is defined in Appendix \ref{SS:FurtherCells} and has for $2$-target $0$ if $n < \hil$ and $- \cupr{i}$ if $n = \hil$.
\end{enumerate}

\subsubsection{Critical branchings between types $C$ and $E$}

\begin{enumerate}[{\bf i)}]
	\item For $\hil < 0$,
	\vskip-10pt
	\[ \hspace{-1.5cm}
	\xymatrix@R=2.1em@C=1.5em{
		\bcritce{i} \ar@3 [r] ^-{\Cil} 
		\ar @3 [d] _-{\fleq} & 0 &  -\capl{i} + \rulebhil{i}{\hil -1} \ar@3 [l] _-{\Ceil{b}^1} &   - \capl{i} + \somme{r \geq 0}{} \rulebf{i}{r}{\hil - r -1} \ar@3 [l] _-{\Ceil{b}} 
		\ar@3 [d] ^-{\fleq}                      \\
		\bcritce{i} \ar@3[r] _-{\Eil} & - \capl{i} + \somme{n=0}{- \hil -1} \somme{r \geq 0}{} \stdaup{i}{n}{-n-r-2}{r} \ar@3 [r] _-{\Ceil{c}^{0,n}} & - \capl{i} + \somme{r \geq 0}{} \stdaupf{i}{- \hil-1}{\hil - r  -1}{r} \ar@3 [r] _-{\Ceil{c}^1} & - \capl{i} + \somme{r \geq 0}{} \rulebf{i}{r}{\hil - r -1}  
	} \]

\vskip-20pt	

	\item For $\hil = 0$,
	\vskip-10pt
	\[
	\xymatrix@R=2.1em@C=2.5em{
		\bcritce{i} \ar@3 [r] ^-{\Cil} 
		\ar @3 [d] _-{\fleq} & \tfishulpbubup{i}{-1} 
		\ar @3 [r] ^-{\Ceil{c}^1} & \tfishul{i} \ar@3 [r] ^-{\Bil} & - \rulebhil{i}{\hspace{-2em} -1} \ar@3 [r] ^-{\Ceil{b}^1} & - \capl{i} \ar@3 [d] ^-{\fleq} \\
		\bcritce{i} \ar@3 [rrrr] _-{\Eil} & & & &  - \capl{i} } \]  
	\vskip-20pt
	\item For $\hil > 0$,
	\[ 
	\xymatrix@R=2.1em@C=1.5em{
		\bcritce{i} 
		\ar@3[r] ^-{\Cil}
		\ar@3 [d] _-{\fleq} & \somme{n=0}{\hil} \tfishulpbub{i}{n} \ar@3 [r] ^-{\Ceil{B'}} & \rulebhil{i}{\hil -1} \ar@3 [r] ^-{\Ceil{b}^1} & - \capl{i} \ar@3 [d] ^-{\fleq} \\
		\bcritce{i} 
		\ar@3 [rrr] _-{\Eil} & & & - \capl{i} } \]
		\vskip-3pt
	where the $3$-cell $\Ceil{B'}$ is defined as the $3$-cell  $ B'_{i,\lambda ,0} + \dots + B'_{i,\lambda,\hil}$, where each $3$-cell $B'_{i,\lambda,k}$ for $0 \leq k \leq \hil$ is defined in \ref{SS:FurtherCells}, and has for $2$-target $0$ if $n < \hil$ and $- \capl{i}$ if $n = \hil$.
\end{enumerate}
\vskip-30pt

\subsubsection{Critical branchings between types $E$ and $F$}
For any $i$ in $I$ and $\lambda$ in $X$, there are two types of critical branchings implying $3$-cells $\Eil$ and $\Fil$, depending on if the source $2$-cell of $\Eil$ is vertically composed below or above the source $2$-cell of $\Fil$. Following Section \ref{SSS:ClassificationOfBranchingsAF}, we denote by $(\Eil,\Fil)$ (resp. $(\Fil,\Eil)$) these two families of critical branchings. We will prove that for any $i$ and $\lambda$, the critical branchings $(\Eil,\Fil)$ are confluent modulo $E$, the other family of branchings would be proved confluent modulo $E$ similarly.

\begin{enumerate}[{\bf i)}]
	\item For $\hil < 0$,
	
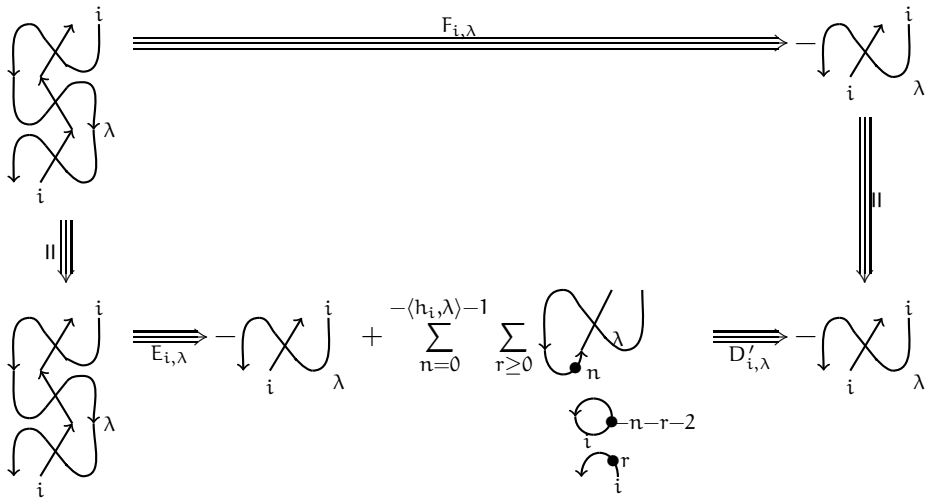
\[ \xymatrix@R=2.1em@C=2.5em{
		\bcritfe{i}{i} \ar@3 [rr] ^-{\Fil}
		\ar@3 [d] _-{\fleq} & & - \tleftcross{i}{i} \ar@3 [d] ^-{\fleq} \\
		\bcritfe{i}{i} \ar@3 [r] _-{\Eil} & - \tleftcross{i}{i} + \somme{n=0}{-\hil -1} \somme{r \geq 0}{} \tfishdlpbubfcapl{i}{n}{-n-r-2}{r} \ar@3 [r] _-{\Ceil{D'}} & - \tleftcross{i}{i} } \]
	where $\Ceil{D'}$ is the $3$-cell of $(\ER)_3^\ell$ defined as the composite of $3$-cells $D'_{i,\lambda,0} + \dots + D'_{i,\lambda, - \hil -1}$, where these cells are defined for $0 \leq k \leq - \hil -1$ in Appendix \ref{SS:FurtherCells}, and have all $0$ as $2$-target.
	\item For $\hil = 0$,
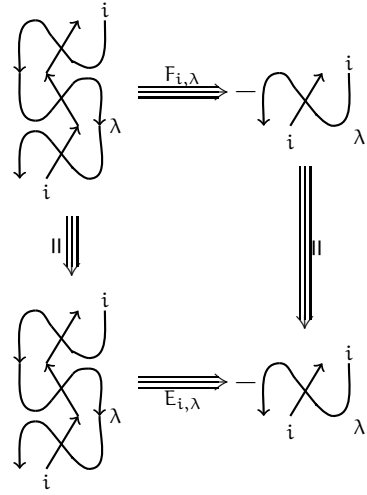
	\[ \xymatrix@R=2em@C=3em{
		\raisebox{7mm}{$\bcritfe{i}{i}$} \ar @3 [r] ^-{\Fil}
		\ar@3 [d] _-{\fleq} & - \tleftcross{i}{i} \ar@3 [d] ^-{\fleq} \\
		\raisebox{7mm}{$\bcritfe{i}{i}$} \ar @3 [r] _-{\Eil} & - \tleftcross{i}{i} } \]
	\item For $\hil > 0$,
	
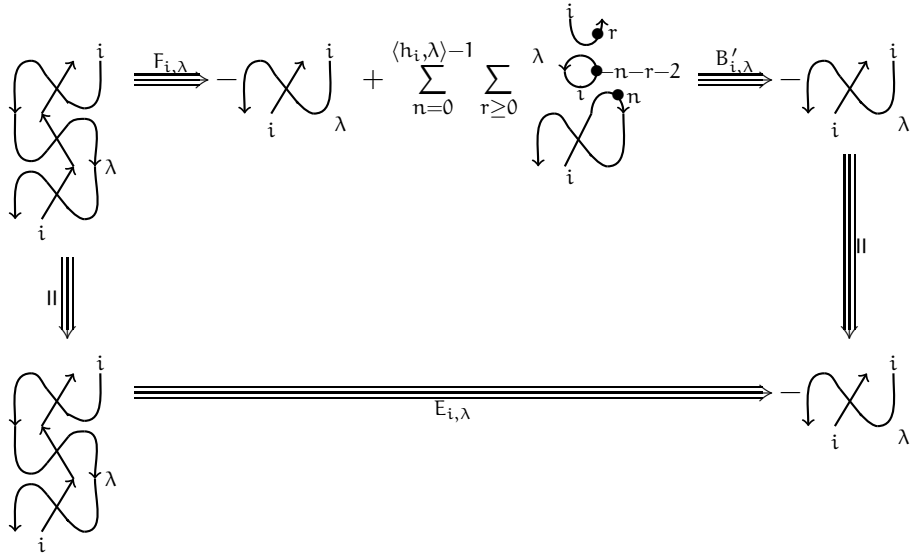
\[ \xymatrix@R=2.1em@C=2.5em{
		\bcritfe{i}{i} \ar@3 [r] ^-{\Fil}
		\ar@3 [d] _-{\fleq} & - \tleftcross{i}{i} + \somme{n=0}{\hil -1} \somme{r \geq 0}{} \tfishulpbubfcupr{i}{r}{-n-r-2}{n} \ar@3 [r] ^-{\Ceil{B'}} & - \tleftcross{i}{i} \ar@3 [d] ^-{\fleq} \\
		\bcritfe{i}{i} \ar@3 [rr] _-{\Eil} & & - \tleftcross{i}{i} } \]
	where $\Ceil{B'}$ is the $3$-cell of $(\ER)_3^\ell$ defined as the composite of $3$-cells $B'_{i,\lambda,0} + \dots + B'_{i,\lambda, \hil -1}$, where these cells are defined for $0 \leq k \leq \hil -1$ in Appendix \ref{SS:FurtherCells}, and have all $0$ as $2$-target.
\end{enumerate}

\end{document}